\documentclass[10pt, reqno]{amsart}
\usepackage{appendix}
\usepackage[a4paper, centering]{geometry}
\usepackage{xcolor}
\usepackage[utf8]{inputenc}
\usepackage[colorlinks=true,hyperindex,pagebackref,linktocpage=true]{hyperref}
\usepackage{cleveref}
\usepackage{thmtools}
\usepackage{mathtools}
\usepackage{amsfonts}
\usepackage{amssymb}
\usepackage{tikz-cd}
\usepackage[T1]{fontenc}
\usepackage{stix}
\usepackage{enumitem}
\usepackage{amsmath}

\hypersetup{
	colorlinks,
	linkcolor={violet!60!black},
	citecolor={cyan!60!black},
	urlcolor={orange!60!black}
}
\usepackage{stmaryrd}
\usepackage{mathrsfs}  
\usepackage{varwidth}
\theoremstyle{plain}
\newtheorem{theorem}{Theorem}[section]
\newtheorem{proposition}[theorem]{Proposition}

\newtheorem{lemma}[theorem]{Lemma}
\newtheorem{corollary}[theorem]{Corollary}

\theoremstyle{definition}
\newtheorem{definition}[theorem]{Definition}
\newtheorem{example}[theorem]{Example}
\newtheorem{remark}[theorem]{Remark}

\newcommand{\nc}{\newcommand}
\nc{\on}{\operatorname}
\nc{\sch}{{\on{sch}}}

\nc{\Q}{\mathbb{Q}}
\nc{\Z}{\mathbb{Z}}
\nc{\cl}{\mathrm{cl}}

\nc{\fraka}{{\mathfrak a}} \nc{\bba}{{\mathbf a}}
\nc{\frakb}{{\mathfrak b}}
\nc{\frakc}{{\mathfrak c}}
\nc{\frakd}{{\mathfrak d}}
\nc{\frake}{{\mathfrak e}}
\nc{\frakf}{{\mathfrak f}}
\nc{\frakg}{{\mathfrak g}}
\nc{\frakh}{{\mathfrak h}}
\nc{\fraki}{{\mathfrak i}}
\nc{\frakj}{{\mathfrak j}}
\nc{\frakk}{{\mathfrak k}}
\nc{\frakl}{{\mathfrak l}}
\nc{\frakm}{{\mathfrak m}}
\nc{\frakn}{{\mathfrak n}}
\nc{\frako}{{\mathfrak o}}
\nc{\frakp}{{\mathfrak p}}
\nc{\frakq}{{\mathfrak q}}
\nc{\frakr}{{\mathfrak r}}
\nc{\fraks}{{\mathfrak s}}
\nc{\frakt}{{\mathfrak t}}
\nc{\fraku}{{\mathfrak u}}
\nc{\frakv}{{\mathfrak v}}
\nc{\frakw}{{\mathfrak w}}
\nc{\frakx}{{\mathfrak x}}
\nc{\fraky}{{\mathfrak y}}
\nc{\frakz}{{\mathfrak z}}
\nc{\frakA}{{\mathfrak A}}
\nc{\frakB}{{\mathfrak B}}
\nc{\frakC}{{\mathfrak C}}
\nc{\frakD}{{\mathfrak D}}
\nc{\frakE}{{\mathfrak E}}
\nc{\frakF}{{\mathfrak F}}
\nc{\frakG}{{\mathfrak G}}
\nc{\frakH}{{\mathfrak H}}
\nc{\frakI}{{\mathfrak I}}
\nc{\frakJ}{{\mathfrak J}}
\nc{\frakK}{{\mathfrak K}}
\nc{\frakL}{{\mathfrak L}}
\nc{\frakM}{{\mathfrak M}}
\nc{\frakN}{{\mathfrak N}}
\nc{\frakO}{{\mathfrak O}}
\nc{\frakP}{{\mathfrak P}}
\nc{\frakQ}{{\mathfrak Q}}
\nc{\frakR}{{\mathfrak R}}
\nc{\frakS}{{\mathfrak S}}
\nc{\frakT}{{\mathfrak T}}
\nc{\frakU}{{\mathfrak U}}
\nc{\frakV}{{\mathfrak V}}
\nc{\frakW}{{\mathfrak W}}
\nc{\frakX}{{\mathfrak X}}
\nc{\frakY}{{\mathfrak Y}}
\nc{\frakZ}{{\mathfrak Z}}
\nc{\bbA}{{\mathbb A}}
\nc{\bbB}{{\mathbb B}}
\nc{\bbC}{{\mathbb C}}
\nc{\bbD}{{\mathbb D}}
\nc{\bbE}{{\mathbb E}}
\nc{\bbF}{{\mathbb F}} \nc{\bbf}{{\mathbf f}}
\nc{\bbG}{{\mathbb G}}
\nc{\bbH}{{\mathbb H}}
\nc{\bbI}{{\mathbb I}}
\nc{\bbJ}{{\mathbb J}}
\nc{\bbK}{{\mathbb K}}
\nc{\bbL}{{\mathbb L}}
\nc{\bbM}{{\mathbb M}}
\nc{\bbN}{{\mathbb N}}
\nc{\bbO}{{\mathbb O}}
\nc{\bbP}{{\mathbb P}}
\nc{\bbQ}{{\mathbb Q}}
\nc{\bbR}{{\mathbb R}}
\nc{\bbS}{{\mathbb S}}
\nc{\bbT}{{\mathbb T}}
\nc{\bbU}{{\mathbb U}}
\nc{\bbV}{{\mathbb V}}
\nc{\bbW}{{\mathbb W}}
\nc{\bbX}{{\mathbb X}}
\nc{\bbY}{{\mathbb Y}}
\nc{\bbZ}{{\mathbb Z}}
\nc{\calA}{{\mathcal A}}
\nc{\calB}{{\mathcal B}}
\nc{\calC}{{\mathcal C}}
\nc{\calD}{{\mathcal D}}
\nc{\calE}{{\mathcal E}}
\nc{\calF}{{\mathcal F}}
\nc{\calG}{{\mathcal G}}
\nc{\calH}{{\mathcal H}}
\nc{\calI}{{\mathcal I}}
\nc{\calJ}{{\mathcal J}}
\nc{\calK}{{\mathcal K}}
\nc{\calL}{{\mathcal L}}
\nc{\calM}{{\mathcal M}}
\nc{\calN}{{\mathcal N}}
\nc{\calO}{{\mathcal O}}
\nc{\calP}{{\mathcal P}}
\nc{\calQ}{{\mathcal Q}}
\nc{\calR}{{\mathcal R}}
\nc{\calS}{{\mathcal S}}
\nc{\calT}{{\mathcal T}}
\nc{\calU}{{\mathcal U}}
\nc{\calV}{{\mathcal V}}
\nc{\calW}{{\mathcal W}}
\nc{\calX}{{\mathcal X}}
\nc{\calY}{{\mathcal Y}}
\nc{\calZ}{{\mathcal Z}}
\newcommand{\solid}{{\scalebox{0.5}{$\square$}}}

\nc{\scrA}{{\mathscr A}}
\nc{\scrB}{{\mathscr B}}
\nc{\scrC}{{\mathscr C}}
\nc{\scrD}{{\mathscr D}}
\nc{\scrE}{{\mathscr E}}
\nc{\scrF}{{\mathscr F}}
\nc{\scrG}{{\mathscr G}}
\nc{\scrH}{{\mathscr H}}
\nc{\scrI}{{\mathscr J}}
\nc{\scrJ}{{\mathscr I}}
\nc{\scrK}{{\mathscr K}}
\nc{\scrL}{{\mathscr L}}
\nc{\scrM}{{\mathscr M}}
\nc{\scrN}{{\mathscr N}}
\nc{\scrO}{{\mathscr O}}
\nc{\scrP}{{\mathscr P}}
\nc{\scrQ}{{\mathscr Q}}
\nc{\scrR}{{\mathscr R}}
\nc{\D}{{\on{D}}}
\nc{\Div}{{\on{Div}}}
\nc{\Perv}{{\on{Perv}}}

\nc{\bnu}{{\bar{ \nu}}}
\nc{\olO}{\bar{\calO}}

\nc{\al}{{\alpha}} 
\nc{\be}{{\beta}}
\nc{\ga}{{\gamma}} \nc{\Ga}{{\Gamma}}
\nc{\hGa}{\hat{\Gamma}}
\nc{\ve}{{\varepsilon}} 
\nc{\la}{{\lambda}} \nc{\La}{{\Lambda}}
\nc{\om}{\omega} \nc{\Om}{\Omega} 
\nc{\sig}{{\sigma}} \nc{\Sig}{{\Sigma}}
\nc{\dR}{{\mathrm{dR}}}
\nc{\Perf}{\mathrm{Perf}^{\on{aff}}}
\nc{\Adic}{\mathrm{Adic}}
\nc{\Perff}{\widetilde{\on{Perf}}}
\nc{\PSch}{\mathrm{PSch}^{\on{aff}}}
\nc{\PSchf}{\widetilde{\on{PSch}}}
\nc{\Gm}{{\mathbb{G}_m}}
\nc{\colim}{{\on{colim}}}
\nc{\et}{\mathrm{\acute{e}t}}
\nc{\mer}{{\on{mer}}}
\nc{\frml}{{\on{frml}}}
\nc{\Red}{{\on{Red}}}
\nc{\An}{{\on{An}}}
\nc{\FF}{{\on{FF}}}
\nc{\sht}{{\on{sht}}}
\nc{\sss}{{\on{ss}}}
\nc{\Sets}{{\on{Sets}}}
\nc{\Grps}{{\on{Grps}}}
\nc{\Ani}{{\on{Ani}}}
\newcommand*\cocolon{%
        \nobreak
        \mskip6mu plus1mu
        \mathpunct{}%
        \nonscript
        \mkern-\thinmuskip
        {:}%
        \mskip2mu
        \relax
}

\nc{\vvb}{{\on{Vect}^{\calO^\sharp}_{\on{v}}}}
\nc{\Vect}{{\on{Vect}}}
\nc{\onv}{{\on{v}}}
\nc{\calvwp}{{\calV_\bbW[\frac{1}{\pi}]}}
\nc{\calvwps}{{\calV^\sch_\bbW[\frac{1}{\pi}]}}
\nc{\calvyp}{{\calV_\calY[\frac{1}{\pi}]}}

\nc{\perfc}{\frakP\mathrm{erf}}
\nc{\infcat}{\mathrm{Cat}_{\infty}}

\nc{\DM}{{\Sht^\sch_\bbW}}
\nc{\DMan}{{\Sht_\bbW}}
\nc{\DMG}{{\Sht^\sch_{\calG}}}
\nc{\IC}{{\mathfrak{B}}}
\nc{\ICG}{{\frakB(G)}}
\nc{\SHT}{{\Sht_\calY}}
\nc{\VEC}{{\on{Bun}}_{\on{FF}}^{\on{mer}}}
\nc{\VECT}{{\on{Bun}}_{\on{FF}}}
\nc{\Catex}{{\on{Cat}^{\otimes, \on{ex}}_{1}}}
\nc{\Cat}{{\on{Cat}^{\otimes}_{1}}}
\nc{\CatexOE}{{\on{Cat}^{\otimes, \on{ex}}_{1,O_E}}}
\nc{\CatOE}{{\on{Cat}^{\otimes}_{1,O_E}}}
\nc{\CatexE}{{\on{Cat}^{\otimes, \on{ex}}_{1,E}}}
\nc{\CatE}{{\on{Cat}^{\otimes}_{1,E}}}
\nc{\LinCat}{{\on{LinCat}^{\otimes, \on{ex}}_{O_E}}}
\nc{\LinCatE}{{\on{LinCat}^{\otimes, \on{ex}}_{E}}}
\nc{\Fil}{\calF{il}}
\nc{\DmG}{{\Sht_{\bbW,\calG}}}
\nc{\Sht}{\on{Sht}}
\nc{\Isoc}{\on{Isoc}}

\nc{\BC}{\calB\calC}
\nc{\PFS}{\bbP_{{\on{FS}}}}
\def\preceqdot{\mathrel{\preceq\kern-.5em\raise.22ex\hbox{$\cdot$}}}

\makeatletter
\DeclareFontEncoding{LS1}{}{}
\DeclareFontSubstitution{LS1}{stix}{m}{n}
\DeclareMathAlphabet{\rhomalpha}{LS1}{stixscr}{m}{n}
\makeatother

\nc{\Spa}{\on{{Spa}}}
\nc{\Spd}{\on{{Spd}}}
\nc{\tnb}{\psi_{\rm tame}}
\nc{\oM}{\overline{{M}}}
\nc{\op}{{\on{op}}}
\nc{\ad}{{\on{ad}}}
\nc{\alg}{{\on{alg}}}
\nc{\Ad}{{\on{Ad}}}
\nc{\Adm}{{\on{Adm}}} \nc{\aff}{{\on{af}}}
\nc{\Aut}{{\on{Aut}}}
\nc{\Bun}{{\on{Bun}}}
\nc{\cha}{{\on{char}}}
\nc{\der}{{\on{der}}}
\nc{\Der}{{\on{Der}}}
\nc{\diag}{{\on{diag}}}
\nc{\End}{{\on{End}}}
\nc{\Fl}{{\calF\!\ell}}
\nc{\Tr}{{\on{Transp}}}
\nc{\TR}{{\calT\!\calR}}
\nc{\Gal}{{\on{Gal}}}
\nc{\Gr}{{\on{Gr}}}
\nc{\Hk}{{\on{Hk}}}
\nc{\rH}{{\on{H}}}
\nc{\Hom}{{\on{Hom}}}
\nc{\id}{{\on{id}}}
\nc{\Id}{{\on{Id}}}
\nc{\ind}{{\on{ind}}}
\nc{\Ind}{{\on{Ind}}}
\nc{\Lie}{{\on{Lie}}}
\nc{\Pic}{{\on{Pic}}}
\nc{\pr}{{\on{pr}}}
\nc{\Res}{{\on{Res}}}
\nc{\res}{{\on{res}}} \nc{\Sat}{{\on{Sat}}}
\nc{\spc}{{\on{sc}}}
\nc{\drv}{{\on{der}}}
\nc{\sgn}{{\on{sgn}}}
\nc{\Spec}{{\on{Spec}}\,}
\nc{\Spf}{\on{Spf}} 
\nc{\Sph}{\on{Sph}}
\nc{\St}{{\on{St}}}
\nc{\tr}{{\on{tr}}}
\nc{\Mod}{{\mathrm{-Mod}}}
\nc{\Hilb}{{\on{Hilb}}} 
\nc{\Ext}{{\on{Ext}}} 
\nc{\vs}{{\on{Vec}}}
\nc{\ev}{{\on{ev}}}
\nc{\nO}{{\breve{\calO}}}
\nc{\tS}{{\tilde{S}}}
\nc{\spe}{{\on{sp}}}
\nc{\loc}{{\on{loc}}}

\nc{\dimt}{{\on{dim.trg}}}

\renewcommand{\right}{\rightarrow}

\nc{\co}{\colon}
\nc{\dia}{{\lozenge}}

\nc{\nscrR}{{\mathscr{R}^{\on{nr}}}}

\nc{\GL}{{\on{GL}}}

\nc{\Gl}{\on{Gl}} 
\nc{\GSp}{{\on{GSp}}}
\nc{\gl}{{\frakg\frakl}}
\nc{\SL}{{\on{SL}}} 
\nc{\SU}{{\on{SU}}} 
\nc{\SO}{{\on{SO}}}
\nc{\PGL}{{\on{PGL}}}

\nc{\Conv}{{\on{Conv}}}
\nc{\Rep}{{\on{Rep}}}
\nc{\Dom}{{\on{Dom}}}
\nc{\red}{{\on{red}}}
\nc{\sh}{{\on{sh}}}
\nc{\disc}{{\on{disc}}}
\nc{\pre}{{\on{pre}}}
\nc{\act}{{\on{act}}}
\nc{\nr}{{\on{nr}}}
\nc{\ctf}{{\on{ctf}}}

\nc{\str}{{\on{-}}} 
\nc{\os}{{\bar{s}}}
\nc{\oeta}{{\bar{\eta}}}

\nc{\hookto}{\hookrightarrow}
\nc{\longto}{\longrightarrow}
\nc{\leftto}{\leftarrow}
\nc{\onto}{\twoheadrightarrow}
\nc{\lonto}{\twoheadleftarrow}

\nc{\pot}[1]{ [\hspace{-0,5mm}[ {#1} ]\hspace{-0,5mm}] }
\nc{\rpot}[1]{ (\hspace{-0,7mm}( {#1} )\hspace{-0,7mm}) }
\nc{\smallpot}{{ <\hspace{-1,0mm}<}}

\numberwithin{equation}{section}

\newcommand{\rar}{\rightarrow}

\begin{document}
	
\title{Meromorphic vector bundles on the Fargues--Fontaine curve}
	
\author[I. Gleason, A. Ivanov, F. Zillinger]{Ian Gleason, Alexander B. Ivanov, Felix Zillinger}

	\address{Mathematisches Institut der Universit\"at Bonn, Endenicher Allee 60, Bonn, Germany}
	\email{igleason@uni-bonn.de}
	
	\address{Fakult\"at f\"ur Mathematik, Ruhr-Universit\"at Bochum, Universit\"atsstrasse 150, D-44780 Bochum, Germany}
	\email{a.ivanov@rub.de}

	\address{Fakult\"at f\"ur Mathematik, Ruhr-Universit\"at Bochum, Universit\"atsstrasse 150, D-44780 Bochum, Germany}
	\email{felix.zillinger@rub.de}
	
	\begin{abstract}
	We introduce and study the stack of \textit{meromorphic} $G$-bundles on the Fargues--Fontaine curve.
	This object defines a correspondence between the Kottwitz stack $\ICG$ and $\Bun_G$. 
	We expect it to play a crucial role in defining and studying an analytification functor that compares the scheme-theoretic and analytic versions of the geometric local Langlands categories. 
	Our first main result is the identification of the generic Newton strata of $\Bun_G^\mer$ with the Fargues--Scholze charts $\calM$. 
	Our second main result is a generalization of Fargues' theorem in families.  
	We call this the \textit{meromorphic comparison theorem}. 
	We expect it to play a key role in proving that the analytification functor is fully-faithful.
	Along the way, we give new proofs to what we call the \textit{topological and scheme-theoretic comparison theorems}.
These say that the topologies of $\Bun_G$ and $\ICG$ are reversed and that the two stacks take the same values when evaluated on schemes. 
	\end{abstract}

	\maketitle
	\tableofcontents
	
	\section{Introduction}
	Let $p$ be a prime number, and let $E$ be a non-Archimedean local field with residue field of characteristic $p$. 
	Let $\ell$ be a prime with $\ell\neq p$, and $\Lambda=\overline{\bbQ}_\ell$. 
	Let $G$ be a connected reductive group over $E$.
	Let $W_E$ be the Weil group and let $^LG=\hat{G}\rtimes W_E$ be the $L$-group. 
	For this introduction, we will further assume that $G$ is quasi-split, but we drop this assumption in the body of the text.
	\subsection{Motivation and context}
Let $\Pi_G$ be the set of isomorphism classes of smooth irreducible representations of the locally profinite group $G(E)$ with values in $\Lambda$, and let $\Phi_G$ be the set of $\hat{G}$-conjugacy classes of $L$-parameters.
The basic form of the local Langlands correspondence gives a map 
\[\on{LLC}_G\colon \Pi_G\to \Phi_G\]
satisfying various desiderata \cite[Conjecture A]{Kal16}, \cite{SZ18}. 
For $\on{GL}_n$, the map $\on{LLC}_{\on{GL}_n}$ is bijective \cite{HT01,Henniart_LLC_GLn,MR1228127}, but this does not hold more generally. 
Nevertheless, $\on{LLC}_{G}$ has finite fibers that are called $L$-packets, which can be understood in terms of the centralizer $S_\phi$ of the parameter $\phi\in \Phi_G$ inside $\hat{G}(\Lambda)$. 
More precisely, for split groups after fixing a Whittaker datum $\frakw$ one can put the elements of an $L$-packet in bijection with the set of isomorphisms classes of irreducible representations of $S_\phi$ whose restriction to the center $Z(\hat{G})$ is trivial, at least when $\phi$ is a discrete parameter (see \cite[Conjecture B]{Kal16} for a version valid for quasi-split groups and tempered parameters).
When $G$ is not quasi-split, Whittaker data do not exist.
Vogan realized that to work with general $G$, it is advantageous to consider its quasi-split inner form $G^\ast$ and parametrize simultaneously the representations of all the pure inner twists of $G^\ast$ \cite{ABV92,Vog93}. 

Motivated by the study of special fibers of Shimura varieties, Kottwitz introduced the set $B(G)$ of isocrystals with $G$-structure \cite{Kottwitz, KottwitzII}. 
The set of basic elements $B(G)_{\on{bas}}\subseteq B(G)$ gives rise to the so-called extended pure inner forms $G_b$ of $G$.
Kottwitz formulated a refined version of the local Langlands correspondence for non-Archimedean local fields using these inner forms that arise from $B(G)_{\on{bas}}$ \cite[Conjecture F]{Kal16}, \cite{SZ18}.

The set $B(G)$ with its natural partial ordering can be used to describe the underlying topological space of two geometric objects.
One object is of analytic nature, $\Bun_G$ (the stack of $G$-bundles on the Fargues--Fontaine curve), and a second object is of scheme-theoretic nature, $\ICG$ (the Kottwitz stack parametrizing isocrystals with $G$-structure).
For every element $b\in B(G)$ one can define locally closed strata $i_b\colon \ICG_b\to \ICG$ and $j_b\colon \Bun_G^b\to \Bun_G$. 
Both $\ICG_b$ and $\Bun_G^b$ are classifying stacks for a group, and sheaves on these classifying stacks can be described in terms of smooth representations of $G_b(E)$, where $G_b$ is an inner form of a Levi subgroup of $G$.  
This leads to the hope that the refined local Langlands correspondence of Kottwitz has a categorical refinement that one can access by studying the geometry of the stacks $\Bun_G$ and/or $\ICG$.

	Recent breakthroughs in $p$-adic and perfect geometry \cite{SW20,FS24,zhu_affine_grassmannians_and_the_geometric_satake_in_mixed_characteristic,XiaoZhu,bhatt_scholze_projectivity_of_the_witt_vector_affine_grassmannian,Zhu20} together with the introduction and study of the stack of $L$-parameters \cite{DHKM20,Zhu20,FS24} have led experts to formulate precise conjectures that capture this hope. 
	These efforts promote the refined local Langlands correspondence mentioned above to a categorical statement \cite{FS24,Zhu20,hellmann2021derived,benzvi2022coherent} in a precise way.

	There is widespread agreement on what to consider on the Galois side, namely a version of the derived category of coherent sheaves $\calD^{b,\on{qc}}_{\on{coh}}(\calX_{\hat{G},\Lambda})$ of the stack $\calX_{\hat{G},\Lambda}$ parametrizing $L$-parameters over $\Lambda$ (see \cite[Conjecture I.10.2]{FS24}, \cite{GA15}). 
	On the automorphic side there are at least two reasonable constructions of the local Langlands category. 
	The essential difference between them arises from the fact that $B(G)$ has two geometric incarnations.  

	Let $G$ be quasi-split and let $W_\frakw$ be the Whittaker representation associated to $\frakw$.
	On the analytic side, Fargues--Scholze construct the category of lisse sheaves $D_{\on{lis}}(\Bun_G,\Lambda)$ \cite[$\mathsection$ VII.7]{FS24} and prove it is compactly generated. In what follows, we use the superscript $(-)^{\omega}$ to denote the full subcategory of compact objects.
	Moreover, they endow this category with the so-called spectral action by the category of perfect complexes $\on{Perf}(\calX_{\hat{G},\Lambda})$.
	They conjecture that there is a unique $\on{Perf}(\calX_{\hat{G},\Lambda})$-linear equivalence of $\infty$-categories
	\[\bbL^{\on{an}}_G\colon \calD_{\on{lis}}(\Bun_G,\Lambda)^\omega \stackrel{\simeq}{\longrightarrow} \calD^{b,\on{qc}}_{\on{coh}}(\calX_{\hat{G},\Lambda})\]
	which, by taking ind-completions, induces an equivalence $\calD_{\on{lis}}(\Bun_G,\Lambda) \stackrel{\simeq}{\longrightarrow} {\rm Ind}(\calD^{b,\on{qc}}_{\on{coh}}(\calX_{\hat{G},\Lambda}))$ that sends the analytic Whittaker sheaf $\calW^{\on{an}}_\frakw=j_{1,!}W_\frakw$ to the structure sheaf $\calO_{\calX_{\hat{G},\Lambda}}$.

	On the scheme-theoretic side, Xiao--Zhu consider the moduli stack of local shtukas $\on{Sht}^{\on{loc}}_k$ in the context of characteristic $p$ perfect geometry. 
	They attach their own candidates for the local Langlands category, namely they construct a triangulated category of cohomological correspondences $\on{P}^{\on{Corr}}(\on{Sht}^{\on{loc}}_k)$, cf.\ \cite[$\mathsection$ 5.4]{XiaoZhu} and \cite{Zhu20}.  
	This approach is pushed further in the work of Zhu \cite{zhu2025tame}, where he constructs an $\infty$-category ${\on{Shv}}(\ICG,\Lambda)$ whose homotopy category agrees with $\on{P}^{\on{Corr}}(\on{Sht}^{\on{loc}}_k)$. 
	Zhu conjectures that there is an equivalence 
	\[\bbL^{\on{sch}}_G\colon\on{Shv}(\ICG,\Lambda) \stackrel{\simeq}{\longrightarrow} \on{Ind}(\calD^{b,\on{qc}}_{\on{coh}}(\calX_{\hat{G},\Lambda}))\,, \]
	sending $\calO_{\calX_{\hat{G},\Lambda}}$ to the scheme-theoretic Whittaker sheaf $\calW^{\on{sch}}_\frakw=i_{1,*}W_\frakw$, see \cite[Conjecture 4.6.4]{Zhu20}.
	In  his work, Zhu already proved the unipotent and tame parts of this equivalence in case $G$ is unramified, cf.\ \cite[Theorem 4.6.11]{Zhu20}. 
	Let us clarify.
	When $\Lambda$ is of characteristic $0$, the stack of $L$-parameters has an open and closed substack $\calX^{\on{unip}}_{\hat{G},\Lambda}\subseteq \calX_{\hat{G},\Lambda}$ defining a full subcategory 
	\[\on{Ind}(\calD^{b,\on{qc}}_{\on{coh}}(\calX^{\on{unip}}_{\hat{G},\Lambda}))\subseteq   \on{Ind}(\calD^{b,\on{qc}}_{\on{coh}}(\calX_{\hat{G},\Lambda}))\,.\]
	One can also define a full subcategory $\on{Shv}^{\on{unip}}(\ICG,\Lambda)\subseteq \on{Shv}(\ICG,\Lambda)$ defined by the property that for all $b\in B(G)$, the restriction to $\ICG_b$ is given by a complex of $G_b$-representations that are unipotent in the sense of Lusztig \cite{Lus95}. 
	Using Bezrukavnikov's equivalence \cite{Bez16}, Zhu proves that there is an equivalence of $\infty$-categories
	\[\bbL^{\on{sch}}_G\colon\on{Shv}^{\on{unip}}(\ICG,\Lambda) \stackrel{\simeq}{\longrightarrow} \on{Ind}(\calD^{b,\on{qc}}_{\on{coh}}(\calX^{\on{unip}}_{\hat{G},\Lambda}))\,. \]

	It is natural to expect that there exists an equivalence 
	\[\Psi\colon \on{Shv}(\ICG,\Lambda)\to \calD_{\on{lis}}(\Bun_G,\Lambda)\,,\] 
	satisfying $\Psi(\calW^{\on{sch}}_\frakw)=\calW^{\on{an}}_\frakw$.
	Indeed, the two local Langlands categories are conjectured to be equivalent to $\on{Ind}(\calD^{b,\on{qc}}_{\on{coh}}(\calX_{\hat{G},\Lambda}))$, and if the two conjectures are true one can simply define $\Psi=\bbL^{\on{an},-1}_{G}\circ \bbL^{\on{sch}}_{G}$.

	A reasonable question the reader can ask is: why do we need two local Langlands categories?
	We believe that it is profitable to construct $\Psi$ directly in order to better understand $\bbL^{\on{sch}}_G$ and $\bbL^{\on{an}}_G$.
	At a technical level, a direct construction of $\Psi$ allows one to transfer Zhu's results on tame categorical local Langlands correspondence to the Fargues--Scholze setup and conversely, endow $\on{Shv}(\ICG,\Lambda)$ with a spectral action.
	It would also allow us to formulate rigorously the eigensheaf property for the Deligne--Lusztig sheaves considered in \cite[Conjecture 9.6]{CI23}.
	More philosophically, the scheme-theoretic perspective and the analytic perspective understand different phenomena. 
	For example, the scheme-theoretic perspective cannot witness the spectral action because "the paw" is fixed. 
	On the other hand, $\on{Shv}(\ICG,\Lambda)$ is directly related to Bezrukavnikov's equivalence and its Frobenius-twisted categorical trace \cite[$\mathsection 3$]{zhu2018geometric} since, in contrast with $\Bun_G$, both $\ICG$ and the Hecke stack are constructed in terms of Witt vector loop groups.  

	At the heart of the equivalence $\Psi$ there should be a geometric explanation. 
	Namely, that the stacks $\IC(G)$ and $\Bun_G$ are incarnations of the same geometric object.
	In this paper, we reveal these geometric relations which we formulate in terms of three comparison theorems (see $\mathsection \ref{section-3-compar}$). 

	One of the achievements of this article is the construction of a third incarnation $\Bun_G^\mer$ that mediates between $\IC(G)$ and $\Bun_G$. Roughly speaking, $\Bun_G^\mer$ is given by the same moduli problem as $\Bun_G$, but we require a meromorphicity condition on the action of Frobenius (see \Cref{meromorphic-vector-bundles}, \Cref{defi-groupoids-of-interest}). 
	This object defines a correspondence 
	\begin{center}
		\begin{tikzcd}
			\Bun_G^\mer \ar{r}{\sigma}\ar{d}{\gamma} & \Bun_G  \\		
		\ICG^\lozenge	
		\end{tikzcd}
	\end{center}
	in the category of $v$-stacks on perfectoid spaces in characteristic $p$, where $(-)^{\lozenge}$ denotes a suitable analytification of $\ICG$, cf.\ \Cref{analytifications}.
	We call the map $\gamma$ the \textit{generic polygon map} and $\sigma$ the \textit{special polygon map} inspired by \cite[Definition 7.4.1]{kedlaya_liu_relative_p_adic_hodge_theory_foundations}. 
	Morally, $\Psi$ should be given by $\sigma_!\circ \gamma^*\circ c^*$, where
	\[c^*\colon \on{Shv}(\ICG,\Lambda)\to \calD(\ICG^\lozenge,\Lambda)\]
	is an analytification functor \cite[$\mathsection 27$]{Sch17}, but we will not pursue the construction of this functor here. 
	\subsection{Main results}
	For $b\in B(G)$ we let $\ICG_b\subseteq \ICG$ denote the locally closed substack determined by $b$. 
	Then $\ICG_b^\lozenge\subseteq \ICG^\lozenge$ is also a locally closed substack and we have an identification
	\[\ICG_b^\lozenge=[\ast\!/\underline{G_b(E)}]\,,\]
	where $\underline{G_b(E)}$ denotes the constant group v-sheaf associated to the locally profinite group $G_b(E)$.
	Recall the moduli stack $\calM$ of Fargues--Scholze \cite[Definition V.3.2]{FS24} that is used to define the smooth charts of $\Bun_G$.
	It comes endowed with a map 
	\[q\colon \calM\to \coprod_{b\in B(G)} [\ast\!/\underline{G_b(E)}]\cong \coprod_{b\in B(G)} \ICG_b^\lozenge\,.\]
	The following \Cref{intro:FScharts} is a relative and Tannakian version of Kedlaya's work on the slope filtration \cite[Section 5.4]{Kedlaya_slope_revisited}, and our first main result.

\begin{theorem}[\Cref{maintheorem}]
	\label{intro:FScharts}
	We have a commutative diagram with Cartesian square
	\begin{center}
	\begin{tikzcd}
		\calM \arrow{r} \arrow{d}{q} \arrow[bend left]{rr}{\pi} & \Bun_G^\mer \arrow{d}{\gamma} \ar{r}{\sigma} & \Bun_G\,.  \\
	 \coprod_{b\in B(G)} \ICG_b^\lozenge \arrow{r} & \ICG^\lozenge
	\end{tikzcd}
	\end{center}
\end{theorem}

In other words, the restriction of $\sigma\colon\Bun_G^{\rm mer} \rar \Bun_G$ to $\gamma^{-1}([\ast\!/\underline{G_b(E)}])$ coincides with the Fargues--Scholze chart $\pi_b \colon \calM_b \right \Bun_G$ \cite[V.3]{FS24}.

\begin{remark}
	The proof of \Cref{intro:FScharts} is carried out by discussing the case $G=\GL_n$ (i.e.\ the case of vector bundles) in length and appealing to Tannakian formalism to prove the statement for general groups $G$. 
To apply Tannakian formalism one has to take subtle care of the exact structure. 
We do this by justifying that a sequence of meromorphic vector bundles is exact if and only if it is exact at every geometric point (see \Cref{DM1p-is-well-bheaved-subcategory-old}). 
\end{remark}

\begin{remark}
	Z. Wu also proved a version of \Cref{maintheorem} independently (see \Cref{rmrk-Zw}).
\end{remark}
\Cref{maintheorem} should be closely related to the essential surjectivity of $\Psi$.
Our second main result, which we now explain, should be related to fully-faithfulness. 
Recall the analytification functor $X\mapsto X^\dagger$, obtained from sheafifying the formula
\[(R,R^+)\mapsto X(\Spec R^\circ)\,,\]
see \Cref{analytifications}. 
For any small v-stack $X$ we have fully-faithful maps
\[\calD_{\acute{e}t}(X,\bbF_\ell)\xrightarrow{c_X^*} \calD_{\acute{e}t}(X^\lozenge, \bbF_\ell) \xrightarrow{b_X^*}  \calD_{\acute{e}t}(X^\dagger, \bbF_\ell)\,.\]
When $X$ is an affine scheme, fully-faithfulness of these functors is shown in \cite[Lemma 4.1]{GL22}.
The passage to small v-stacks follows formally from this case.

\begin{theorem}[\Cref{meromorphic-comparison-theorem}]
	\label{intro:meromorphiccomparison}	
	We have the following identification of small v-stacks
	\[\Bun_G^\mer \cong \ICG^\dagger\] 	
	and an identification of maps $b^*_{\ICG}=\gamma^*$.
A similar statement holds for the stack of local $\calG$-shtukas for $\calG$ a parahoric model of $G$. 
\end{theorem}
\begin{remark}
	This statement can be regarded as a version of Fargues' theorem in families (see \Cref{remark-Faruges-theorem}).
\end{remark}

\begin{corollary}
	We have a fully-faithful comparison map
	\[	\gamma^*\circ c^*\colon \calD_{\acute{e}t}(\ICG,\bbF_\ell)\to \calD_{\acute{e}t}(\Bun_G^\mer,\bbF_\ell)\,.\]
\end{corollary}

\Cref{intro:meromorphiccomparison} provides an approach to prove that $\Psi$ is fully-faithful.
Indeed, it suffices to prove that $\sigma_!$ is fully-faithful when restricted to those objects in the essential image of $\gamma^*\circ c^*$. 
The advantage being that the geometry of $\Bun_G^\mer$ is much closer to that of $\Bun_G$ than that of $\ICG$. 
 
\begin{remark}
	We warn the reader that it is unknown to the authors whether the category $\calD_{\acute{e}t}(\ICG,\bbF_\ell)$ is equivalent to $\on{Shv}(\ICG,\bbF_\ell)$ or not.
	The definition of $\on{Shv}(\ICG,\bbF_\ell)$ is through a Kan extension procedure with respect to $!$-pullback, and the functor $\on{Shv}(-,\bbF_\ell)$ neither satisfies v-descent, nor is it left-complete when evaluated on a general stack $X$, unlike $\calD_{\acute{e}t}(-,\bbF_\ell)$. 
	On the other hand, the definition of $\calD_{\acute{e}t}(\ICG,\bbF_\ell)$ is through v-descent along $*$-pullback, and by construction the value of $\calD_{\acute{e}t}(X,\Lambda)$ is a left-complete category for all stacks $X$. 
For these reasons, it is not even clear from the definitions how to construct a map from one category to the other.
More crucially, it is not clear to the authors whether one should expect $\calD_{\acute{e}t}(\ICG,\bbF_\ell)$ to be compactly generated or not.
\end{remark}

One of the main ingredient in our proof of \Cref{intro:meromorphiccomparison} is the following statement. 

\begin{proposition}[\Cref{lemma-extending-at-infty}]
	\label{intro:keylemmaextendinfty}
	Every $G$-bundle on the Fargues--Fontaine curve extends v-locally at $\infty$.
\end{proposition}

This together with the main theorem of \cite{Ans18} (which is itself an ingredient in the proof of \Cref{intro:keylemmaextendinfty}) has as a consequence the classification of \Cref{intro:classification} below. 
We fix some notation. 
Let $S=\Spa (R,R^+)$ be a product of points with $R^\circ=\prod_{i\in I}\! O_{C_i}$ and a family of pseudo-uniformizers $\varpi_\infty=(\varpi_i)_{i\in I}$ such that $\varpi_\infty$ defines the topology on $R^\circ$. 
We moreover fix an untilt $S^\sharp$ given by a non-zero divisor $\xi_\infty=(\xi_i)_{i\in I}$. 
This induces for all $i\in I$ an untilt $C^\sharp_i$.

\begin{corollary}
	\label{intro:classification}
	Assume that $S$ is a product of points. Then the following categories are equivalent:
	\begin{enumerate}
		\item The category of local shtukas over $S$ with paw at $S^\sharp$. 
		\item The category of Breuil--Kisin--Fargues modules over $\bbA_{\on{inf}}(R^{\circ,\sharp})$.
		\item The category of $I$-indexed families $\{(M_i,\Phi_i)\}_{i\in I}$ of Breuil--Kisin--Fargues modules over $\bbA_{\on{inf}}(O_{C^\sharp_i})$ with uniformly bounded poles and zeroes at $\xi_\infty$.
	\end{enumerate}
	Furthermore, a similar statement with $\calG$-structure holds.
\end{corollary}

\begin{remark}
	That vector bundles extend v-locally at $\infty$ was also shown independently by Zhang in her proof of Scholze's fiber product conjecture \cite[Proof of Proposition 8.14]{Zha23}. 
	The previous version of this article discussed a proof of \Cref{intro:keylemmaextendinfty} that was substantially more complicated and had restrictions on the characteristic of $E$. 
	It used the $\calM_b$-charts of Fargues--Scholze to uniformize $\Bun_G$ and extend at $\infty$. 
	The current approach uses the Beauville--Laszlo uniformization of $\Bun_G$ to extend at $\infty$.
	The argument provided in the current version of the article is quite close to the one of Zhang. 
	We are very grateful to her for several conversations related to this.
\end{remark}

\subsection{New proofs of two established results}
As a consequence of our considerations we found new approaches to previously proven theorems relating the geometry of $\ICG$ and $\Bun_G$.
\subsubsection{The scheme-theoretic comparison}
Recall the reduction functor introduced by the first author in \cite[\S 3]{Gle22}. 
Roughly, it is an analogue in the context of $v$-sheaves of the functor that takes a formal scheme to its reduced special fiber; the reduction $X^\red$ of a $v$-stack $X$ on perfectoid spaces over $\bbF_p$ is a $v$-stack on perfect schemes over $\bbF_p$ (see also $\mathsection 2$).
The following theorem is a reformulation of a result of Pappas--Rapoport \cite[Theorem 2.3.8]{PR21}, which in turn genealizes a result of Ansch\"utz \cite[Theorem 1.1]{Anschuetz_absFF}. We take a new approach.
\begin{theorem}[\Cref{comparison-reduction}]
	\label{intro:schematiccomparison}
	We have an identification of scheme-theoretic v-stacks
	\[(\Bun_G)^\red \cong \ICG\,.\]
	If $\calG$ is a parahoric model of $G$, then a similar statement holds for the stack of $\calG$-shtukas.
\end{theorem}
\begin{remark}
	We regard \Cref{intro:schematiccomparison} as a classicality statement.
	Ansch\"utz proves the equivalence of categories $(\Bun_G)^\red(k) \cong \ICG(k)$ for an algebraically closed field $k/\bbF_p$ using the classification of vector bundles over the Fargues--Fontaine curve \cite[Theorem 3.4]{Anschuetz_absFF}.
	Pappas--Rapoport \cite[Theorem 2.3.8]{PR21} prove this more generally using the result of Ansch\"utz and in particular rely on the $\varphi$-structure (i.e.\ descent datum along the map $Y_{(0,\infty)}\to Y_{(0,\infty)}/\varphi$).
	We give a uniform proof and work directly with the category of v-vector bundles over $Y_{(0,\infty)}$, showing that classicality is unrelated to the $\varphi$-structure. 
	G\"uthge also realized this independently \cite{Gut23} (see \Cref{calV-is-a-vsheaf}). 
\end{remark}
\subsubsection{The topological comparison}
Recall that $B(G)$ comes endowed with a topology induced by its partial order.
We can also consider $B(G)^{\on{op}}$ endowed with the topology induced by the opposite partial order. 
For either a scheme-theoretic or an analytic $v$-stack $X$, let $|X|$ denote its underlying topological space.
Viehmann \cite[Theorem 1.1]{Viehamann_Newton_BdR} proves that $|\Bun_G|^{\on{op}}\cong B(G)$.
Rapoport--Richartz \cite{RR_96} and He \cite[Theorem 2.12]{He_hecke_padic} prove $|\ICG|\cong B(G)$.

We give an alternative proof of the following theorem.
\begin{theorem}[\Cref{topologicalcomparisonmainthext}]
\label{intro:topologicalcomparison}
	The natural maps are homeomorphisms
	\[|\Bun_G|^{\on{op}}\cong |\ICG|\cong |\ICG^\lozenge|\,.\]
\end{theorem}



	\subsection{Acknowledgements}
	We would like to thank Peter Scholze for many illuminating conversations. 
	We thank Alexander Bertoloni-Meli and Linus Hamann for explaining several things about the refined local Langlands correspondence. 
	We thank Ko Aoki for explaining to us several things related to the $\infty$-categorical background necessary to carry this project. 
	The third author wants to thank the first author for the excellent supervision of his master thesis.
	He moreover wants to thank the first and second author for their constant help and encouragement and for inviting him to join this project.
	The third author also wants to thank Johannes Ansch\"utz for several helpful conversations about extending \Cref{theorem-evaluating-opensubs} to perfect complexes.
	The first and second author would like to thank the third author whose master thesis work reshaped their perspective on how this article should be written.
	We also thank David Hansen, Georgios Pappas, Jared Weinstein for helpful comments on a preliminary version of the article.
	We also thank Johannes Ansch\"utz, Anton G\"uthge, Ben Heuer, Jo\~ao Louren\c{c}o, Sug Woo Shin, Eva Viehmann, Pol van Hoften, Mingjia Zhang, and Xinwen Zhu for stimulating conversations. 
	Finally, we are very grateful to the anonymous referee for carefully reading the article and providing very helpful feedback.
	
	This paper was written during stays at Max-Planck-Institut für Mathematik and Universität Bonn, we are thankful for the hospitality of these institutions. The first author has received funding by DFG via the Leibniz-Preis of Peter Scholze.

    \section{Notation, terminology and generalities}\label{sec:notation}
   We fix the following notation throughout the text.
    Given a Huber pair $(A,A^+)$, we let $\Spa(A,A^+)$ denote the preadic space of \cite[\S 3.4]{SW20}.
    We will mostly restrict to stably uniform analytic Huber pairs $(A,A^+)$, so that $\Spa(A,A^+)$ is an adic space.
    Whenever $A^+=A^\circ$ we simply write $\Spa A$.
    Analogously, given a Huber pair over $\bbZ_p$, we let $\Spd(A,A^+)$ denote the small v-sheaf of \cite[Lemma 15.1]{Sch17}. 
    Whenever $A^+=A^\circ$ we simply write $\Spd A$.
    We let $\Adic$ denote the category of uniform analytic adic spaces over $\bbZ_p$ \cite[Definitions 3.2.1 and 5.3.1]{SW20}.
    For an adic space $X$, by a geometric point of $X$, denoted by $\overline{x}\to X$, we mean a map of adic spaces $\overline{x}\colon \Spa(C,C^+)\to X$, where $C$ is a complete non-Archimedean algebraically closed field and $C^+\subseteq C$ is an open and bounded valuation subring.
	For such a geometric point $\bar{x}\to X$, we denote the associated field by $k(\bar{x})$ or only $k(x)$ if $x\to X$ is a general point (we will use this analogously in the scheme case).
    Whenever $C^+=O_C$ we say that $\overline{x}\to X$ is geometric rank $1$ point. 
    If $\overline{x}$ is an isomorphism we will say that $X$ (the space itself) is a geometric point.

    As in the introduction, $E$ denotes a non-Archimedean local field, we let $O_E\subseteq E$ denote the ring of integers, we let $\pi\in O_E$ denote a choice of uniformizer, we assume that $\bbF_q=O_E/\pi$, we denote by $\bbC$ a fixed completed algebraic closure of $E$.

    We let $\PSch$ denote the category of perfect affine schemes over $\bbF_q$.
    If $\calS=\Spec(A)\in \PSch$ with associated v-sheaf $\calS^{\diamond}=\Spd(A,A)$, we denote by $\bbW A$ the topological ring of $O_E$-Witt vectors which has the multiplicative Teichmüller section $[-]\colon A \to \bbW A$.
    More precisely, if $E$ is of characteristic $0$ then $\bbW A\coloneqq\calW(A)\otimes_{\bbZ_p}O_E$, where $\calW(A)$ denotes the $p$-typical Witt vectors, and if $E$ is of characteristic $p$ then $\bbW A=A\hat{\otimes}_{\bbF_q} O_E\cong A[[\pi]]$.
    We let $Y_{\calS^{\diamond}}\!\coloneqq\Spa \bbW A[\frac{1}{\pi}]$, this is an analytic sous-perfectoid adic space (see \Cref{ffschematic}). 
    Note that $A \mapsto \bbW A$ is a functor, and hence $\calS \mapsto Y_{\calS^{\diamond}}$ also is. 
    In particular, we have a functorial lift $\varphi\colon Y_{\calS^{\diamond}}\to Y_{\calS^{\diamond}}$ of the absolute Frobenius on $\calS$.

    \smallskip
    
    We let $\Perf$ denote the category of affinoid perfectoid spaces over $\bbF_q$.
    Fix $S\in \Perf$. 
    If $S=\Spa(R,R^+)$, we let $\bbA_{\on{inf}}(R^+)$ denote $\bbW(R^+)$ endowed with the $(\pi,[\varpi])$-adic topology, where $\varpi$ denotes a(ny) uniformizer of $R^+$.
    We let $\calY_S$ be the locus in $\Spa \bbA_{\on{inf}}(R^+)$ where $[\varpi]\neq 0$ for some pseudo-uniformizer $\varpi\in R^+$, and $Y_S$ be the locus in $\Spa \bbA_{\on{inf}}(R^+)$ where $\pi\cdot [\varpi]\neq 0$. 
    These are sous-perfectoid adic spaces \cite[Proposition II.1.1]{FS24}.
    Recall \cite[\S 12.2]{SW20} that after fixing a pseudo-uniformizer $\varpi\in R^+$ we have a continuous function
    \[\kappa_{\varpi}\colon |\calY_S|\to [0,\infty]\,.\]
    For every interval $I\subseteq [0,\infty]$ we let $\calY_{S,I}\subseteq \calY_S$ denote the interior of $\kappa^{-1}_{\varpi}(I)$.
    We use the notation $B_{S,I}$ to denote $\on{H}^0(\calY_{S,I},\calO)$ and $B^+_{S,I}$ to denote $\on{H}^0(\calY_{S,I},\calO^+)$.
    If $S$ is understood from the context we may drop it from the notation.

	\begin{remark}
		We want to point out that in this article we cite several references that discuss the geometry and the vector bundle theory of $\calY$ and related geometric objects. 
		Occasionally, these sources have the working assumption that $E= \Q_p$ or that $E$ is a characteristic $0$ field. 
		Nevertheless, the proofs often generalize to all non-Archimedean local fields $E$.
		For convenience of the reader, we have compiled a list of statements that we cite which although stated under restrictions on $E$ hold more generally: The specified references are 
		\cite[Theorem 3.8]{Kedlaya_16_ringAinf}, \cite[Proposition 2.1.3]{PR21}, \cite[Theorem 3.6\,(ii)]{RR_96} and \cite[Proposition 11.16]{Zha23}.
	\end{remark}

    \subsection{Grothendieck topologies}
    We endow $\Adic$ with the analytic topology.
    We endow $\Perf$ with the v-topology \cite[Definition 8.1]{Sch17}.
    We will consider several topologies on $\PSch$, mainly the scheme-theoretic v-topology \cite[Definition 3.2]{bhatt_scholze_projectivity_of_the_witt_vector_affine_grassmannian}, the arc-topology \cite{BhattMathew_21} and the pro\'etale topology \cite{BS15}.

    For the convenience of the reader we recall how these topologies are defined. 
    Since we will work with sheaves of categories, we prefer to use the language of covering sieves \cite[\S 6.2.2]{lurie2009htt}. 
    All the topologies that we consider below are finitary. 
    This means that if $X\in \Perf$ (or $X\in \PSch$) and $\calU\subseteq \Perf_{\!/X}$ is a covering sieve, then there is a finite set of objects $\{Y_i\to X\}_{i=1}^n\!\subseteq \calU$ such that the sieve generated by the $Y_i$ is a covering sieve of $X$ which refines $\calU$. 
    Since both $\Perf$ and $\PSch$ admit finite disjoint unions, every covering can be refined by the covering sieve generated by one map $Y\to X$.
    Therefore, it suffices to specify which maps of affinoid perfectoid spaces $\Spa(A,A^+)\to \Spa(B,B^+)$ (or of affine schemes $\Spec A\to \Spec B$) are covers.

    The v-topology on $\Perf$ declares every map $\Spa(A,A^+)\to \Spa(B,B^+)$ to be a cover as long as $|\Spa(A,A^+)|\to |\Spa(B,B^+)|$ is surjective.

    A map of affine schemes $\Spec A\to \Spec B$ is:
    \begin{enumerate}
	    \item A pro\'etale cover if it is weakly \'etale in the sense of \cite[Definition 1.2]{BS15}. 
	    \item A v-cover if for every map $B\to V$ with $V$ a valuation ring there is an extension of valuation rings $V\subseteq W$ and a commutative diagram 
		    \begin{center}
		    \begin{tikzcd}
		     \Spec W \arrow{r} \arrow{d}  & \Spec A \arrow{d} \\
		     \Spec V \arrow{r} & \Spec B\,.
		    \end{tikzcd}
		    \end{center}
	    \item An arc-cover if for every map $B\to V$ with $V$ a rank 1 valuation ring there is an extension of rank $1$ valuation rings $V\subseteq W$ and a commutative diagram as above.
    \end{enumerate}

   We let $\Sets$ denote the category of sets, we let $\Grps$ denote the $(2,0)$-category of groupoids. 
   We let $\on{Cat}_1$ be the $(2,1)$-category of small categories.
   We let $\Cat$ (resp.\ $\CatOE$, $\CatE$) denote the $(2,1)$-category of rigid symmetric monoidal small additive categories
   (resp.\ $O_E$-linear, $E$-linear symmetric monoidal small additive categories). 
   Moreover, we let $\Catex$, $\CatexOE$, $\CatexE$ denote the versions of $\Cat$, $\CatOE$, $\CatE$ equipped with exact structure.
   We clarify in \Cref{Appendix-sect} the precise meaning of these categories, and we show that they are presentable and compactly generated.
   The main purpose \Cref{Appendix-sect} is to justify that sheafification of presheaves with values in the above categories behaves as expected.
   The reader is encouraged to ignore the appendix and follow their intuition of what sheafification does and means. \\

   In the body of the text we will be interested in extending the domain of definition of functors that originally are only defined over $\Perf$ (or $\PSch$).  
   This is the case because many interesting spaces (like diamonds) are not members of $\Perf$, but can be regarded as functors over $\Perf$. 
   For the convenience of the reader we recall how this can be done.
   We will need the following additional notation.
   We let $\Ani$ denote the $(\infty,0)$-category of anima and $\widehat{\on{Cat}}_\infty$ the $(\infty,1)$-category of large categories. 
   We identify the category of $(2,1)$-categories with the full subcategory of $2$-truncated objects of $\widehat{\on{Cat}}_\infty$, thus regarding $\Cat$ and all of its versions above as objects in $\widehat{\on{Cat}}_\infty$. In the following discussion category means $\infty$-category.
   Fix $\calT\in \{\Perf\!, \PSch\}$.

    \begin{definition}
	    \label{convention-sheaves}
	    Given a complete category $\calC$ and a functor $\calF\colon \calT^\op\to \calC$, we will say that $\calF$ is a $\calC$-valued v-sheaf (analogously $\calC$-valued pro\'etale sheaf, scheme-theoretic v-sheaf, and arc-sheaf) if for any $X\in \Perf$ and any covering sieve $\calS\subseteq \Perf_{\!/X}$ the limit map 
	    \[\calF(X)\to\!\!\!\!\!\varprojlim_{\{Y\to X\}\in \calS}\!\!\!\!\calF(Y)\] 
	    is an equivalence in $\calC$.
    \end{definition}

    For a category $\calC$, we let $\calP(\calT,\calC)$ denote the category of $\calC$-valued presheaves, and if $\calC$ is complete, we let $\calS(\calT,\calC)\subseteq \calP(\Perf\!,\calC)$ be the full subcategory of $\calC$-valued v-sheaves as in \Cref{convention-sheaves}.
    If $\calC=\Ani$, we simply write $\calP(\calT)$ and $\calS(\calT)$.
    If $\calC=\Grps$, we write $\Perff$ (resp.\ $\PSchf$) instead of $\calS(\Perf\!,\Grps)$ (resp.\ $\calS(\PSch\!,\Grps)$). 
	Every geometric object we consider lies either in $\Perff$ or in $\PSchf$, and we refer to them as small v-stacks or small scheme-theoretic v-stacks.

    \begin{remark}
	    \label{presheaves-sheaves}
	    Recall that the Yoneda embedding $y\colon \calT\to \calP(\calT)$ realizes $\calP(\calT)$ as the $(\infty,1)$-category freely generated by $\calT$ under small colimits. 
	    In particular, if $\calC$ is a cocomplete $(\infty,1)$-category, then any functor $\calF\colon \calT\to \calC$ can uniquely be extended to a cocontinuous functor $\calG\colon \calP(\calT)\to \calC$. 
	    By abuse of notation, we will still denote $\calG$ by $\calF$.
    \end{remark}

    \begin{remark}
	    In the context of \Cref{presheaves-sheaves}, if $\calC$ is presentable \cite[\S 5.5.1]{lurie2009htt}, the inclusion $\calS(\calT,\calC)\subseteq \calP(\calT,\calC)$ admits a left-adjoint \cite [Lemma A.4.15.]{heyer20246} 
	    \[(-)^\sh\colon \calP(\calT,\calC)\to \calS(\calT,\calC)\]
	    which we can call \textit{sheafification}. 
		Furthermore, if $\calF\in \calS(\calT,\calC)$ the natural cocontinuous map 
	    \[\calF\colon \calP(\calT)\to \calC\]
	    described above factors as a composition $\calP(\calT)\xrightarrow{\sh}\calS(\calT)\to \calC$.
    \end{remark}

    Since $\Perff\subseteq \calP(\Perf)$ is a full subcategory, given any functor $\calF\colon \Perf\to \calC$ with $\calC$ a cocomplete category, one can contemplate the value $\calF(X)\in \calC$ for any $X\in \Perff$ by \Cref{presheaves-sheaves}. 
    An analogous statement holds for $X\in \PSchf$.
    This allows us to make several interesting constructions.
   
\begin{definition}
	\label{analytifications}
	Let $\calF\in\calP(\PSch\!,\calC)$.
	We define 
	\[\calF^{\diamond_\pre},\calF^{\dagger_\pre},\calF^{\Diamond_\pre}\in \calP(\Perf\!,\calC)\]
	with formula
	\begin{enumerate}
		\item $\calF^{\diamond_\pre}(R,R^+)\coloneqq\calF(\Spec R^+)$.
		\item $\calF^{\dagger_\pre}(R,R^+)\coloneqq\calF(\Spec R^\circ)$.
		\item $\calF^{\Diamond_\pre}(R,R^+)\coloneqq\calF(\Spec R)$.
	\end{enumerate}
	If $\calC$ is presentable, we define 
	\[\calF^{\diamond},\calF^{\dagger},\calF^{\Diamond}\in \calS(\Perf\!,\calC)\]
	by applying sheafification to the functors above.
We call these constructions \textit{the small diamond}, \textit{the dagger} and \textit{the big diamond} constructions, respectively.
\end{definition}

\begin{remark}
Since for every $\Spa(R,R^+)$ we have ring inclusions $R^+\subseteq R^\circ \subseteq R$, we obtain morphisms in $\calS(\Perf\!,\calC)$
\[\calF^\diamond\to \calF^\dagger \to \calF^\Diamond\,.\]
\end{remark}

\begin{example}
Given $\calS=\Spec(A)\in \PSch$, one can think of $\calS$ via the Yoneda embedding as $\calS\in \PSchf$.
Then one can verify that 
\begin{align*}
	\calS^{\diamond_\pre}(R,R^+)&=\{A\to R^+\!\mid \text{ring maps}\}\,, \\
	\calS^{\dagger_\pre}(R,R^+)&=\{A\to R^\circ\!\mid \text{ring maps}\}\,, \\
	\calS^{\Diamond_\pre}(R,R^+)&=\{A\to R \mid \text{ring maps}\}\,.
\end{align*}
Moreover, by \cite[Theorem 8.7]{Sch17} these functors satisfy v-descent. Thus for all $?\in \{\diamond, \dagger, \Diamond\}$ one has $S^?=S^{?_\pre}$.
Furthermore, one can see concretely that $\calS^\diamond$ is represented by $\Spd A$ when $A$ is endowed with the discrete topology, and that $\calS^\Diamond$ is represented by $\Spd(A,\bbF_q^{\on{min}})$, where $\bbF_q^{\on{min}}\subseteq A$ is the integral closure of $\bbF_q$ in $A$.
On the other hand, $\calS^\dagger$ is not representable, but it can be identified with a closed subsheaf of $\calS^\Diamond$ (the bounded locus, see \cite[Proposition 2.25, Definition 2.2]{Gle22}).
\end{example}

\begin{definition}
	Let $\calC$ be a complete category and let $\calF\in\calP(\Perf\!,\calC)$.
	We define 
	\[\calF^{(\red_\pre)}\in \calP(\PSch\!,\calC)\]
	by setting
	$\calF^{(\red_\pre)}(\Spec R)\coloneqq\calF(\Spec R^\diamond)$.
	Here we use \Cref{presheaves-sheaves} to make sense of the right-side term.
	If $\calC$ is presentable, we define 
	\[\calF^\red \in \calS(\PSch\!,\calC)\,,\]
	as the v-sheafification of $\calF^{(\red_\pre)}$. 
\end{definition}

\begin{remark}\label{smdiamredadjunction}
    As it turns out, the functor 
    \[\diamond\colon \PSch\to \widetilde{\Perf}\]
    turns scheme-theoretic v-covers into surjective maps of v-sheaves by \cite[Proposition 3.7]{Gle22}.
    For this reason, if $\calF\in \calS(\Perf\!,\calC)$ then $\calF^{(\red_\pre)}\in \calS(\PSch\!,\calC)$, and 
    $\calF^\red(\Spec A)=\calF(\Spec A^\diamond)$.
    This shows that $(-)^\red\colon \Perff \to \PSchf$ is a right-adjoint to $\diamond\colon \PSchf\to \Perff$.
\end{remark}

We consider a final family of constructions. 
Given a topological ring $R$ we let $R_{\disc}$ denote the ring $R$ endowed with the discrete topology.

\begin{definition}
	\label{loci-functors}
	Let $\calC$ be a complete category and let $\calF\in\calP(\Perf\!,\calC)$.
	We define, using \Cref{presheaves-sheaves},
	\[ \calF^{(\mer_\pre)}, \calF^{(\An_\pre)}  \in \calP(\Perf\!,\calC)\]
	with formulas
	\begin{enumerate}
		\item $\calF^{(\mer_\pre)}(R,R^+)\coloneqq\calF(\Spd(R_\disc, R^+_\disc))$\,,
		\item $\calF^{(\An_\pre)}(R,R^+)\coloneqq\calF(\Spd(R_\disc, R_\disc))$\,.
	\end{enumerate}
	If $\calC$ is presentable we define
	\[ \calF^\mer, \calF^\An  \in \calS(\Perf\!,\calC)\]
	as the v-sheafification of the functors defined above.
	When $\calC=\Grps$ we call
	\[ \calF^\mer, \calF^\An\colon\Perff\to \Perff\]
	the \textit{meromorphic} and the \textit{unbounded} locus functors of $\calF$.
\end{definition}

\begin{remark}
	\label{general-functoriality}
	For any affinoid perfectoid $\Spa(R,R^+)\in \Perf$ we have maps $(R,R^+) \leftarrow (R_{\rm disc},R_{\rm disc}^+) \to (R_{\rm disc},R_{\rm disc})$ of Huber pairs. 
	Thus, if $\calC$ is presentable, we obtain maps \begin{center}
\begin{tikzcd}
	\calF^\mer \arrow[r] \arrow[d] & \calF \\
	\calF^\An. 
\end{tikzcd}
\end{center} 
	One can easily show that $\calF^\An$ coincides with $(\calF^\red)^\Diamond$.
\end{remark}

One of the main goals of this article is to better understand the correspondence 

\begin{center}
\begin{tikzcd}
	(\Bun_G)^\mer \arrow[r] \arrow[d] & \Bun_G \\
	(\Bun_G)^\An 
\end{tikzcd}
\end{center}
that one obtains from applying the considerations of \Cref{general-functoriality} to $\Bun_G\in \Perff$.



    \subsection{Combs and product of points}
    The big advantage of working with the v-topology and the scheme-theoretic v-topology is that plenty of questions can be reduced to studying valuation rings and their ultra-products. 

    \begin{definition}\label{defcontractible}
	    \begin{enumerate}
	\item
	We say that an affine scheme $\calS=\Spec A$ is a \textit{comb} if for all $x\in \pi_0(\calS)$ the closed subscheme attached to $x$ is of the form $\Spec V_x$, where $V_x$ is a valuation ring with algebraically closed fraction field.
    \item 
	    We say that a comb is an \textit{extremally disconnected comb} if it is w-contractible in the sense of \cite[Definition 2.4.1]{BS15}. 
    \item 
	    If $A=\prod_{i\in I}\! V_i$, where each $V_i$ is a valuation ring with algebraically closed fraction field, then we say that $\calS$ is a \textit{product comb}.
	\item \cite[Definition 1.1]{MW23}
		We say that a perfectoid space $X \in \Perf$ is w-contractible if it is a w-strictly local perfectoid space and $\pi_0(X)$ is extremally disconnected.
	    \end{enumerate}
    \end{definition}

\begin{remark} \label{rem:scheme_v_covers} 
	By \cite[Theorem 1.8]{BS15}, extremally disconnected combs can be characterized as those combs $S$ such that: (1) $S$ is w-local, (2) $\pi_0(S)$ is extremally disconnected. 
	Indeed, every local ring of $S$ is a valuation ring with algebraically closed fraction field, and these are strictly Henselian local rings.
	
	We observe that product combs are extremally disconnected combs.
	Indeed, the proof of \cite[Lemma 6.2]{bhatt_scholze_projectivity_of_the_witt_vector_affine_grassmannian} precisely shows this.
	Moreover, \cite[Lemma 6.2]{bhatt_scholze_projectivity_of_the_witt_vector_affine_grassmannian} also shows any qcqs scheme admits a v-cover by a product comb.
	In particular, product combs are a basis for the scheme-theoretic v-topology.
\end{remark}

    \begin{definition}
	    Suppose that $\calS = \Spec A \in \PSch$ is a product comb and $\varpi\in A$ is a non-zero divisor. 
		Let $R^+=\widehat{A}_{\varpi}$ be the $\varpi$-adic completion of $A$ and let $R=R^+[\frac{1}{\varpi}]$. 
		Then $\Spa (R,R^+)$ is a strictly totally disconnected space, and we call any space obtained this way a \textit{product of points}.
    \end{definition}

    \begin{remark}
	    \label{remark-product-of-points}
	    It follows from \cite[Example 1.1]{Gle22} that products of points form a basis for the v-topology on $\Perf$.
		Moreover, if $S \in \Perf$ is a product of points, it is also w-contractible (in the sense of \Cref{defcontractible} (4)).
		If $S=\Spa(R,R^+)$, and $\varpi\in R^+$ is a pseudo-uniformizer, then the two properties follow from the fact that the specialization map 
		\[\Spa(R,R^+)\to \Spec R^+/\varpi\]
		is a homeomorphism (see \cite[Proposition 4.3]{Gle22}).
		Indeed, by \Cref{two-versions-of-product-of-combs} $\Spec R^+$ is a product comb, and hence by \Cref{rem:scheme_v_covers} it is w-local and $\pi_0(\Spec R^+)$ is extremally disconnected.
		But the two properties pass to the closed subset $\Spec R^+/\varpi\subseteq \Spec R^+$, which allows us to conclude.
    \end{remark}

    \begin{lemma}
	\label{quasi-pro-etale-admitsection}
Let $S$ be a product of points and $X \to S$ a pro-étale cover. Then, this morphism has a section $S \to X$.
\end{lemma}
\begin{proof}
	By \Cref{remark-product-of-points}, $S$ is $w$-contractible.
	Then the result follows from  \cite[Lemma 1.2]{MW23}.
\end{proof}

\begin{proposition}
	\label{product-combs-bezout}
	Product combs are B\'ezout rings (i.e.\ every finitely generated ideal is generated by a single element).	
\end{proposition}
\begin{proof}
	Let $R=\prod_{i\in I} V_i$ with each $V_i$ a valuation ring.
	Let $a,b\in R$ be two elements. 
	Let $I_a\subseteq I$ (resp.\ $I_b\subseteq I$) denote the subset of those $i\in I$ such that $a_i\mid  b_i$ (resp.\ $b_i\mid a_i$).
	Since each $V_i$ is a valuation ring, we have $I=I_a\cup I_b$. 
	For a subset $S\subseteq I$, we let $e_S$ denote the idempotent with $(e_S)_i=1$ if and only if $i\in S$.
	Let $c:=e_{I_a\setminus I_b}\cdot a+e_{I_b}\cdot b$. 
	We observe that $c$ is in the ideal generated by $a$ and $b$, that $c\mid a$ and that $c\mid b$. 
\end{proof}

    \begin{proposition}
	    \label{two-versions-of-product-of-combs}
	    Let $\calS=\Spec A$ be a product comb with $A=\prod_{i\in I}\!V_i$ and $\varpi\in A$ a non-zero divisor. Let $R^+=\widehat{A}_\varpi$ be the $\varpi$-adic completion. Let $\varpi_i$ be the image of $\varpi$ in $V_i$ which is also a non-zero divisor. Let $K_i^+=\widehat{V}_{i,\varpi_i}$ be the $\varpi_i$-adic completion. 
		Then, the family of projection maps $R^+\to K_i^+$ induces a ring isomorphism $R^+\!=\prod_{i\in I}\! K_i^+$.
    \end{proposition}
    \begin{proof}
	    Let $I\times \bbN$ be the partial order with $(i_1,n_1)\leq (i_2,n_2)$ if $i_1=i_2$ and $n_1\leq n_2$. 
	    We have a functor from $I\times \bbN$ to the category of rings sending $(i,n)$ to $V_i/\varpi_i^n$. 
	    The constructions of $R^+$ and $\prod_{i\in I}\! K_i^+$ correspond to two different ways of computing the limit of this diagram.
    \end{proof}

    \begin{proposition}
	    \label{combs_and-products}
	    If $\Spa (R,R^+)$ is a product of points, then $\Spec R$ is a comb.
    \end{proposition}
    \begin{proof}
	    By \Cref{two-versions-of-product-of-combs}, $R^+\!=\prod_{i\in I}\! C_i^+$, where $C^+_i$ are valuation rings with algebraically closed fraction fields.
		Since ultraproducts of valuation rings with algebraically closed fraction field are again valuation rings with algebraically closed fraction fields, $\Spec R^+$ is a comb.
		Now, since affine Zariski localizations of combs are combs again, $\Spec R$ is a comb.
    \end{proof}

We will use the following proposition implicitly throughout the article.

    \begin{proposition}
	    \label{quotient-by-groups}
	    Let $H$ be a locally profinite group. Then
		\[[\ast\!/\underline{H}]^\lozenge=[\ast\!/\underline{H}]^\diamond=[\ast\!/\underline{H}]\,,\]
		where $\underline{H}(S) = C^0(|S|,H)$ parametrizes continuous maps from $|S|$ to $H$ for either $S \in \PSch$ (for the first two terms) or $S \in \Perf$ (for the last term).
    \end{proposition}
    \begin{proof}
	    Let $\Spa(R,R^+)\in \Perf$. 
	    Observe that $\underline{\smash{H}}^\diamond=\underline{H}^\lozenge=\underline{H}$.
	    Since $\lozenge$ (resp.\ $\diamond$) commutes with limits, it suffices to prove that the map $\ast\to [\ast\!/\underline{H}]^\lozenge$ (resp.\ $\ast\to [\ast\!/\underline{H}]^\diamond$) is surjective.
	    This amounts to showing that if $\calF$ is a $\underline{H}$-torsor for the scheme-theoretic v-topology over $\Spec R$ (resp.\ $\Spec R^+)$, then there is an analytic v-cover of $\Spa (R',R'^+)\to \Spa(R,R^+)$ such that $\calF$ restricted to $\Spec R'$ is trivial.
	    We can take $\Spa (R',R'^{+})$ to be a product of points.
	    It follows from a theorem of Gabber \cite[Theorem 1.5]{HS21} that every $\underline{H}$-torsor is pro-\'etale locally trivial. 
	    Since $\Spec R'^+$ is an extremally disconnected comb by \Cref{rem:scheme_v_covers}, every pro-\'etale cover splits, which settles the identity $[\ast\!/\underline{H}]^\diamond=[\ast\!/\underline{H}]$.
For the case $[\ast\!/\underline{H}]^\lozenge=[\ast\!/\underline{H}]$ we have to be more subtle since $\Spec R'$ is not an extremally disconnected comb, in fact it is not w-local. 

We still claim that every $\underline{H}$-torsor over $\Spec R'$ is trivial.
Fix $\calF$ an $\underline{H}$-torsor over $\Spec R'$. 
Let $K\subseteq H$ be an open subgroup that is profinite, and write $K=\varprojlim_{n\in \bbN} K/K_n$ with $K_n$ a decreasing sequence of normal open subgroups that converge to $0$.
We may write $\calF=\varprojlim_{n\in \bbN} \calF/\underline{K_n}$ which is a limit of spaces over $\Spec R'$ with surjective finite \'etale transition maps. 
We note that by \cite[Theorem 1.5]{HS21} each map $\calF/\underline{K_n}\to \Spec R'$ is an \'etale cover.
By \Cref{combs_and-products} $\Spec R'$ is a comb, and in particular it splits every \'etale cover. 
Thus we may find a section $\Spec R'\to \calF/\underline{K}$. 
We may construct compatible sections $\Spec R'\to \calF/\underline{K_n}$ since each transition map is finite \'etale and surjective.
Overall, this gives a section $\Spec R'\to \calF$, which shows that $\calF$ is isomorphic to the trivial $\underline{H}$-torsor. 
    \end{proof}

\section{Categories of vector bundles}
Our starting point is the functor 
\[\calV\colon (\Adic)^\op\to \Catex\]
which takes uniform analytic adic spaces (as in \cite[Definition 5.3.1]{SW20}) to its category of vector bundles (i.e.\ locally finite free $\calO_X$-modules). 
This is well-behaved, satisfies analytic descent, and satisfies that $\calV(\Spa(R,R^+))\simeq \on{Proj}(R)$ for any stably uniform Huber pair $(R,R^+)$, where the latter denotes the symmetric monoidal additive category of finite dimensional projective $R$-modules \cite[Theorem 5.2.8]{SW20}, \cite[Theorem 2.7.7]{kedlaya_liu_relative_p_adic_hodge_theory_foundations}.
We note that the exact structure on $\on{Proj}(R)$ is induced by its additive structure (this implies by analytic descent that a sequence of vector bundles on an adic space $\Sigma\coloneqq[\calE_1\to \calE_2\to \calE_3]$ is exact if and only if it is analytic locally split exact).

We wish to study the functor $\calV$, and our main technique will be to reduce general statements about $\calV(X)$ to statements about $\calV(\bar{x})$ as $\bar{x}\to X$ ranges over the geometric points of $X$. 
We record some observations that are useful for this purpose.

\begin{lemma}\label{geompts_inj}
	Let $X$ be either a perfect scheme or an uniform analytic adic space. 
	Then, the map 
	$$\calO_X(X) \to \prod_{\bar{x}\to X} k(\bar{x})$$
	is injective, where the $\bar{x}\to X$ run over all geometric points (in either the scheme-theoretic or analytic sense) of $X$.
\end{lemma}
\begin{proof}
This reduces to $X$ being either affine or affinoid. 
In the scheme-theoretic case, this follows since perfect rings are reduced (and hence the intersection of all prime ideals is the zero ideal).
In the perfectoid case, this is implied by \cite[Theorem 5.2.1]{SW20}.
\end{proof}

At this point we note that for a geometric point $\Spa(C,C^+)$ with canonical inclusion $\Spa(C,O_C)\subseteq \Spa(C,C^+)$ we obtain equivalences in $\Catex$
\[\calV(\Spa(C,C^+))\simeq \calV(\Spa(C,O_C))\simeq \on{Proj}(C)\,.\]
In particular, for most of the statements that we discuss below one can replace the usage of geometric point by that of rank $1$ geometric point. 
\begin{lemma}
	\label{analytic-vector-bundles-exact-geometric}
	Let $X$ be a uniform analytic adic space and let $\Sigma\coloneqq[\calE_1\to \calE_2\to \calE_3]$ be a sequence of vector bundles over $X$. 
	The sequence $\Sigma$ is exact if and only if for every geometric point $\overline{x}\to X$ the restricted sequence $\overline{x}^*\Sigma$ is exact.
\end{lemma}
\begin{proof}
	Since we can verify exactness locally (by analytic descent), we may assume that $X=\Spa(R,R^+)$ with $(R,R^+)$ a stably uniform Huber pair and that each $\calE_i$ is free. 
	Then, if $\Sigma$ is an exact sequence, it is clear that the restricted sequences $\bar{x}^*\Sigma$ are exact, since all the $\calE_i$ are free (and hence tensoring with any module preserves exactness).
	We now assume that all the restricted sequences $\bar{x}^*\Sigma$ are exact and want to show that $\Sigma$ itself is an exact sequence.
	We may verify that $\Sigma$ is a complex on geometric points by \Cref{geompts_inj}. 
	Injectivity on the left can be done similarly. 
	It suffices to prove surjectivity, since exactness in the middle follows from this.
	We can reinterpret $\Sigma$ as a sequence over $\on{Spec} R$. 
	Passing to determinant bundles, we may assume that $\calE_3=R$ and that the image of $\calE_2$ is an ideal $I\subseteq R$.
	We now assume that the map is not surjective.
	Then, $I$ is proper, and hence contained in a maximal ideal $I\subseteq \frakm\subseteq R$.
	We note that $k(\frakm)$ with the induced quotient topology is a complete Tate Huber ring, since $\frakm \subset R$ is closed.
	Let us denote by $\bar{k}(\frakm)^{\widehat{}}$ the completion of the algebraic closure of the residue field $k(\frakm)$, which induces a geometric point $\bar{x}\colon\! \Spa(\bar{k}(\frakm)^{\widehat{}}, \bar{k}(\frakm)^{\widehat{}\ ,+}) \to \Spa(R,R^+)$.
	Then, $\bar{x}^{*}(\calE_2 \to \calE_3)$ factors as
	\[\calE_2\otimes_R \bar{k}(\frakm)^{\widehat{}}\to I\otimes_R \bar{k}(\frakm)^{\widehat{}} = 0 \to \bar{k}(\frakm)^{\widehat{}}\]
	and hence $x^*\Sigma$ is not an exact sequence. 
	We have constructed a geometric point $\overline{x}\to \Spa (R,R^+)$ for which $\overline{x}^*\Sigma$ is not exact, contradicting our assumption.
\end{proof}

When we restrict the functor $\calV$ to $\Perf$ one has stronger descent results.
Indeed, by \cite[Lemma 17.1.8]{SW20}, $\calV\colon (\Perf)^\op\to \Catex$ is a v-sheaf.
Moreover, by \cite[Proposition 6.3.4]{SW20}, sous-perfectoid Huber pairs are stably uniform, so if $X=\Spa(R,R^+)$ for $(R,R^+)$ a sous-perfectoid Huber pair, then $X \in \Adic$ and $\calV(X)=\on{Proj}(R)$.
We consider functors
\[\calV_\calY\colon (\Perf)^\op\to \CatexOE \quad \text{ and } \quad \calV_Y\colon (\Perf)^\op\to \CatexE\]
with 
\[\calV_\calY(S)\coloneqq\{\text{Vector bundles over } \calY_S\} \quad
	\text{ and } \quad
\calV_Y(S)\coloneqq\{\text{Vector bundles over } Y_S\}\,.\]
It follows from \cite[Proposition 19.5.3]{SW20} that both $\calV_\calY$ and $\calV_Y$ are v-sheaves.
In this case, the exact structures on $\calV_\calY(S)$ and $\calV_Y(S)$ can be tested on geometric points of $S$ by \Cref{analytic-vector-bundles-exact-geometric}, so v-descent of the exact structure is immediate to verify. 

The category of vector bundles on the Fargues--Fontaine curve \cite[\S II.2]{FS24} can be defined by the pull-back square
\begin{equation}\label{Diag_BunFF}
\begin{tikzcd}
	\Bun_{\FF}\arrow{r} \arrow{d}  & \calV_Y \arrow{d}{\Delta} \\
	\calV_Y \arrow{r}{(\on{id},\varphi^*)} & \calV_Y\times \calV_Y\,. 
\end{tikzcd}
\end{equation}
In order to understand the functors $(\Bun_\FF)^\mer$ and $(\Bun_\FF)^\An$ giving rise to the correspondence
\begin{center}
\begin{tikzcd}
 (\Bun_\FF)^\mer \arrow{r} \arrow{d}  & \Bun_\FF\,   \\
 (\Bun_\FF)^\An & 
\end{tikzcd}
\end{center}
we will have to study the outcome of evaluating $\calV_Y$ on v-sheaves of the form $\Spd(R_\disc,R_\disc)$ and $\Spd(R_\disc,R^+_\disc)$.

\subsection{$E_\infty$-sous-perfectoid spaces}
We let $\mathrm{OpenSch}\subseteq \Perff$ denote the full subcategory of v-sheaves which open locally admit an open immersion into a space of the form $\calS^\diamond$ for $\calS\in \PSch$.
We note that by \Cref{Yfexamples}, this full subcategory contains $\Perf$.
The purpose of this section is, for $\calF\in \on{OpenSch}$, to construct the space $Y_{\calF}$ and to prove \Cref{theorem-evaluating-opensubs} below, which characterizes $\calV_Y(\calF)$.

The main observation is that the proof of \cite[Proposition 19.5.3]{SW20} can be generalized to a broad class of analytic adic spaces.
\begin{definition}
	A Huber ring $R$ over $E$ is called $E_{\infty}$-sousperfectoid if the completed tensor product $R \widehat{\otimes}_E E_{\infty}$ is perfectoid. 
	If we have an adic space $X=\Spa(A,A^+)$ over $\Spa(E)$ such that $A$ is $E_{\infty}$-sousperfectoid, we call $X$ \emph{affinoid $E_{\infty}$-sousperfectoid}.
\end{definition}

\begin{remark}
Note that $E_{\infty}$ is topologically countably generated over $E$ and hence topologically free by \cite[§2.7, Theorem 4]{BGR84}. 
This shows that being $E_{\infty}$-sousperfectoid implies being sousperfectoid \cite[Definition 6.3.1]{SW20}. 
Moreover, the proof of \cite[Proposition 6.3.3.(1)]{SW20} shows that being $E_{\infty}$-sousperfectoid is stable under rational localization.
\end{remark}

\begin{definition}
	We define the category $E_{\infty}\textup{-SPerfd}$ as the category of analytic adic spaces $X$ over $\Spa(E)$ such that $X$ can be covered by affinoid $E_{\infty}$-sousperfectoid spaces.
\end{definition}

\begin{definition}
	A collection of morphisms $\{f_i \colon Y_i \to X\}_{i \in I}$ between $E_{\infty}$-sousperfectoid spaces is a v-cover if for each quasicompact open subset $U \subseteq X$, there exists a finite subset $J \subseteq I$ and quasicompact open subsets $V_i \subseteq Y_i$ for $i \in J$, such that $U = \bigcup_{i \in J} f_i(V_i)$.
\end{definition}

\begin{proposition}
	The category $E_{\infty}\textup{-SPerfd}$ together with v-covers is a site.
\end{proposition}
\begin{proof}
	It suffices to show that fiber products exist, so let $Y \to X \leftarrow Z$ be a map of $E_{\infty}$-sousperfectoid spaces. 
	We can assume that $X= \Spa(A,A^+), Y=\Spa(B,B^+),Z=\Spa(C,C^+)$, where $A,B,C$ are Tate and their completed tensor product with $E_{\infty}$ over $E$ is perfectoid. 
	But then 
    $$(B \widehat{\otimes}_A C) \widehat{\otimes}_E E_{\infty} \cong (B \widehat{\otimes}_E E_{\infty}) \widehat{\otimes}_{A \widehat{\otimes}_E E_{\infty}} (C \widehat{\otimes}_E E_{\infty})$$
    and we know that the term on the right hand side is perfectoid since completed tensor products of perfectoid Tate Huber rings are perfectoid. 
	This shows in particular that $B \widehat{\otimes}_A C$ is again $E_{\infty}$-sousperfectoid and hence sheafy.
\end{proof}

\begin{example}
	As it turns out, $\Spa(E_{\infty})$ together with its natural map to $\Spa(E)$ is not an example of an $E_{\infty}$-sousperfectoid space.
	Indeed, when $\mathrm{char}(E)=p$, this is evident since $\mathbb{F}_q(\!(t^{1/p^{\infty}})\!) \widehat{\otimes}_{\mathbb{F}_q(\!(t)\!)} \mathbb{F}_q(\!(t^{1/p^{\infty}})\!)$ is not perfect.
	When $\mathrm{char}(E)=0$, we let $C \coloneqq \widehat{\bar{E}}$ and $B \coloneqq C \widehat{\otimes}_E C$ with ring of definition $B_0 \coloneqq \mathcal{O}_C \widehat{\otimes}_{\mathcal{O}_E} \mathcal{O}_C$.
	Using that completed tensor products of perfectoid rings are again perfectoid twice, it suffices to show that $B$ is not perfectoid by contradiction.
	For $n \in \mathbb{N}$, we fix a primitive $p^n$-th root of unity $\zeta_{p^n} \in C$. We now construct the elements
	$$a_n \coloneqq \sum_{i=0}^{p^n -1} \zeta_{p^n}^i\otimes \zeta_{p^n}^{-i} \in B.$$
	We can now define elements $b_n \coloneqq \frac{a_n}{p^{{n}/{2}}} \in B$ which are idempotent and hence power-bounded. 
	However, $p^{(n-1)/2} \cdot b_n \notin B_0$ for all $n$ such that $\zeta_{p^n} \notin E$ and hence for $n\gg0$. 
	This shows that $B$ is not uniform, and consequently it is not perfectoid.
	The same proof shows that $\mathbb{C}_p$ is not $E_{\infty}$-sousperfectoid.
\end{example}

We can generalize \cite[Proposition 19.5.3]{SW20} to perfect complexes. 
For an analytic adic space $X$, we denote by $\perfc(X)$ the $\infty$-category of perfect complexes on $X$ and by $\perfc^{[a,b]}(X)$ its subcategory of complexes of tor-amplitude $[a,b]$. 
We note that the presheaves
\begin{align*}
	\perfc \colon \Adic &\to \infcat \\
	X &\mapsto \perfc(X)
\end{align*}
and
\begin{align*}
	\perfc^{[a,b]} \colon \Adic &\to \infcat \\
	X &\mapsto \perfc^{[a,b]}(X)
\end{align*}
satisfy analytic descent by \cite[Theorem 1.4]{And21}.

\begin{theorem}\label{sousstackcomplex}
    The presheaves
    \begin{align*}
        \perfc \colon E_{\infty}\textup{-SPerfd} &\to \infcat \\
        X &\mapsto \perfc(X)
    \end{align*}
	and
	\begin{align*}
        \perfc^{[a,b]} \colon E_{\infty}\textup{-SPerfd} &\to \infcat \\
        X &\mapsto \perfc^{[a,b]}(X)
    \end{align*}
    are sheaves for the v-topology.
\end{theorem}
\begin{proof}
    The same proof method as in \cite[Proposition 2.3]{AB21} applies to this context.
\end{proof}

The category of vector bundles identifies with the full subcategory of perfect complexes which have tor-amplitude $[0,0]$. 
Moreover, exactness, an $E$-linear structure and a $\otimes$-structure can be checked after passing to a v-cover so we get the following:

\begin{corollary}
	\label{descent-of-vectors-Einfty-sp}
The association
\begin{align*}
E_{\infty}\textup{-SPerfd} &\to \CatexE \\
X &\mapsto \calV(X)
\end{align*}
is a sheaf for the v-topology.
\end{corollary}

\begin{corollary}
	The association
	\begin{align*}
		E_{\infty}\textup{-SPerfd} &\to \Grps\\
		X &\mapsto \{G\textup{-torsors over X}\}
	\end{align*}
	valued in Groupoids is a sheaf for the v-topology.
\end{corollary}

\begin{corollary}
The site of $E_{\infty}$-sousperfectoid spaces with the v-topology is subcanonical.
\end{corollary}
\begin{proof}
We can use Theorem \ref{sousstackcomplex} and proceed analogously as in the second part of \cite[Corollary 8.6]{Sch17}.
\end{proof}

The main observation about $E_{\infty}$-sousperfectoid spaces is that we can check a lot of properties of morphisms after base-changing to $E_{\infty}$:

\begin{proposition}\label{sperfconservative}
	The functor
    \begin{align*}
        E_{\infty}\textup{-SPerfd} &\to \mathrm{Perfd}/\!\Spa(E_{\infty}) \\
        X &\mapsto X \times_{\Spa(E)} \Spa(E_{\infty})
    \end{align*}
    is conservative and faithful.
\end{proposition}
\begin{proof}
	It suffices to show that the composition
    $$E_{\infty}\textup{-SPerfd} \to \mathrm{Perfd}/\!\Spa(E_{\infty}) \xrightarrow{(-)^{\flat}} \mathrm{Perf}/\!\Spa(E_{\infty})^{\flat}$$
    reflects isomorphisms. 
	Let $f\colon Y \to X$ be a morphism in $E_{\infty}-\mathrm{SPerfd}$ such that 
	$$f_{\infty}\colon (Y \times_{\Spa(E)} \Spa(E_{\infty}))^{\flat} \to (X \times_{\Spa(E)} \Spa(E_{\infty}))^{\flat}$$
    is an isomorphism. 
	Since taking diamonds preserves fiber products, we have a Cartesian square of v-sheaves
	\begin{center}
	    \begin{tikzcd}
	        Y \times_{\Spa(E)} \Spa(E_{\infty}) \ar[r, "f_{\infty}", "\cong"'] \ar[d] & X \ar[d] \times_{\Spa(E)} \ar[d] \Spa(E_{\infty}) \\
	       	Y^{\lozenge} \ar[r, "f^{\lozenge}"'] & X^{\lozenge}.
	   \end{tikzcd}
	\end{center}
	Then $f^{\lozenge}\colon Y^{\lozenge} \to X^{\lozenge}$ is an isomorphism, as this can be checked after pulling back via a v-surjective map of v-sheaves. 
	Then, \cite[Lemma 15.6]{Sch17} implies that $|f|\colon |Y| \to |X|$ is a homeomorphism. 
	Observing that there are naturally split inclusions of sheaves $\calO_X \hookrightarrow g_*\calO_{X \times_{\Spa(E)} \Spa(E_{\infty})}$ (and analogously for $\calO_Y$), $f$ is also an isomorphism on structure sheaves, which concludes the proof. 
	We can proceed analogously to show that the functor is faithful.
\end{proof}

We now construct the spaces $Y_{\calF}$.

\begin{lemma}\label{ffschematic}
	There is, up to unique isomorphism, a pair $(Y_{(-)},\cong)$, where $Y_{(-)}$ is a functor
    \begin{align*}
        Y_{(-)} \colon \mathrm{OpenSch} &\to E_{\infty}\textup{-SPerfd} \\
        \calF &\mapsto Y_{\calF}
    \end{align*}
	such that:
	\begin{enumerate}
		\item $Y_{\Spec(A)^{\diamond}} \coloneqq \Spa(\bbW(A)[\frac{1}{\pi}], \bbW(A))$,
		\item $Y_{(-)}$ preserves open immersions,
		\item and $\cong$ is a natural equivalence of functors in $\on{Fun}(\mathrm{OpenSch},\Perff)$
	\[(Y_{(-)})^\lozenge\cong (-)\times \Spd(E)\,.\]
	\end{enumerate}
	Furthermore, this construction commutes with products and more generally taking \v Cech nerves.
\end{lemma}
\begin{proof}
    Assume that $\calS =\Spec(A)$ is an affine perfect scheme. 
	We can use the same proof as in \cite[Proposition II.1.1]{FS24} to show that 
	$$Y_{\calS^{\diamond}}=\Spa(\bbW(A)[\frac{1}{\pi}],\bbW(A))$$ 
    is affinoid $E_{\infty}$-sousperfectoid, where $\bbW(A)$ carries the $\pi$-adic topology (see also \cite[Remark 13.1.2]{SW20}).
	Indeed, it suffices to show that 
    $$R=\bbW(A)[\frac{1}{\pi}] \widehat{\otimes}_{E} E_{\infty} \cong \bbW (A)[\frac{1}{\pi}] \widehat{\otimes}_{O_E} O_{E_{\infty}}$$ 
    is perfectoid. 
	Taking the open subring
    $$R_0=\bbW(A) \widehat{\otimes}_{O_E} O_{E_{\infty}}\,,$$
    we observe that $R_0/\pi^{1/p}\cong A \otimes_{\mathbb{F}_q} \mathbb{F}_q[t^{1/p^{\infty}}]/t^{1/p} \cong A[t^{1/p^{\infty}}]/t^{1/p}$ and $R_0/\pi \cong A \otimes_{\mathbb{F}_q} \mathbb{F}_q[t^{1/p^{\infty}}]/t \cong A[t^{1/p^{\infty}}]/t$.
	This implies that
    $$\Phi\colon R_0/\pi^{1/p} \xrightarrow{\cong} R_0/\pi$$
    is an isomorphism. 
	Using \cite[Lemma 3.10 (ii)]{BMS18}, we see that $R_0$ is integral perfectoid and \cite[Lemma 3.21]{BMS18} shows that $R_0[1/\pi] \cong R$ is perfectoid.
    
    We now want to show that $(Y_{\calS^{\diamond}})^{\lozenge} \cong \calS^{\diamond} \times \Spd(E)$. 
	Namely, we want to see that for any perfectoid space $T$ over $\Spa(E)$, giving a map $T \to Y_{\calS^\diamond}$ is equivalent to giving a map $T^{\flat} \to \calS^{\diamond}=\Spa(A,A)$. 
	We can assume that $T=\Spa(B,B^+)$ is affinoid. 
	In this case, giving a map $T \to Y_{S^{\diamond}}$ is the same as giving a map $\bbW(A) \to B^+$ (the image of $\pi$ in $B$ is always invertible since $B$ lives over $E$). 
	By the universal property of $\bbW(A)$, this is equivalent to giving a map $A\to B^+\!/\pi$ since $B^+$ is $\pi$-complete. 
	Since $A$ is perfect, this is in turn equivalent to giving a map $A \to (B^+)^{\flat}$, which is precisely a map $T^{\flat}\to \calS^{\diamond}=\Spa(A,A)$.
    
	We can now glue this to construct our desired space for general perfect schemes. 
	Note that the small diamond functor $\diamond \colon \widetilde{\PSch} \to \widetilde{\Perf}$ preserves open subsheaves (see \cite[Definition 2.10]{anschütz2022padictheorylocalmodels} and the comments right after), which allows us to glue by using Proposition \ref{sperfconservative}.
	For $\calF\in \mathrm{Open}$ with an open immersion $\calF\subseteq \calS^\diamond$, we get an open immersion of v-sheaves $\calF \times \Spd(E)\subseteq (Y_{\calS^{\diamond}})^{\lozenge}$.
	It follows from \cite[Lemma 15.6]{Sch17} that there is a unique (necessarily $E_\infty$-sousperfectoid) open subspace $Y_\calF\subseteq Y_{\calS^{\diamond}}$  inducing the open immersion $\calF\times \Spd E\subseteq (Y_{\calS^{\diamond}})^{\lozenge}$.
    Commutation with \v Cech nerves follows from \Cref{sperfconservative}. 
\end{proof}

\begin{remark}\label{Yfexamples}
    It is helpful to describe the spaces $Y_{\calF}$ from \Cref{ffschematic} explicitly for different v-sheaves $\calF$.
    Suppose that $X=\Spa(R,R^+)\in \Perf$ and that $\calS=\Spec R^+\in \PSch$.
    Fix $\varpi\in R^+$ a pseudo-uniformizer.
    Recall our notational convention that $\bbW(R^+)$ denotes the Witt vectors as a topological ring endowed with the $\pi$-adic topology while $\bbA_{\on{inf}}(R^+)$ also denotes the Witt vectors, but endowed with the $(\pi,[\varpi])$-adic topology.
    \begin{enumerate}
\item If $\calF=\Spd(R_\disc^+,R_\disc^+)=\calS^\diamond$ where $R_\disc^+=R^+$ carries the discrete topology we get the $E_{\infty}$-sousperfectoid space 
	$$Y_{\calF}=\Spa(\bbW(R^+)[\frac{1}{\pi}],\bbW(R^+)).$$ 
	\item 
		If $\calF=\Spd(R^+,R^+)$, where $R^+$ carries the $\varpi$-adic topology, then we have an open immersion $\calF\subseteq \calS^\diamond$ corresponding to the locus where $\varpi$ is topologically nilpotent (see \cite[Lemma 2.24]{Gle22}).
		Moreover, we get the $E_{\infty}$-sousperfectoid space
		$$Y_{\calF}=Y_{X,{(0,\infty]}}=\Spa(\bbA_{\on{inf}}(R^+),\bbA_{\on{inf}}(R^+)) \setminus V(\pi)\,.$$ 
        In other words, $Y_{\calF}$ is the curve with $\infty$ included.
	\item 
		If $\calF=\Spd(R_{\disc},R_\disc^+)$, where both rings carry the discrete topology, we have an open immersion $\calF\subseteq \calS^\diamond$ corresponding to the locus where $\varpi\neq 0$.
		Moreover, we get the $E_{\infty}$-sousperfectoid space
		$$Y_{\calF}=\Spa(\bbW(R^+)[\frac{1}{\pi}],\bbW(R^+))\setminus V([\varpi])\,.$$
        \item Finally, assume that $\calF = X$. Then we have an open immersion $X\subseteq \calS^\diamond$ corresponding to the intersection of the loci where $\varpi\neq 0$ and where $\varpi$ is topologically nilpotent.
		We get the $E_{\infty}$-sousperfectoid space
		$$Y_\calF=Y_{X,{(0,\infty)}}=\Spa(\bbA_{\on{inf}}(R^+),\bbA_{\on{inf}}(R^+)) \setminus V(\pi [\varpi])$$ 
		from \cite[Definition II.1.15]{FS24}, which satisfies that $Y_X^{\lozenge} \cong X \times \Spd(E)$. 
		Using \Cref{sperfconservative}, we can construct the spaces $Y_X$ for general perfectoid spaces $X$ (over $\mathbb{F}_q$).
    \end{enumerate}
\end{remark}

\begin{theorem}
	\label{theorem-evaluating-opensubs}
	For any $\calF\in \on{OpenSch}$ there is a functorial identification 
	\[\calV_Y(\calF)\simeq \calV(Y_\calF)\,.\]
\end{theorem}
\begin{proof}
	By definition, $\calV_Y(\calF)=\varprojlim \calV(Y_X)$ as $X$ varies in $\Perf\!/\calF$.
	By \Cref{ffschematic}, we have a functor
	\[Y_{(-)}\colon\on{OpenSch}\to E_{\infty}\textup{-SPerfd}\]
	that preserves v-covers.
	By \Cref{descent-of-vectors-Einfty-sp}, $\calV(Y_{(-)})$ satisfies v-descent. 
	Moreover, $\Perf\subseteq \on{OpenSch}$ is a basis for the v-topology, so $\Perf/\calF$ is a covering sieve.
	Thus we obtain
	\[\calV_Y(\calF)=\varprojlim \calV(Y_X)\cong \calV(Y_{(\varprojlim X)})=\calV(Y_\calF)\,.\]
\end{proof}
\begin{remark}
	The above proof has already found an application in \cite{anschuetz20246functorformalismsolidquasicoherent} in which a $6$-functor formalism $D_{(0,\infty)}(-)$ for solid quasi-coherent sheaves on the spaces $Y_{(-)}$ is constructed.
	Adjusting the above proof to this setting gives an analagous result $D_{(0,\infty)}(\Spd \bbF_p) \cong D_{\solid}(\mathrm{AnSpec}\,\Q_p)$, see \cite[Theorem 6.3.1]{anschuetz20246functorformalismsolidquasicoherent}.
\end{remark}

Let $X_S\coloneqq Y_S/\varphi^{\Z}$ denote the relative Fargues--Fontaine curve with respect to a perfectoid space $S$ from \cite[Definition II.1.15]{FS24}. 
We define the presheaves of $\infty$-categories

\begin{align*}
\perfc_Y \colon (\Perf)^{\op} &\to \on{Cat}_\infty \\
S &\mapsto \perfc(Y_S)
\end{align*}
and
\begin{align*}
	\perfc_{\mathrm{FF}} \colon (\Perf)^{\op} &\to \on{Cat}_\infty \\
	S &\mapsto \perfc(X_S)
	\end{align*}
which are sheaves by \Cref{sousstackcomplex}.
Now note that the proof of \Cref{theorem-evaluating-opensubs} generalizes to perfect complexes.

\begin{theorem}
	\label{theorem-evaluating-opensubs-perf}
	For any $\calF\in \on{OpenSch}$ there is a functorial identification of $\infty$-categories
	\[\perfc_Y(\calF)\simeq \perfc(Y_\calF)\,.\]
\end{theorem}

In particular, this positively answers \cite[Conjecture 1.2]{Anschuetz_absFF}. 
The third author wants to thank Johannes Anschütz for related discussions.

\begin{definition}
    Let $\calS \in \PSch$. 
	We define the category $\perfc_{\mathfrak{B}}(\calS)$ of perfect complexes of isocrystals over $\calS$ as the equalizer of $\infty$-categories
    \begin{center}
        \begin{tikzcd}
			\perfc_{\mathfrak{B}}(\calS) \ar[r] & \perfc(Y_{\calS^{\diamond}}) \ar[r,shift left=.75ex,"\mathrm{id}"]
  \ar[r,shift right=.75ex,swap,"\varphi^*"] & \perfc(Y_{\calS^{\diamond}})\,.
        \end{tikzcd}
    \end{center}
    We call a pair $(K, \alpha_K \colon K \cong \varphi^*K)$ in $\perfc_{\mathfrak{B}}(\calS)$ strictly perfect if it can be written as a finite limit of objects $(\mathcal{E}, \alpha_\mathcal{E} \colon \calE \cong \varphi^*\mathcal{E})$, where $\calE$ is a vector bundle on $Y_{\calS^\diamond}$.
	This is equivalent to $(K, \alpha_K)$ representing an honest bounded complex of isocrystals on $Y_{\calS^\diamond}$ (i.e.\ a sequence of maps between isocrystals, which form a bounded complex on the underlying vector bundles).
\end{definition}

\begin{theorem}
Let $\calS \in \PSch$.
The natural functor of $\infty$-categories
$$\perfc_{\mathfrak{B}}(\calS) \to \perfc_{\mathrm{FF}}(\calS^{\diamond})$$
is an equivalence. 
Moreover, the category $\perfc_{\mathfrak{B}}(\calS)$ is equivalent to the category of bounded complexes of isocrystals.
\end{theorem}
\begin{proof}
We note that for all $S \in \Perf$ there is an equalizer diagram of $\infty$-categories
\begin{equation}
	\begin{tikzcd}
		\perfc_{\mathrm{FF}}(S) \ar[r] & \perfc(Y_S) \ar[r,shift left=.75ex,"\mathrm{id}"]
\ar[r,shift right=.75ex,swap,"\varphi^*"] & \perfc(Y_S)\,.
	\end{tikzcd}
\end{equation}
The equivalence of categories then follows from Theorem \ref{theorem-evaluating-opensubs-perf} applied to $\calF = \calS^{\diamond}$. 
The second statement follows from \cite[Proposition 2.6]{AB21} applied to the equalizer diagram of $\infty$-categories
\begin{center}
	\begin{tikzcd}
		\perfc_{\mathfrak{B}}(\calS) \ar[r] & \perfc(Y_{\calS^{\diamond}}) \ar[r,shift left=.75ex,"\mathrm{id}"]
\ar[r,shift right=.75ex,swap,"\varphi^*"] & \perfc(Y_{\calS^{\diamond}})\,. 
	\end{tikzcd}
\end{center}
\end{proof}

As a consequence, we get:

\begin{corollary}
	Let $\lambda \in \Q^{\times}$ with associated absolute Banach--Colmez spaces $\mathcal{BC}_{\lambda,i}\!\coloneqq\mathcal{BC}(\calO(\lambda)[i])$ for $i\in\{0,1\}$\footnote{Here, we use the notation from \cite[Section II.2]{FS24}}. 
	Then $\mathcal{BC}_{\lambda,i}^{\on{red}}=0$, and the counit map $(\mathcal{BC}_{\lambda,i}^{\on{red}})^{\diamond} \to \mathcal{BC}_{\lambda,i}$ (given by \Cref{smdiamredadjunction}) corresponds to the zero section.\footnote{The identity $\mathcal{BC}_{\lambda,i}^{\on{red}}=0$ should be interpreted in the category of sheaves on $\PSch_{\bar{\bbF}_q}$ with values on $E$-vector spaces. More precisely, the $0$-section $\Spec {\bar{\bbF}}_q\xrightarrow{0} \mathcal{BC}_{\lambda,i}^{\on{red}}$ is an isomorphism.}
\end{corollary}
\begin{proof}
Let $\calS=\Spec(A) \in \PSch_{\bar{\bbF}_q}$. 
By \Cref{theorem-evaluating-opensubs-perf}, 
$$\mathcal{BC}_{\lambda,i}(S^{\diamond})=H^{i}(R\Gamma_{\perfc_{\mathfrak{B}}(\Spec(A))}(\calO(\lambda)))=0\,,$$ where $\calO(\lambda)$ denotes the simple standard isocrystal\footnote{Here, we use the notation from \Cref{signconvention}} over $\Spec(A)$ of slope $\lambda$.
A direct computation shows that 
$$R\Gamma_{\perfc_{\mathfrak{B}}(\Spec(A))}(\calO(\lambda)) = [\calO(\lambda) \xrightarrow{\varphi_{\calO(\lambda)}-\on{Id}} \calO(\lambda)]\overset{\lambda \neq 0}{\simeq}0\,,$$
where $\varphi_{\calO(\lambda)} \colon \calO(\lambda) \to \calO(\lambda)$ denotes the $\varphi$-linear automorphism of $\calO(\lambda)$.
\end{proof}

\subsection{On $\calV_Y^\mer$ and $\calV_Y^\An$}
In this subsection we analyze $\calV_Y^\mer$ and $\calV_Y^{\An}$, the emphasis will be on clarifying their structure as objects in $\calS(\Perf\!,\CatE)$.  
The main point is that both of these objects can be approximated by separated presheaves \Cref{separated-presheaves} and that these are easier to understand. 

\begin{definition}
Consider the functors 
\[\calV_\bbW\in \calP(\Perf\!,\CatexOE) \quad
	\text{ and } \quad
	\calV^\sch_\bbW\in \calP(\PSch\!,\CatexOE)\]
with 
\[\calV_\bbW(\Spa(R,R^+))\coloneqq\{\text{Finite projective modules over } \bbW R\}\]
and analogously
\[\calV^\sch_\bbW(\Spec R)\coloneqq\{\text{Finite projective modules over } \bbW R\}\,.\]
\end{definition}
From \cite[Corollary 17.1.9]{SW20} (applied to the case where $n=\infty$ and $R^\sharp=R$) it follows that $\calV_\bbW$ is a v-sheaf.
From \cite[Theorem 4.1.(ii)]{bhatt_scholze_projectivity_of_the_witt_vector_affine_grassmannian} it follows that $\calV^\sch_\bbW$ is a scheme-theoretic v-sheaf.

\begin{remark}
	\label{calV-is-a-vsheaf}
	It is also clear that $\calV_\bbW\simeq (\calV^\sch_\bbW)^{(\Diamond_\pre)}\simeq (\calV^\sch_\bbW)^{\Diamond}$. 
	We expect that the identity $\calV^\sch_\bbW\simeq (\calV_\bbW)^\red$ also holds. 
	When $E$ is of mixed-characteristic this latter identity is a result of G\"uthge \cite[$\mathsection$ 3]{Gut23}.
\end{remark}

\begin{definition}
	Consider functors 
	\[\calV_\bbW[\frac{1}{\pi}],\calV_\calY[\frac{1}{\pi}]\in \calP(\Perf\!, \CatE)	\quad
	\text{ and } \quad
	\calV^\sch_\bbW[\frac{1}{\pi}]\in\calP(\PSch\!,\CatE)\]	
	given by the formulas
	\[\calvwp\coloneqq\calV_\bbW\otimes_{O_E}E\text{, }  \calvyp\coloneqq\calV_\calY\otimes_{O_E}E \text{ and } \calvwps\coloneqq\calV^\sch_\bbW\otimes_{O_E}E\,.\]
\end{definition}

\begin{remark}
	\label{diamond-identification-Witt}
At the moment, we do not endow $\calvwp$, $\calvyp$ or $\calvwps$ with exact structure, but later we will identify these categories with other categories that carry a natural exact structure. 	
\end{remark}

\begin{lemma}
	\label{general-separated-nonsense}
	Let $\calT \in \{\Perf\!, \PSch\}$. Given $\calC\in \calS(\calT,\CatOE)$ be a v-sheaf of $O_E$-linear categories, let $\calD\coloneqq\calC\otimes_{O_E} E$. 
	Then $\calD\in \calP(\calT,\CatE)$ is a separated presheaf.
\end{lemma}
\begin{proof}
	By \Cref{pass-to-anima} and \Cref{check-on-categories}, it suffices to show that for every v-cover $[U\to X]\in \Perf$ the map
	\[\calD(X)\to \on{Desc.}(\calD,X\!/U)\]
	is fully faithful, or equivalently that for any two objects the presheaf of morphisms is a v-sheaf.
	By construction, objects in $\calD(X)$ agree with objects in $\calC(X)$. 
	Since objects in $\calC(X)$ are dualizable, the diagramatic characterization of dualizable objects implies that objects in $\calD(X)$ are also dualizable. 
	In particular, we may compute morphisms in terms of the $\otimes$-unit and internal Hom-objects.
	Indeed, for $V, W\in \calD(X)$ we have $\on{Hom}(V,W)\simeq \on{Hom}(\mathbb{1}\otimes V, W)\simeq \on{Hom}(\mathbb{1},V^\vee\otimes W)$.
	It remains to show that for all objects $V\in \calC(X)$ the presheaf of $E$-vector spaces 
	\[\calH\coloneqq\on{Hom}_{\calC\otimes_{O_E} E}(\mathbb{1},V)= \on{Hom}_{\calC}(\mathbb{1},V)\otimes_{O_E} E\]
	is a sheaf.
	Nevertheless, $\calH$ can be written as a sequential colimit of the form 
	\[\calH=\varinjlim \ [\on{Hom}_{\calC}(\mathbb{1},V)\xrightarrow{\cdot \pi} \on{Hom}_{\calC}(\mathbb{1},V)\xrightarrow{\cdot \pi} \dots ]\,.\]
	Since in the category of $O_E$-modules finite limits commute with filtered colimits, and $\on{Hom}_{\calC}(\mathbb{1},V)$ is a sheaf of $O_E$-modules, we can conclude that $\calH$ is also a sheaf of $E$-vector spaces. 
\end{proof}

\begin{corollary}
	\label{separatedness-of-isogeny}
	$\calV_\bbW[\frac{1}{\pi}]$, $\calV_\calY[\frac{1}{\pi}]$ and $\calV^\sch_\bbW[\frac{1}{\pi}]$ are separated presheaves.
\end{corollary}

\begin{definition}
Consider the functor
\[\calV^\sch_Y \in \calP(\PSch\!,\CatexE)
\text{  with }\]
\[\calV^\sch_Y(\Spec R)\coloneqq\{\text{Finite projective modules over } \bbW R[\frac{1}{\pi}]\}\,.\]
\end{definition}
\begin{proposition}
	\label{properties-schematic-V}
	The following statements hold:
	\begin{enumerate}
		\item $\calV^\sch_Y\simeq (\calV_Y)^\red$.
		\item $\calV^\sch_Y$ is an arc-sheaf. 
		\item If $S=\Spec R$ is a comb then $\calV^\sch_Y(S)\simeq \calvwps(S)$ in $\CatE$\,.
		\item The v-sheafification of $\calvwps$ is equivalent to $\calV^\sch_Y$ in $\calS(\PSch\!,\CatE)$\,. 
	\end{enumerate}
\end{proposition}
\begin{proof}
	The first statement follows from \Cref{theorem-evaluating-opensubs} and the definitions.
	By \cite[Proposition 5.9]{Ivanov_arc_descent} $\calV^\sch_Y$ is an arc-sheaf. 
	By \cite[Theorem 6.1]{Ivanov_arc_descent} the values of $\calvwps$ and $\calV^\sch_Y$ agree on combs. 
	Since combs form a basis for the scheme-theoretic v-topology (see \Cref{rem:scheme_v_covers}), the fourth claim follows.
\end{proof}

We see that $\calV_\bbW^\sch[\frac{1}{\pi}]$ has a natural immersion into $\calV^\sch_Y$ and acquires the structure of a separated presheaf with values in $\CatexE$ since it inherits an exact structure.

\begin{proposition}
	\label{Big-diamond-properties}
	We have the following identities:
	\begin{enumerate}
		\item $(\calV^\sch_Y)^{\Diamond_\pre}\simeq (\calV_Y)^{(\An_\pre)}$ and $\calV^\sch_\bbW[\frac{1}{\pi}]^{\Diamond_\pre}\simeq \calV_\bbW[\frac{1}{\pi}]$ in $\calP(\Perf\!,\CatexE)$\,. 
		\item $\calV_\bbW[\frac{1}{\pi}]^{\sh}\simeq \calV_\bbW^\sch[\frac{1}{\pi}]^\Diamond\simeq (\calV^\sch_Y)^{\Diamond} \simeq \calV_Y^{\An}$ in $\calS(\Perf\!,\CatE)$\,. 
		
	\end{enumerate}
\end{proposition}
\begin{proof}
	The first claim follows directly from \Cref{theorem-evaluating-opensubs} and the definitions. 
	The second claim follows from \Cref{properties-schematic-V} and \Cref{combs_and-products}.
\end{proof}

\begin{corollary}
	\label{separated-approx-VW}
	The presheaf $\calV_\bbW[\frac{1}{\pi}]$ (with the exact structure inherited from $(\calV^\sch_Y)^{\Diamond_\pre}$) is separated and sheafifies to $\calV_Y^{\An}$.
\end{corollary}
\begin{proof}
This follows from \Cref{check-on-categories}.	
\end{proof}

We now move on to rewrite $(\calV_Y)^\mer$.

\begin{proposition}
	\label{Lattice-to-Vyp}
	We have a Cartesian square in $\calP(\Perf\!,\CatE)$
	\begin{center}
	\begin{tikzcd}
		\calV_\calY \ar{r} \ar{d} &		\calV_\calY[\frac{1}{\pi}] \ar{d} \\
		\calV_\bbW \ar{r} & \calV_\bbW[\frac{1}{\pi}]\,.
	\end{tikzcd}
	\end{center}
\end{proposition}
\begin{proof} 
	The argument is a standard application of Beauville--Laszlo descent \cite[Lemma 5.2.9]{SW20}.
	We provide the details for the convenience of the reader.
	Fix $S\in \Perf$. By analytic descent, there is a Cartesian diagram
	\begin{center}
	\begin{tikzcd}
		\calV_\calY(S) \arrow{r} \arrow{d}  & \calV(Y_{[1,\infty),S})  \arrow{d} \\
	\calV(Y_{[0,1],S})	\arrow{r} & \calV(Y_{[1,1],S})\,.
	\end{tikzcd}
	\end{center}
	These spaces correspond to the loci $Y_{[0,1],S}=\{|\pi|\leq |[\varpi]|\neq 0\}$ and $Y_{[1,\infty),S}=\{|[\varpi]|\leq |\pi|\neq 0\}$. 
	Moreover, $\calV(Y_{[0,1],S})$ is the category of finite projective modules over $B_{[0,1],S}$. 
	Since $\pi$ is already invertible in $Y_{[1,\infty),S}$ and $Y_{[1,1],S}$,
	this formally leads to the following commutative diagram with Cartesian squares

	\begin{center}
	\begin{tikzcd}
		\calV_\calY \ar{r} \ar{d}	& \calV_\calY[\frac{1}{\pi}](S)\arrow{r} \arrow{d}  & \calV(Y_{[1,\infty),S}) \arrow{d} \\
		(\on{Proj} B_{[0,1],S}) \ar{r}	  & (\on{Proj} B_{[0,1],S})\otimes_{O_E} E	\arrow{r} & \calV(Y_{[1,1],S})\,.
	\end{tikzcd}
	\end{center}

	Furthermore, $\calV_\bbW(S)$ is equivalent to $\on{Proj}(\bbW R)$. 
	Since we have the identity of rings $\bbW R=\widehat{(B_{[0,1]})}_\pi$, we have the following Cartesian diagram by \cite[Lemma 5.2.9]{SW20}.

	\begin{center}
	\begin{tikzcd}
		\on{Proj} B_{[0,1],S}	\arrow{r} \arrow{d}  &  \arrow{d} \calV_\bbW(S) \\
	 (\on{Proj} B_{[0,1],S})\otimes_{O_E} E	\arrow{r} & \calV_\bbW[\frac{1}{\pi}](S)\,.
	\end{tikzcd}
	\end{center}
	This implies that the commutative diagram below is also Cartesian 

	\begin{center}
	\begin{tikzcd}
		\calV_\calY(S) \ar{r}\ar{d} & \on{Proj} B_{[0,1],S}	\arrow{r} \arrow{d}  &  \arrow{d} \calV_\bbW(S) \\
		\calV_\calY[\frac{1}{\pi}](S) \ar{r} &  (\on{Proj} B_{[0,1],S})\otimes_{O_E} E	\arrow{r} & \calV_\bbW[\frac{1}{\pi}](S)\,.
	\end{tikzcd}
	\end{center}
\end{proof}

\begin{proposition}
	\label{Cartesian-YVp}
	The diagrams 
	\begin{center}
	\begin{tikzcd}
		\calV_\calY \ar{r} \ar{d} &		\calV_\calY[\frac{1}{\pi}]	\arrow{r} \arrow{d}  & (\calV_Y)^{\mer_\pre} \arrow{d} \\
		\calV_\bbW \ar{r} & \calV_\bbW[\frac{1}{\pi}] \arrow{r} & (\calV_Y)^{(\An_\pre)}
	\end{tikzcd}
	\end{center}
	are Cartesian in $\calP(\Perf\!,\CatOE)$.
	Moreover, the right square is Cartesian in $\calP(\Perf\!,\CatE)$ and the outer square is Cartesian in $\calP(\Perf\!,\CatexOE)$.
\end{proposition}
\begin{proof}
	Let $S=\Spa(R,R^+)\in \Perf$ and let $T=\Spd(R_\disc,R^+_\disc)$.
	Pick a pseudo-uniformizer $\varpi\in R^+$. 
	As defined in \Cref{sec:notation}, we denote the two different topologies on the ring of $O_E$-Witt vectors by $\bbW(R^+)$ and $\bbA_{\on{inf}}(R^+)$.   

	The category $\calV_\calY(S)$ is the category of vector bundles over $\Spa(\bbA_{\on{inf}}(R^+))_{\{[\varpi]\neq 0\}}$. 
	Arguing as in \Cref{Lattice-to-Vyp}, we get the following Cartesian diagram  

	\begin{center}
	\begin{tikzcd}
		\calV_\calY[\frac{1}{\pi}](S)\arrow{r} \arrow{d}  & (\on{Proj} B_{[0,1],S})\otimes_{O_E} E \arrow{d} \\
		\calV(Y_{[1,\infty),S})\arrow{r} & \calV(Y_{[1,1],S})\,.
	\end{tikzcd}
	\end{center}

	On the other hand, $(\calV_Y)^{(\mer_\pre)}(S)\simeq \calV(Y_T)$. 
	This is the category of vector bundles on $\Spa \bbW(R^+)_{\{\pi\cdot [\varpi]\neq 0\}}$. 
	This space is the union of the loci $\{|[\varpi]|\leq |\pi|\neq 0\}$ and $\{|\pi|\leq |[\varpi]|\neq 0\}$. 
	The former agrees with $Y_{[1,\infty),S}$ while the latter is affinoid of the form $\Spa (B^{\disc}_{[0,1]}[\frac{1}{\pi}],B^{\disc,+}_{[0,1]}[\frac{1}{\pi}])$.
	By analytic descent, we have a pullback diagram
	\begin{center}
	\begin{tikzcd}
		\calV(Y_T)\arrow{r} \arrow{d}  & (\on{Proj} B^\disc_{[0,1],S}[\frac{1}{\pi}]) \arrow{d} \\
		\calV(Y_{[1,\infty),S})\arrow{r} & \calV(Y_{[1,1],S})\,.
	\end{tikzcd}
	\end{center}

	Note that the natural map $B^{\on{disc}}_{[0,1]}\to B_{[0,1]}$ is a continuous isomorphism of rings (which is not a homeomorphism!), see \cite[Lemma 14.3.1]{SW20}. 
	Hence, we have an identification $(\on{Proj} B^\disc_{[0,1],S})\simeq (\on{Proj} B_{[0,1],S})$ in $\CatOE$.
	Together with \Cref{Lattice-to-Vyp} and Beauville--Laszlo descent (see \cite[Lemma 5.2.9]{SW20}), this leads to the following commutative diagram with Cartesian squares

	\begin{center}
	\begin{tikzcd}
		\calV_\calY(S)\arrow{r} \arrow{d} & 	\calV_\calY[\frac{1}{\pi}](S) \arrow{r} \arrow{d} & \calV(Y_T)\arrow{r} \arrow{d}  &  \calV(Y_{[1,\infty),S}) \arrow{d} \\
		(\on{Proj} B_{[0,1],S}) \ar{r} \ar{d} &  (\on{Proj} B_{[0,1],S})\otimes_{O_E} E \arrow{r}\ar{d}	& (\on{Proj} B^\disc_{[0,1],S}[\frac{1}{\pi}])	\arrow{r} \ar{d} & \calV(Y_{[1,1],S}) \\
		(\on{Proj} \bbW(R)) \ar{r} & (\on{Proj} \bbW(R))\otimes_{O_E} E \ar{r} & (\on{Proj} \bbW(R)[\frac{1}{\pi}])\,.   
	\end{tikzcd}
	\end{center}
	It follows from the definitions that $(\on{Proj} \bbW(R))\simeq \calV_\bbW(S)$, $(\on{Proj} \bbW(R))\otimes_{O_E} E\simeq   \calV_\bbW(S)[\frac{1}{\pi}]$, and $(\on{Proj} \bbW(R)[\frac{1}{\pi}])\simeq (\calV_Y)^{(\An_\pre)}$ which allow us to conclude.
\end{proof}

\begin{corollary}
	\label{corollary-add-a-lattice}
	The square 
	\begin{center}
	\begin{tikzcd}
		\calV_\calY \ar{r} \ar{d} & (\calV_Y)^\mer \arrow{d} \\
		\calV_\bbW \ar{r}  & (\calV_Y)^\An
	\end{tikzcd}
	\end{center}
	is Cartesian in $\calS(\Perf\!,\CatexOE)$ .
\end{corollary}

\begin{corollary}
	\label{VYp-is-separated-presheaf}
	The following statements hold:
	\begin{enumerate}
		\item The natural map $\calV_\calY[\frac{1}{\pi}]\to (\calV_Y)^{(\mer_\pre)}$ is fully-faithful.
		\item The v-sheafification of $\calV_\calY[\frac{1}{\pi}]$ is $(\calV_Y)^\mer$ .
	\end{enumerate}
\end{corollary}

\begin{proof}
	The first statement follows from the proof of \Cref{Cartesian-YVp} and from the fact that 
	\[(\on{Proj} \bbW(R))\otimes_{O_E} E \to (\on{Proj} \bbW(R)[\frac{1}{\pi}])\]
	is fully-faithful.

	To show $\calV_\calY[\frac{1}{\pi}]^\sh\simeq \calV_Y^\mer$ it suffices by \Cref{remark-product-of-points} to show that $\calV_\calY[\frac{1}{\pi}](S)\simeq (\calV_Y)^{(\mer_\pre)}(S)$ when $S$ is a product of points.
	If $\Spa(R,R^+)$ is a product of points, $(\on{Proj} \bbW(R))\otimes_{O_E} E \simeq (\on{Proj} \bbW(R)[\frac{1}{\pi}])$ by \cite[Theorem 6.1]{Ivanov_arc_descent}.
	Using \Cref{Cartesian-YVp}, we can conclude.
\end{proof}

\begin{corollary}
	\label{separated-aprrox-Vy}
	The presheaf $\calV_\calY[\frac{1}{\pi}]$ (with the exact structure inherited from $(\calV_Y)^{\mer_\pre}$) is separated and sheafifies to $\calV_Y^{\An}$.
\end{corollary}
\begin{proof}
This follows from \Cref{check-on-categories}.	
\end{proof}

\begin{remark}
	\label{abstract-categories}
	Although for $S\in \Perf$ the categories $\calV_Y^\An(S)\in \CatexE$ and $\calV_Y^\mer(S)\in \CatexE$ together with their exact structure are fairly abstract, one can still have some formal control over them by applying \Cref{Describing-separated-presheaves} to $\calV_\bbW[\frac{1}{\pi}]$ and $\calV_\calY[\frac{1}{\pi}]$. Indeed, these presheaves are separated by \Cref{separated-aprrox-Vy} and \Cref{separated-approx-VW}. 
\end{remark}

\section{Meromorphic vector bundles on the Fargues--Fontaine curve}
In this section we study $\Bun_\FF^\mer$ and $\Bun_\FF^\An$. 
We note that their structure as objects in $\calS(\Perf\!,\CatE)$ is easy to deduce from our study of $\calV^\mer_\calY$ and $\calV_\calY^\An$ simply by adding Frobenius structure.
Nevertheless, as we will see, the Frobenius structure allows us to further understand their behavior as objects in $\calS(\Perf\!,\CatexE)$.
\subsection{Dieudonn\'e modules and Isocrystals}
For this subsection we fix a test object $\calS=\Spec A \in \PSch$.
    \begin{definition}
	    \label{def:DM_isocrystals}
	    A \textit{Dieudonn\'e module} over $\calS$ is a pair $(\calE,\Phi_\calE)$, where $\calE$ is a finite projective module over $\bbW(A)$ and $\Phi_\calE$ is an isomorphism 
		\[\Phi_\calE\colon \varphi^*\calE_{Y_{\calS^\diamond}}\rightarrow  \calE_{Y_{\calS^\diamond}}\,.\]
	    An \textit{isocrystal} over $\calS$ is a pair $(\calF,\Phi_\calF)$, where $\calF$ is a vector bundle over $Y_{\calS^\diamond}$ and $\Phi_\calF$ is an isomorphism 
		\[\Phi_\calF\colon \varphi^*\calF \rightarrow  \calF\,.\]
	    Morphisms of these data are $\varphi$-equivariant maps.
	    We declare a sequence of morphisms to be exact if the underlying sequence of projective modules (resp.\ vector bundles) is exact.
	    We denote these categories by $\DM(\calS)\in \CatexOE$ and $\IC(\calS)\in \CatexE$.
	    These rules organize into functors $\DM\in \calP(\PSch\!,\CatexOE)$ and $\IC\in \calP(\PSch\!,\CatexE)$.
    \end{definition}
    \begin{remark}
		Dieudonn\'e modules are also considered by Pappas--Rapoport under the name \textit{meromorphic Frobenius crystals} \cite[Definition 2.3.6]{PR21}.
    \end{remark}

    \begin{proposition}
	    \label{Cartesian-expression-for-DM-and-IC}
	    We have Cartesian diagrams 
	\begin{center}
	\begin{tikzcd}
		\DM \ar{r} \ar{d} & \calV^\sch_\bbW[\frac{1}{\pi}] \ar{d}{\Delta} \ar{r} & \calV^\sch_Y \arrow{d}{\Delta} & 
		\IC \ar{r} \ar{d} & \calV^\sch_Y \arrow{d}{\Delta} \\
		\calV^\sch_\bbW\ar{r}{(\on{id},\varphi^*)}  & \calV^\sch_\bbW[\frac{1}{\pi}] \times \calV^\sch_\bbW[\frac{1}{\pi}] \ar{r} & \calV^\sch_Y\times \calV^\sch_Y &
		\calV^\sch_Y\ar{r}{(\on{id},\varphi^*)}  & \calV^\sch_Y\times \calV^\sch_Y
	\end{tikzcd}
	\end{center}
	in $\calP(\PSch\!, \CatexOE)$ resp.\ $\calP(\PSch\!, \CatexE)$.
    \end{proposition}
    \begin{proof}
	    These are formal reinterpretations of \Cref{def:DM_isocrystals}. 
    \end{proof}

    \begin{definition}
	    \label{latticed-isocrystals}
	    We define $\DM[\frac{1}{\pi}]\in \calP(\PSch\!,\CatexE)$ as the fiber product
	\begin{center}
	\begin{tikzcd}
		\DM[\frac{1}{\pi}] \ar{r} \ar{d} & \calV^\sch_\bbW[\frac{1}{\pi}] \ar{d}{\Delta} \\
		\calV^\sch_\bbW[\frac{1}{\pi}]\ar{r}{(\on{id},\varphi^*)}  & \calV^\sch_\bbW[\frac{1}{\pi}] \times \calV^\sch_\bbW[\frac{1}{\pi}]\,. 
	\end{tikzcd}
	\end{center}
    \end{definition}

    One can think of $\DM[\frac{1}{\pi}](\calS)$ as the category of isocrystals over $\calS$ that admit a lattice.

    \begin{proposition}
	    \label{formal-properties-IC-DM}
The following statements hold.
\begin{enumerate}
	\item $\IC\in \calS(\PSch\!,\CatexE)$ and $\DM\in \calS(\PSch\!,\CatexOE)$.
	\item $\DM[\frac{1}{\pi}]\in \calP(\PSch\!, \CatexOE)$ is v-separated and $\DM[\frac{1}{\pi}]^\sh\simeq\IC$.
	\item The following diagram in $\calP(\PSch\!, \CatexOE)$ has Cartesian squares
	\begin{center}
	\begin{tikzcd}
		\DM \ar{r} \ar{d} & \DM[\frac{1}{\pi}] \ar{r}\ar{d} &  \IC  \arrow{d} \\
		\calV_\bbW^\sch \ar{r} & \calV_\bbW^\sch[\frac{1}{\pi}] \ar{r}  & \calV^\sch_Y\,.
	\end{tikzcd}
	\end{center}
\end{enumerate}
    \end{proposition}
    \begin{proof}
	    We note that the property of being a separated presheaf (resp.\ a sheaf) is stable under finite limits of presheaves by \Cref{separated-presheaves-finitelimits}. 
		Then, the first two claims are implied by \Cref{properties-schematic-V} together with the Cartesian diagram \ref{Diag_BunFF} and \Cref{Cartesian-expression-for-DM-and-IC}. 
		The last claim follows formally from \Cref{latticed-isocrystals}, \Cref{Cartesian-expression-for-DM-and-IC} and the fact that $\calV^\sch_\bbW[\frac{1}{\pi}]$ is separated with $\calV^\sch[\frac{1}{\pi}]^\sh\simeq \calV^\sch_Y$
    \end{proof}

    \begin{proposition}
	    \label{exactness-of-Dieudonne-mods-at-points}
	    Let $\Sigma\coloneqq[\calE_1\to \calE_2\to \calE_3]$ be a sequence in $\DM(\calS)$. 
	    Then $\Sigma$ is exact if and only if for every geometric point $\bar{x}\to \calS$, the sequence $\Sigma_{\bar{x}}$ is exact.
	    If $\Sigma$ is already a complex then it suffices to check exactness on geometric points with closed image in $\calS$.
    \end{proposition}
    \begin{proof}
	    Since $\calV_\bbW^\sch$ is a Zariski sheaf and since a basis over $\Spec A$ always deforms to a basis over $\Spec \bbW A$ (since $\bbW A$ is $\pi$-adically complete, an ideal generated by elements $a_1,...,a_n$ in $A=\bbW(A)/\pi$ is the unit ideal iff the ideal generated by $[a_1],...,[a_n]$ in $\bbW A$ is the unit ideal), we may assume that each $\calE_i$ is free and of constant finite rank.
	    The maps $\calE_i\to \calE_j$ are now given by matrices with values in $\bbW A=A^\bbN$ (as sets).
    	The map $A\to \prod_{\bar{x}\to \calS}\! K(\bar x)$ is injective by \Cref{geompts_inj}, so one can check on geometric points that the sequence is a complex.
	    
    Once we know the sequence is a complex, exactness can be checked on closed points of $\Spec \bbW A$.
    As we vary over geometric points $\overline{x}\to \calS$ with closed image, the induced family of maps $\bbW(\overline{x})\colon \Spec \bbW C\to \Spec \bbW A$ covers all closed points of $\Spec \bbW A$. This allows us to conclude. \qedhere
    \end{proof}

    To prove an analogue of \Cref{exactness-of-Dieudonne-mods-at-points} for $\frakB$, we first give a reinterpretation.

    \begin{proposition}
	    \label{isomorphism-isocrystals-bung}
	    The following statements hold.
	    \begin{enumerate}
		    \item For all $\calS\in \PSch$, $\on{Bun}_{\on{FF}}(\calS^\diamond)\simeq \IC(\calS)$ in $\CatexE$\,.
		    \item $\IC\simeq (\on{Bun}_{\on{FF}})^\red$ in $\calS(\PSch\!,\CatexE)$.
	    \end{enumerate}
    \end{proposition}
    \begin{proof}
	    Both claims follow by applying \Cref{properties-schematic-V} and \Cref{Cartesian-expression-for-DM-and-IC} in combination with the Cartesian diagram \ref{Diag_BunFF}. \qedhere
    \end{proof}

    \begin{remark}
	    The result of \Cref{isomorphism-isocrystals-bung} is implicitly proved during the proof of \cite[Theorem 2.3.8]{PR21} when $E$ is of mixed characteristic. 
	    Their approach relied on Sen theory which is a tool that is only available in mixed characteristic.
	    A previous version of this article also relied on Sen theory and for this reason we also had to restrict to the mixed characteristic setup.
    \end{remark}

    \begin{proposition}
	    \label{check-exact-isoc-points}
	    Let $\calS=\Spec R$ and $\Sigma\coloneqq[\calE_1\to \calE_2\to \calE_3] \in \IC(\calS)$ be a sequence such that the underlying projective modules $\calE_i$ over $\bbW R[\frac{1}{\pi}]$ have constant rank $\on{rk.}(\calE_i)=r_i$ and $r_1+r_3=r_2$. The following statements hold.
	    \begin{enumerate}
		    \item The sequence is exact if and only if for every geometric point $\overline{x}\to \calS$ the sequence $\Sigma_{\bar{x}} \in \IC(\overline{x})$ is exact.
		    \item Moreover, if the sequence is already assumed to be a complex, then exactness can be checked on geometric points $\overline{x}\to \calS$ whose image is a closed point.
	    	
	    \end{enumerate}
    \end{proposition}
    \begin{proof}
	    The forward implication of part (1) is evident.
	    We reduce the backward implication of part (1) to the statement in part (2) as follows.
	    Assume that for every geometric point of $\calS$, the sequence is exact.
	    By \Cref{formal-properties-IC-DM} we may test exactness v-locally. 
	    Thus, we may assume that $\calS=\Spec R$ is a comb and by \cite[Theorem 6.1]{Ivanov_arc_descent} that all the underlying projective modules are free. 
	    We write $M_1\in \on{M}_{r_2\times r_1}(\bbW(R)[\frac{1}{\pi}])$ and $M_2\in \on{M}_{r_3\times r_2}(\bbW(R)[\frac{1}{\pi}])$ the matrices representing the maps $\calE_1\to \calE_2$ and $\calE_2\to \calE_3$, respectively. 
	    The induced map $\calE_1\to \calE_3$ is the $0$ map if and only if the matrix $M_2\cdot M_1=0$. 
	    This can be tested on geometric points since $R$ is perfect and in particular reduced. 
	    We have shown that $\Sigma$ is a complex whenever each $\Sigma_{\overline{x}}$ is exact, so it suffices to prove (2).
	    To show (2), we keep the same setup as above.
	    With this setup, exactness can be expressed in terms of the rank of $M_1$ and $M_2$ at the different points of $\Spa \bbW(R)[\frac{1}{\pi}]$. 

	    The locus where $M_1$ has rank strictly smaller to $r_1$ is a Zariski closed subset (cut out by the minors of $M_1$) $Z\subseteq \Spa \bbW(R)[\frac{1}{\pi}]$. 
	    Moreover, since the map $\calE_1\to \calE_2$ is $\varphi$-equivariant, we have $\varphi(Z)=Z$. 
	    Indeed, the rank of $M_1$ equals the rank of $\varphi(M_1)$.

	    Suppose $Z\neq \emptyset$ and let $z\in Z$. 
	    Endow $R$ with the discrete topology and consider the projection map $f\colon \Spd \bbW(R)[\frac{1}{\pi}]\to \Spd (R)$.
	    By the classification of points in the olivine spectrum \cite[Defintion 2.1]{Gle22} and \cite[Proposition 2.14]{Gle22} either $f(z)$ is $d$-analytic or it is discrete, see \cite[Definition 2.2]{Gle22}. 
	    In the following, we will argue that if $Z\neq \emptyset$ then there is $z\in Z$ such that $f(z)$ is algebraic in the sense of \cite[Definition 2.2.(1)]{Gle22}.
	    Consider the quotient $g\colon \Spd \bbW(R)[\frac{1}{\pi}]\to \Spd \bbW(R)[\frac{1}{\pi}]/\varphi$. By $\varphi$-invariance, $Z$ has the form $g^{-1}(Z')$ for a closed subset $Z'\subseteq |\Spd \bbW(R)[\frac{1}{\pi}]/\varphi|$. 
	    Recall from \cite[Definition II.1.19]{FS24} the diamond $\on{Div}^1_E$, which parametrizes degree $1$ divisors on the Fargues--Fontaine curve.
	    We have an identification at the level of topological spaces 
	    \[\alpha\colon |\Spd \bbW(R)[\frac{1}{\pi}]/\varphi|\cong| \Spd(R)\times (\on{Div}^1_E)|\] which fits in the following commutative diagram. 
	    \begin{center}
	    \begin{tikzcd}
		    \mid \Spd R\times \Spd E \mid \arrow{r} \arrow{d} \arrow[bend right=75, "\mid f \mid"']{dd}  & \mid (\Spd R\times \Spd E)/\on{Frob}_R\times \on{id}_E \mid  \arrow{dd} \ar{ld}{\alpha}\\
		    \mid \Spd R\times \on{Div}^1_E \mid  \arrow{d}  &  \\
		    \mid \Spd R \mid \arrow{r}{\simeq} & \mid \Spd R/{\on{Frob}_R}\mid\,. 
	    \end{tikzcd}
	    \end{center}

	    This shows that $f(Z)$ is the image of $Z'$ under the projection map $|\Spd R\times \on{Div}^1_E|\to |\Spd R|$.

	    By \cite[Proposition II.1.21]{FS24}, this image is a closed subset of $\Spd R$.
	    Furthermore, since $Z$ is stable under vertical generization, the same is true about $f(Z)$.
	    If $z\in f(Z)$ and it is discrete, then its largest vertical generization is already algebraic.
	    Suppose instead that $f(z)$ is a $d$-analytic point.  
	    Recall that every $d$-analytic point in $\Spd(R,R)$ is formal in the sense of \cite[Definition 2.2.(4)]{Gle22}, and that formal points have a unique formal specialization \cite[Proposition 2.9.(2)]{Gle22}. 
	    Since $f(Z)\subseteq \Spd(R,R)$ is closed, it must contain the formal specialization of $f(z)$. 
	    In particular, we see that $f(Z)$ has a discrete point and by the above an algebraic one as we wanted to show. \\

	    We showed above that $f(Z)$ contains an algebraic point.
	    We fix $x\in f(Z)$ that is algebraic, in particular there is a unique $y\in \Spec(R)$ inducing $x$. 
	    If $k(y)$ is the residue field of $y$ we obtain a map $\Spd k(y)\to \Spd R$, and the only point of $|\Spd R|$ in the image is $x$.
	    In this case $f^{-1}(f(x))=\Spd k(y)\times \Spd E=\Spd \bbW(k(y))[\frac{1}{\pi}]$, which consists of one point.
	    In particular, $\Spd \bbW(k(y))[\frac{1}{\pi}]\cap Z\neq \emptyset$ implies $\Spa \bbW(k(y))[\frac{1}{\pi}]\subseteq Z$.  
	    Since $\bbW(k(y))[\frac{1}{\pi}]$ is a field, this shows that every $r_1$-minor in $M_1$ thought of as an element in $(R^\bbZ)^{r_2\cdot r_1}\supseteq \on{M}_{r_2\times r_1}( \bbW(R)[\frac{1}{\pi}])$ vanishes identically when restricted to $k(y)^\bbZ$. 
	    The same must be true for every point in the Zariski closure of $\Spec(k(y))\subseteq \Spec(R)$.
	    In particular, we have found a closed point $\overline{x}\to \Spec(R)$ for which $\calE_{1,\overline{x}}\to \calE_{2,\overline{x}}$ is not injective. 
	    This contradicts our assumption, so $Z=\emptyset$.
	    A similar argument proves that $\calE_2\to \calE_3$ is surjective and by rank considerations the sequence is also exact in the middle.
    \end{proof}

    We use the $\Diamond$-functor to consider the categories of Dieudonn\'e modules and of isocrystals as analytic objects.

    \begin{definition}
	    If $S=\Spa(R,R^+)$ we let $\DMan(S)\coloneqq(\DM)^{\lozenge_{\on{pre}}}(S)=\DM(\Spec R)$. 
	    This rule defines a functor $\DMan\in \calP(\Perf\!,\CatexOE)$.
    \end{definition}
    \begin{proposition}
	    \label{Passing_DM_to_diamondworld}
	    The following statements hold.
	    \begin{enumerate}
		    \item $\DMan\in \calS(\Perf\!,\CatexOE)$.
		    \item We have a commutative diagram in $\calP(\Perf\!, \CatOE)$
	\begin{center}
	\begin{tikzcd}
		\DMan \ar{r} \ar{d} & \calV_\bbW[\frac{1}{\pi}] \arrow{d}{\Delta} \ar{r} & (\calV_Y)^\An \arrow{d}{\Delta} \ar{d}   \\
		\calV_\bbW\ar{r}{(\on{id},\varphi^*)}  & \calV_\bbW[\frac{1}{\pi}] \ar{r} \times  \calV_\bbW[\frac{1}{\pi}]  & (\calV_Y)^\An\times (\calV_Y)^\An
	\end{tikzcd}
	\end{center}
	with Cartesian squares.
	    \end{enumerate}
    \end{proposition}
    \begin{proof}
	    This formally follows from \Cref{Cartesian-expression-for-DM-and-IC}, \Cref{Big-diamond-properties}, \Cref{separatedness-of-isogeny} and \Cref{calV-is-a-vsheaf}.
 %
    \end{proof}

    \begin{definition}
	    Let $S=\Spa(R,R^+)\in \Perf$.
	    We let $\DMan[\frac{1}{\pi}](S)\coloneqq\DM[\frac{1}{\pi}](\Spec R)$ and $\IC^{\lozenge_{\on{pre}}}(S)=\IC(\Spec R)$.
	    These rules organize into objects $\DMan[\frac{1}{\pi}]\in \calP(\Perf\!,\CatE)$ and $\IC^{\Diamond_\pre}\in \calP(\Perf\!,\CatexE)$.
 We let $\IC^\lozenge$ be the $v$-sheafification of $\IC^{\lozenge_{\on{pre}}}$.
	    We call objects of $\IC^\lozenge(S)$ \textit{analytic isocrystals} over $S$.
    \end{definition}

    \begin{remark}
	    It follows from \Cref{general-functoriality} and \Cref{isomorphism-isocrystals-bung} that $(\Bun_{\on{FF}})^\An\simeq \IC^\Diamond$.
    \end{remark}

    \begin{proposition}
	    \label{DM1p-is-well-bheaved-subcategory}
	    Let $S=\Spa (R,R^+) \in \Perf$, the following hold:
	    \begin{enumerate}
		    \item $\DMan[\frac{1}{\pi}]$ is a v-separated presheaf in $\calP(\Perf\!,\CatE)$.
		    \item The v-sheafification of $\DMan[\frac{1}{\pi}]$ is equivalent to $\IC^{\lozenge}$ in $\calS(\Perf\!,\CatE)$.
		    \item A sequence in $\DMan[\frac{1}{\pi}]$ is exact in $\IC^\Diamond$ if and only if it is exact in $\IC^{\Diamond_\pre}$. 
	    \end{enumerate}
    \end{proposition}
    \begin{proof}
	    The first two claims follow formally from \Cref{separatedness-of-isogeny} and \Cref{Big-diamond-properties}.

	    For the third claim let $\Sigma\coloneqq[\calE_1\to \calE_2\to \calE_3]\in \DMan[\frac{1}{\pi}](S)$ be a sequence i.e.\ $\Sigma$ is a sequence in $\DM[\frac{1}{\pi}](\Spec R)$.
	    Since $\IC^{\Diamond_\pre}(S)\to \IC^\Diamond(S)$ is a map in $\CatexE$, it is clear that if $\Sigma$ is exact in $\IC^{\Diamond_\pre}$ then it is also exact in $\IC^\Diamond$.
	    Assume that $\Sigma$ is exact in $\IC^\lozenge(S)$.
	    By definition, this means that $\Sigma$ is exact in $\IC(\Spec R')$ for a v-cover $\Spa(R',R'^+)\to \Spa(R,R^+)$.
	    Since $\Sht_\bbW[\frac{1}{\pi}]\to \IC^\Diamond$ is fully-faithful on $(R,R^+)$-points, we deduce that the sequence is a complex.
	    By the second part of \Cref{check-exact-isoc-points}, we can check exactness on closed points of $\Spec R$. 
	    Since $\Spa(R',R'^+)\to \Spa(R,R^+)$ is a v-cover, every closed point of $\Spec R$ is in the image of $\Spec R'\to \Spec R$.
	    Indeed, closed points of $\Spec R$ support continuous valuations each of which will lift to a continuous valuation of $R'$. 
	    We have shown that if a sequence becomes exact in $\IC(\Spec R')$, then it was already exact in $\IC(\Spec R)$.
    \end{proof}

    \begin{proposition}
	    \label{check-exact-isoc-points_2}
	    Let $S=\Spa(R,R^+)$.
	    Let $\Sigma=[\calE_1\to \calE_2\to \calE_3]$ be a sequence with $\calE_i\in \IC^\lozenge(S)$, and each of constant rank $\on{rk.}(\calE_i)=r_i$ such that $r_1+r_3=r_2$.
	    The sequence is exact if and only if for every geometric point $\overline{x}\to S$ the sequence $\overline{x}^*\Sigma$ with $\overline{x}^*\calE_i \in \IC^\lozenge(\overline{x})$ is exact.
    \end{proposition}
    \begin{proof}
	    The forward implication is evident.
	    Assume that for every geometric point of $S$ the sequence is exact.
	    By the definition of the exact structure on $\IC^\lozenge(S)$ via sheafification, we may test exactness v-locally. 
	    More precisely, it suffices to find a v-cover $f\colon S'\to S$ with $S'=\Spa(R',R'^+)$ such that each $\calE_i\in \DMan[\frac{1}{\pi}](S')$ and such that $g^*\Sigma$ is exact in $\IC^{\Diamond_\pre}(S')$.
	    Without loss of generality, $S=S'$ and $\calE_i\in \IC^{\Diamond_\pre}(S)$.
	    Since the map $R\to \prod_{x\in \Spa(R,R^+)}\! k(x)$ is injective by \Cref{geompts_inj}, we can test on geometric points if the map is a complex.  
	    Once we know it is a complex, by \Cref{check-exact-isoc-points} we can test exactness on closed points of $\Spec R$.
	    As every closed point of $\Spec R$ supports a geometric point of $\Spa(R,R^+)$, the proof is complete.
    \end{proof}

\subsection{Shtukas, isoshtukas and meromorphic vector bundles}

    \begin{definition}
	    \label{defishtukas}
	Let $S=\Spa(R,R^+)\in \Perf$.
	A \textit{crystalline shtuka} over $S$ is a pair $(\calE,\Phi_\calE)$, where $\calE$ is a vector bundle over $\calY_S$ and $\Phi_\calE$ is an isomorphism \[\Phi_\calE\colon (\varphi^*\calE)_{Y_S}\rightarrow  \calE_{Y_S}\]
	that is meromorphic (cf. \cite[Definition 5.3.5]{SW20}) along $\pi=0$. 
	    Morphisms of these data are $\varphi$-equivariant maps.
	    We declare a sequence of morphisms to be exact if the underlying maps of vector bundles form an exact sequence.
	    We denote this category by $\SHT(S)\in \CatexOE$.
	    This induces a functor $\SHT\in \calP(\PSch\!,\CatexOE)$. 
    \end{definition}

 %
    \begin{proposition}
	    \label{Berkeley-notes-shtukas-is-v-sheaf}
	    The following statements hold.
	    \begin{enumerate}
		    \item $\SHT\in \calS(\Perf\!,\CatexOE)$.
		    \item We have the following commutative diagram in $\calP(\Perf\!, \CatOE)$ with Cartesian squares:
	\begin{center}
	\begin{tikzcd}
		\SHT \ar{r} \ar{d} & \calV_\calY[\frac{1}{\pi}] \arrow{d}{\Delta} \ar{r} & (\calV_Y)^\mer \arrow{d}{\Delta} \ar{d}   \\
		\calV_\calY\ar{r}{(\on{id},\varphi^*)}  & \calV_\calY[\frac{1}{\pi}] \ar{r} \times  \calV_\calY[\frac{1}{\pi}]  & (\calV_Y)^\mer\times (\calV_Y)^\mer
	\end{tikzcd}
	\end{center}
	    \end{enumerate}
    \end{proposition}
    \begin{proof}
	    That the left-hand square is Cartesian is a reinterpretation of \Cref{defishtukas}. 
	    That the right-hand square is Cartesian follows formally from \Cref{VYp-is-separated-presheaf} and \Cref{separatedness-of-isogeny}.
    \end{proof}

    \begin{definition}
	    We define $\SHT[\frac{1}{\pi}]\in \calP(\Perf\!,\CatE)$ as the fiber product
	\begin{center}
	\begin{tikzcd}
		\SHT[\frac{1}{\pi}] \ar{r} \ar{d} & \calV_\calY[\frac{1}{\pi}] \ar{d}{\Delta} \\
		\calV_\calY[\frac{1}{\pi}]\ar{r}{(\on{id},\varphi^*)}  & \calV_\calY[\frac{1}{\pi}] \times \calV_\calY[\frac{1}{\pi}]\,. 
	\end{tikzcd}
	\end{center}
	in $\calP(\Perf\!,\CatE)$.
	We call objects of $\SHT[\frac{1}{\pi}](S)$ \textit{isoshtukas} over $S$. 
    \end{definition}

    \begin{definition}
	    \label{meromorphic-vector-bundles}
	    We let $\Bun_\FF^\mer\coloneqq(\Bun_\FF)^\mer$ as an object in $\calS(\Perf\!,\CatexE)$. 
	    For $S\in \Perf$ we call $\Bun_\FF^\mer(S)$ the stack of \textit{meromorphic vector bundles on the relative Fargues--Fontaine curve} over $S$. 
    \end{definition}
    It follows from \Cref{general-functoriality} and \Cref{isomorphism-isocrystals-bung} that we have a correspondence
	\begin{equation}
		\label{fundamental-correspondence}
	\begin{tikzcd}
		\Bun_\FF^\mer \ar{r}{\sigma} \ar{d}{\gamma} & \Bun_\FF\,. \\
		\IC^\Diamond  & 
	\end{tikzcd}
	\end{equation}

	We give names to these maps.

\begin{definition}
	\begin{enumerate}
		\item We call the map $\sigma\colon \VEC\to \Bun_{\on{FF}}$ constructed in (\ref{fundamental-correspondence}) the \textit{special polygon map}.
		\item We call the map $\gamma\colon \VEC\to \IC^\lozenge$ constructed in (\ref{fundamental-correspondence}) the \textit{generic polygon map}.
	\end{enumerate}
\end{definition}

    We now study basic properties of $\VEC$. Let $S = \Spa(R,R^+) \in \Perf$.

    \begin{proposition}
	    \label{DM1p-is-well-bheaved-subcategory-old}
	    The following statements hold.
	    \begin{enumerate}
		    \item $\SHT[\frac{1}{\pi}]$ is a v-separated presheaf in $\calP(\Perf\!,\CatE)$.
		    \item The v-sheafification of $\SHT[\frac{1}{\pi}]$ is equivalent to $\Bun^\mer_{\on{FF}}$ in $\calS(\Perf\!,\CatE)$.
		    \item A sequence in $\SHT[\frac{1}{\pi}]$ is exact in $\Bun^\mer_{\on{FF}}$ if and only if it is exact in $\Bun^{(\mer_\pre)}_{\on{FF}}$. 
		    \item Exactness in $\Bun^\mer_\FF$ can be verified on geometric points.
	    \end{enumerate}
    \end{proposition}
    \begin{proof}
	    The first two claims follow formally from \Cref{separatedness-of-isogeny} and \Cref{VYp-is-separated-presheaf}. 

	    Let $S=\Spa(R,R^+) \in \Perf$.
	    Fix $\Sigma\coloneqq[\calE_1\to \calE_2\to \calE_3]$ a sequence with $\calE_i\in \SHT[\frac{1}{\pi}](S)$.
	    It is clear that if $\Sigma$ is exact in $(\Bun_\FF)^{(\mer_\pre)}(S)$ then it is also exact in $\Bun_\FF^\mer(S)$.
	    Assume that $\Sigma$ is exact in $\Bun^\mer_\FF(S)$.
	    By definition, this means that $\Sigma$ is exact in $\Bun_\FF(\Spd (R'_\disc,R'^+_\disc))$ for a v-cover $S'\coloneqq\Spa(R',R'^+)\to \Spa(R,R^+)$, and we need to show that $\Sigma$ was already exact in $\Bun_\FF(\Spd (R_\disc,R^+_\disc))$.

	    Since the sheafification map $\Sht_\calY[\frac{1}{\pi}]\to \Bun^\mer_\FF$ is fully-faithful on $(R,R^+)$-points, we deduce that the sequence is a complex.
	    Let $T'$ and $T$ denote $\Spd (R'_\disc,R'^+_\disc)$ and $\Spd (R_\disc,R^+_\disc)$.
	    We can verify exactness of $\Sigma$ on geometric points of $T$.
	    We warn the reader that although the map $S'\to S$ is a v-cover the map $T'\to T$ might no longer be surjective even at the level of topological spaces.
	    Nevertheless, it is surjective on the loci where $\varpi$ is topologically nilpotent for a pseudo-uniformizer $\varpi\in R^+$.
	    Indeed these loci agree with $S'$ and $S$, respectively.
	    So it suffices to prove exactness of $\Sigma$ on the complement of $S$ in $T$.

	    Let $U=\Spd (R_\disc,R_\disc)$, this is the locus in $T$ where $|\varpi|\geq 1$. 
	    The complement of $S$ in $T$, is the locus in which $\varpi$ is not topologically nilpotent. 
	    If $x\in T\setminus S$ then there is a vertical generization $y$ of $x$ for which $|\varpi|_y=1$.
	    Indeed, $x$ is represented by a geometric point $\Spa(C,C^+)\to \Spd(R_{\on{disc}},R_{\on{disc}}^+)$ and $y$ is represented by the induced map $\Spa(C,O_C)\to \Spd(R_{\on{disc}},R_{\on{disc}}^+)$, so we see that $\varpi$ maps to $O_C$.
	    If $\varpi$ lands in the maximal ideal of $O_C$ then $y$ (and consequently $x$) are in the locus in which $\varpi$ is topologically nilpotent. 
	    Otherwise, the value of $|\varpi|_y=1$.
	    This shows that $T\setminus S\subseteq \overline{U}$ (where $\overline{U}$ denotes the closure), and since $S\subseteq T$ is an open subset, we must have $T\setminus S=\overline{U}$.
	    Moreover, as we argued above $\overline{U}\setminus U$ consists of vertical specializations of elements in $U$, and the same holds for $\overline{U} \times \Spd E$ and $U\times \Spd E$. 
	    We can now conclude that $\Sigma$ is exact over $\overline{U}$ if and only if it is exact over $U$.
	    Indeed, for any affinoid perfectoid $(A,A^+)$ the restriction functor 
	    \[\calV({Y_{\Spa(A,A^+)}})\to \calV({{Y_{\Spa(A,A^\circ)}}})\]
	    is an exact equivalence, so exactness can be verified on rank $1$ points and in particular it is insensitive to passing to vertical generizations.

	    By hypothesis, $\Sigma$ is exact when restricted to $\Spd (R',R')$.
	    By \Cref{isomorphism-isocrystals-bung}, we may interpret $\Sigma$ restricted to $U$ as a sequence in $\IC(\Spec R)$ that becomes exact over $\IC(\Spec R')$.
	    By \Cref{check-exact-isoc-points}, we can finish verifying exactness on closed points of $\Spec R$.
		But the map $\Spec R'\to \Spec R$ covers all closed points, since every maximal ideal of $R$ supports a valuation that is continuous for the $\varpi$-adic topology. 
		The kernel of any lift of such a valuation to $R'$ maps to this maximal ideal.
 
		For the final claim, we wish to prove that a sequence $\Sigma\coloneqq[\calE_1\to \calE_2\to \calE_3]$ is exact in $\VEC(S)$ if and only if for every geometric point $\overline{x}\to S$ the sequence $\Sigma_{\overline{x}}$ is exact. 
		By definition, exactness can be verified v-locally. 
		Hence, we may assume that $S=\Spa(R,R^+)$ is a product of points with $R^+=\prod_{i\in I}\!C_i^+$ and that each $\calE_j\in \SHT[\frac{1}{\pi}]$ for $j\in \{1,2,3\}$.

		Since the map $R\to \prod_{i\in I}\! C_i$ is injective, we can deduce that $\Sigma$ is a complex. 
		We can argue as above to show that $\Sigma$ is exact when interpreted as a sequence in $\Bun_\FF(\Spd(R_\disc,R^+_\disc))$.
		Namely, we show that $\Sigma$ is exact on all points of $\Spd(R_\disc,R^+_\disc)$. 
		This is clear on the locus where $\varpi$ is topologically nilpotent by our assumption.
		To verify exactness on $\Spd(R_\disc,R_\disc)$ we interpret this as an object in $\IC(\Spec R)$ and we may check exactness on closed points. 
		For any closed point, the map induced by the residue field $\Spec C\to \Spec R$ can be promoted to a geometric point $\Spa C\to \Spa(R,R^+)$ and the induced sequence in $\IC(\Spec C)$ is induced from the corresponding one in $\SHT[\frac{1}{\pi}](C,O_C)$, which is exact by assumption.
    \end{proof}

	\begin{remark}
		We want to point out that part (2) of \Cref{DM1p-is-well-bheaved-subcategory-old} gives us a moduli-theoretic interpretation of the a priori very abstract v-stack $\VEC$. Indeed, although a sheafification procedure is involved which makes thing not very explicit, at least the presheaf $\SHT[\frac{1}{\pi}]$ has an explicit moduli-theoretic interpretation as ``shtukas up-to-isogeny''.
	\end{remark}

The following statement will be key for our purposes.

\begin{corollary}
	\label{exactness-bungmer}
	A sequence $\Sigma\colon [\calE_1\to \calE_2\to \calE_3]$ in $\VEC(S)$ is exact if and only if its image in $\Bun_{\on{FF}}(S)$ is exact.
\end{corollary}
\begin{proof}
	Since both statements are v-local and can be verified at the level of geometric points, we may assume $S=\Spa (C,C^+)$.
	We observe that if $T\subseteq S$ is the rank $1$ point then $\VEC(T)\simeq \VEC(S)$ and $\Bun_{\on{FF}}(T)\simeq \Bun_{\on{FF}}(S)$.
	Indeed, every moduli problem involved in the construction of these categories is insensitive to the ring of integral elements since $[\varpi]$ is inverted.
	We assume that $S=\Spa (C,O_C)$ and that $\calE_i\in \SHT[\frac{1}{\pi}](S)$. 
	Fix a pseudo-uniformizer $\varpi\in O_C$. 
	In this case, exactness on $\VEC(S)$ is equivalent to exactness of underlying vector bundles over $\Spa \bbW(O_C)_{\{\pi\cdot [\varpi]\neq 0\}}$, while exactness on $\Bun_\FF(S)$ is equivalent to exactness of underlying vector bundles over $Y_{(0,\infty),S}=\Spa \bbA_{\on{inf}}(O_C)_{\{\pi\cdot [\varpi]\neq 0 \}}$.
	We can analyze the behavior on the loci $\{|[\varpi]|\leq |\pi|\}$ and $\{|\pi|\leq |[\varpi]|\}$.
	On the former locus, the two spaces agree so their categories of vector bundles have the same exact structure. 
	On the latter locus, we are comparing vector bundles over $\Spec B_{[0,1],S}[\frac{1}{\pi}]$ against vector bundles over $\calY_{(0,1],S}$.

	Recall from \cite[Theorem II.0.1, Corollary II.1.12]{FS24} that $B_{[0,1],S}$ is a principal ideal domain, and that the closed ideals give rise to untilts of $C$.
	The claim now follows from the fact that
	the map of locally ringed topological spaces 
	\[f\colon \calY_{(0,1],S}\to \Spec B_{[0,1],S}[\frac{1}{\pi}]\]
	covers every maximal ideal of the target and that $B_{[0,1],S}[\frac{1}{\pi}] \to \on{H}^0(\calY_{(0,1],S},\calO)$ is injective. 
	This implies that $f^*$ reflects exactness which is what we needed to show.
\end{proof}

\begin{proposition}
	\label{Beauville-Laszlo-presheaves}
	\label{all-squares-are-cartesian}
	The following diagrams are Cartesian in $\calS(\Perf\!,\CatexOE)$ and $\calP(\Perf\!,\CatE)$ respectively:
\begin{equation}
\begin{tikzcd}
	\SHT \arrow{r} \arrow{d}  & \Bun^\mer_\FF  \arrow{d}  &
\SHT[\frac{1}{\pi}] \arrow{r} \arrow{d}  & \Bun^\mer_\FF  \arrow{d}  \\
\DMan \arrow{r} & \IC^\Diamond & 
\DMan[\frac{1}{\pi}] \arrow{r} & \IC^\Diamond
\end{tikzcd}
\end{equation}
Here, the horizontal arrows in the right-hand square are the ones induced by sheafification.
\end{proposition}
\begin{proof} 
	The argument is a diagram chase whose key ingredients are \Cref{corollary-add-a-lattice} and \Cref{Cartesian-YVp}.  
	Since the two arguments are identical, we only provide the details for the first diagram. 
	From \Cref{Berkeley-notes-shtukas-is-v-sheaf} and \Cref{Passing_DM_to_diamondworld} we have the following Cartesian diagrams 
\begin{equation}
\label{diag:shtukas_as_pullback_2}
\begin{tikzcd}
	\SHT \arrow{r} \arrow{d}  & (\calV_Y)^\mer \arrow{d}{\Delta}  &
	\DMan \arrow{r} \arrow{d}  & (\calV_Y)^\An \arrow{d}{\Delta} \\
	\calV_\calY \arrow{r} & (\calV_Y)^\mer\times (\calV_Y)^\mer &
\calV_\bbW \arrow{r} & (\calV_Y)^\An\times (\calV_Y)^\An\,.
\end{tikzcd}
\end{equation}
Similarly, we obtain Cartesian diagrams
\begin{equation}\label{diag:isoshtukas_as_pullback}
\begin{tikzcd}
	\VEC \arrow{r} \arrow{d}  & (\calV_Y)^\mer \ar{d} & 
	\IC^\Diamond \arrow{r} \arrow{d}  & (\calV_Y)^\An  \arrow{d} \\
	(\calV_Y)^\mer \arrow{r} & (\calV_Y)^\mer\times (\calV_Y)^\mer &
	(\calV_Y)^\An \arrow{r} & (\calV_Y)^\An\times (\calV_Y)^\An\,.
\end{tikzcd}
\end{equation}
Moreover, these four Cartesian diagrams can be organized in a commutative square of Cartesian diagrams. 
For any fixed $i \in \{left,right\}$ and $j \in \{upper, lower\}$, their $(i,j)$th corners form a commutative diagram, which we denote $C_{i,j}$. 
For example, $C_{left,upper}$ is the diagram that we wish to prove is Cartesian. 
Note that $C_{left,lower}$ is Cartesian by \Cref{corollary-add-a-lattice} and that for any $j\in \{upper, lower\}$ the square $C_{right,j}$ is automatically Cartesian, since the horizontal maps in it are isomorphisms. 
From this and the fact that taking limits commutes with each other, it formally follows that $C_{upper,left}$ is also Cartesian. \qedhere
\end{proof}

\section{Semi-stable filtrations} \label{sec:semistable_filtrations}

As we have justified in \Cref{DM1p-is-well-bheaved-subcategory-old}  (resp.\ \Cref{DM1p-is-well-bheaved-subcategory}), given $S\in \Perf$, the category $\SHT[\frac{1}{\pi}](S)$ (resp.\ $\DMan[\frac{1}{\pi}](S)$) is a full subcategory of $\VEC(S)$ (resp.\ $\IC^\Diamond(S)$) and it unambiguously inherits an exact structure.
From this point on we will treat $\SHT[\frac{1}{\pi}]$ and $\DMan[\frac{1}{\pi}]$ as objects in $\calP(\Perf\!,\CatexE)$. 
For fixed $S\in \Perf$, we may think of $\DMan[\frac{1}{\pi}](S)$ as the full subcategory of those analytic isocrystals over $S$ that admit a lattice. Similarly, we think of $\SHT[\frac{1}{\pi}](S)$ as the full subcategory of those meromorphic vector bundles over $S$ that admit a lattice.

\begin{definition}
	\label{signconvention}
	Fix $S\in \Perf$ with $S=\Spa(R,R^+)$.	
	Given $\lambda\in \bbQ$ with $\lambda=\frac{m}{n}$ and $(m,n)=1$, we let $\calO(\lambda)\in \DMan(S)[\frac{1}{\pi}]$ be given by the pair $(\bbW(R)[\frac{1}{\pi}]^n,M)$, where $M$ is the matrix operator with $M\cdot e_i=e_{i+1}$ for $1\leq i\leq n-1$ and $M\cdot e_{n}=\pi^{-m}e_1$.
	We call $\calO(\lambda)$ the \textit{simple standard analytic isocrystal of slope} $\lambda$. 

	We say that an object in $\calF\in \IC^\Diamond(S)$ is \emph{standard} if it is isomorphic to one of the form
	\[\bigoplus_{\lambda \in \bbQ}\calO(\lambda)^{m_\lambda}\,,\]
	where $m\colon \bbQ\to \bbN$ is a multiplicity function with finite support.
\end{definition}

\begin{remark}
We warn the reader that our parametrization of standard analytic isocrystals reverses the signs of the parametrization of ``usual'' isocrystals used in most classical conventions.
\end{remark}

For us a \emph{Newton polygon} is a tuple $p=(\lambda_{1}, \lambda_{2},\dots, \lambda_{n_p}) \in \bbQ^{n_p}$, such that $\lambda_1 \geq \lambda_{2} \geq \dots \geq \lambda_{n_p}$, where $n_p \in \bbN_{>0}$.
We denote by $\calN$ the set of all Newton polygons. 
Then $\calN$ is endowed with the partial order $(\lambda_1, \lambda_2,\dots,\lambda_{n_1}) \leq (\mu_1, \mu_2,\dots, \mu_{n_2})$ if and only if $n_1 = n_2$,
\[\sum_{i=1}^{n_1} \lambda_i = \sum_{i=1}^{n_2} \mu_i \]
and for all $j = 1,\dots,n_1$ one has 
\[\sum_{i=1}^{j} \lambda_i \geq \sum_{i=1}^{j} \mu_i\,.\] 
To a standard analytic isocrystal $\bigoplus_{i=0}^n \calO(\lambda_i)^{m_i}$ (with $\lambda_1 > \lambda_2 > \dots > \lambda_n$ and $\lambda_i = \frac{p_i}{q_i}, \mathrm{gcd}(p_i,q_i)=1$), we can assign the Newton polygon 
$$(\underbrace{\lambda_1, \dots, \lambda_1}_{m_1q_1}, \underbrace{\lambda_2, \dots, \lambda_2}_{m_2q_2}, \dots, \underbrace{\lambda_{n},\dots, \lambda_n}_{m_nq_n}).$$
We say a Newton polygon $f$ is \emph{semi-stable} if $\lambda_i$ is constant for all $i$.
We let $\calN^{\on{ss}}\subseteq \calN$ denote the subset of semi-stable polygons, these are the minimal elements in $\calN$.

We wish to use $\calN$ to stratify $\VEC$. Before we do this, we make the following sanity check.
It says that over geometric points analytic isocrystals and meromorphic vector bundles admit a lattice. 
Alternatively, it says that sheafification does not change the naive value on geometric points. 

    \begin{lemma}
	    \label{lemma-sanity-check}
	    Let $S=\Spa(C,C^+)$ be a geometric point. We have the following equivalences in $\CatexE$:
	    \begin{enumerate}
		    \item $\IC^\Diamond(S)\simeq \Sht_\bbW[\frac{1}{\pi}](S)\simeq \IC(\Spec \overline{\bbF}_q)$.
		    \item $\VEC(S)\simeq \SHT[\frac{1}{\pi}](S)$.
	    \end{enumerate}
    \end{lemma}
    \begin{proof}
	    The second claim follows from the first one and \Cref{all-squares-are-cartesian}. 
	    By inspection, $\Sht_\bbW[\frac{1}{\pi}](S)\simeq \IC(\Spec \bar{\bbF}_q)$ and these categories sit fully-faithfully inside $\IC^\Diamond(S)$.
	    It suffices to show that $\Sht_\bbW[\frac{1}{\pi}](S)\subseteq \IC^\Diamond(S)$ is also essentially surjective. 
	    Fix $\calE\in \IC^\Diamond(S)$. 
		Then there is a v-cover $f\colon S'\to S$ such that $f^*\calE\in \Sht_\bbW[\frac{1}{\pi}](S')$.
	    We may assume that $S'=\Spa(C',C'^+)$ is a geometric point.
	    In this case, $f^*\calE$ is of the form $\bigoplus_{\lambda \in \bbQ}\calO(\lambda)^{m_\lambda}$ by the Dieudonné--Manin classification. 
	    Moreover, $\calE\in \on{Desc.}(\Sht_\bbW[\frac{1}{\pi}],S^{\prime}\!/S)$.
	    The descent datum can be recorded by an automorphism of $\bigoplus_{\lambda \in \bbQ}\calO(\lambda)^{m_\lambda}$ over $\Spec (C'\hat{\otimes}_C C')$.
	    Now, this is a connected affine scheme by \cite[Lemma 14.6]{Sch17}.
	    The descent datum is necessarily given by a constant function in $\prod_{\lambda \in \bbQ}\!\on{Aut}(\calO(\lambda)^{m_\lambda})$ and since it has to satisfy the cocycle condition upon pullback to $\Spec (C'\hat{\otimes}_C C'\hat{\otimes}_C C')$, this function is necessarily the identity. 
	    Consequently, $\calE$ is isomorphic to $\bigoplus_{\lambda \in \bbQ}\!\calO(\lambda)^{m_\lambda}$ already over $S$.
    \end{proof}

If $S$ is a geometric point, then isomorphism classes of objects in $\Bun_{\on{FF}}(S)$ and $\IC^\lozenge(S)$ are both in natural bijection with $\calN$. 
Indeed, for analytic isocrystals this is \Cref{lemma-sanity-check}, and for vector bundles on the Fargues--Fontaine curve this is proven in \cite{fargues_g_torseurs_en_theorie_de_hodge_p_adique} and \cite[Theorem 3.11]{Ans19}.
In other words, we have canonical bijections
\[\nu\colon \IC^\Diamond(S)\xrightarrow{\simeq} \calN \xleftarrow{\simeq} \Bun_\FF(S)\cocolon \nu\,.\]

\begin{definition}
	Given $S\in \Perf$ and $\calE \in \VEC(S)$ we define two functions $\gamma_\calE,\sigma_\calE\colon |S|\to \calN$ which we call the \textit{generic polygon} and \textit{special polygon}, respectively. 
	For $x\in|S|$ we choose a geometric point $\overline{x}\to S$ over $x$ and we let $\gamma_\calE(x)\coloneqq\nu(\gamma(\calE_{\overline{x}}))$. We define $\sigma_\calE$ similarly. 
\end{definition}
\begin{remark}
	Using a different language, Kedlaya proves that for any $\calE\in \VEC$ we have $\gamma_\calE\geq \sigma_\calE$, see \cite[Prop. 5.5.1]{Kedlaya_slope_revisited}. 
	This a key step in Kedlaya--Liu's proof of the semicontinuity theorem  \cite[Theorem 7.4.5]{kedlaya_liu_relative_p_adic_hodge_theory_foundations}.
\end{remark}

\begin{definition}
	Let $\calE\in \VEC(S)$ with constant rank and image $\calF\in \IC^\lozenge(S)$ under the map $\gamma\colon \VEC(S)\to \IC^\lozenge(S)$. 
	\begin{enumerate}
		\item We say that $\calF$ is \textit{locally standard} if its Newton polygon is locally constant.
		\item We say $\calE$ is \textit{generically locally standard} if $\calF$ is locally standard, equivalently if $\gamma_\calE$ is locally constant.
	\end{enumerate}
	We let $(\VEC)^{\on{loc}}(S)$ and $(\IC^\Diamond)^{\on{loc}}(S)$ denote the full subcategories described above.
\end{definition}

\begin{remark}
	The functors 
	\[(\VEC)^{\on{loc}},\,(\IC^\Diamond)^\loc \in \calP(\Perf\!,\CatexE)\]
	are still v-sheaves since the condition defining them can be verified v-locally.  
	Indeed, for a v-cover $f\colon \Spa(R_1,R_1^+)\to \Spa(R_2,R_2^+)$ and an open and closed decomposition $\Spa(R_1,R_1^+)=\coprod_{\gamma \in \calN} U_\gamma$ we must have $U_\gamma=f^{-1}(f(U_\gamma))$ since the generic Newton polygon is an invariant of the geometric points of $\Spa(R_2,R_2^+)$.  
	Since $|f|$ is a quotient map, $f(U_\gamma)$ is also closed and open in $|\Spa(R_2,R_2^+)|$ so the Newton polygon on $\Spa(R_2,R_2^+)$ is locally constant.
\end{remark}

\begin{definition}
	Let $S=\Spa(R,R^+)$.
	We say that an object $(\calF,\Phi)\in \DMan(S)$ is \textit{anti-effective} if the isomorphism $\Phi^{-1}\colon \calF\to \varphi^*\calF$ over $\Spec\bbW(R)[\frac{1}{\pi}]$ extends to a map $\Psi\colon \calF\to \varphi^*\calF$ defined over $\Spec \bbW(R)$. 
	An object in $\calE\in \SHT(S)$ is \textit{anti-effective} if its image in $\DMan(S)$ is anti-effective.
\end{definition}
\begin{proposition}
	\label{Lifting-to-antieffective}
	Let $\calE\in \VEC(S)$ such that the function $\gamma_\calE$ is constant and such that its smallest slope is $0$. 
	Then it lifts v-locally to an anti-effective crystalline shtuka. 
\end{proposition}
\begin{proof}
	By \Cref{Beauville-Laszlo-presheaves} it suffices to prove that locally standard analytic isocrystals of smallest slope $0$ lift v-locally to an anti-effective Dieudonn\'e module. 
	Working v-locally, we may assume $\gamma(\calE)\in \DMan[\frac{1}{\pi}](S)$, and since $\gamma(\calE)$ is locally standard, we may by \cite[Theorem 2.11]{HamacherKim_22} even assume $\gamma(\calE)\cong \oplus_{i=1}^n \calO(\lambda_i)^{m_i}$. By assumption, $\lambda_i\geq 0$ for $i$.
	The standard models of $\calO(\lambda_i)$ already define an anti-effective crystalline shtuka by inspection of \Cref{signconvention}. 
\end{proof}

\begin{lemma}
	\label{key-lemma-filtration-extends}	
	Suppose that $S=\Spa(R,R^+)$ is a product of points. 
	Let $(\calE,\Phi)\in \SHT(S)$ be anti-effective, then 
	\[\on{Hom}_{\VEC}(\calO,\calE)=\on{Hom}_{\IC^\lozenge}(\calO,\gamma(\calE))\,.\]
	Moreover, if $f\in \on{Hom}_{\IC^\lozenge}(\calO,\gamma(\calE))$ defines a sub-isocrystal $\calO\subseteq \calE$, then the corresponding lift also defines a sub-bundle $\calO\subseteq \calE$ in $\VEC$.     
\end{lemma}
\begin{proof}
	By \Cref{DM1p-is-well-bheaved-subcategory-old}, we may compute $\on{Hom}_{\VEC}(\calO,\calE)$ in $\SHT[\frac{1}{\pi}]$.
	Since $B_{[0,r],S}\subseteq \bbW R$, the map  
	\[\on{Hom}_{\VEC}(\calO,\calE)\to \on{Hom}_{\IC^\lozenge}(\calO,\gamma(\calE))\]
	is injective.
	To prove surjectivity, we fix a basis of $\beta\colon \calO^n\to \calE$ over $\calY_{[0,\frac{q}{N}]}$ for some $N\in \bbN$.
	This induces a basis $\varphi^*\beta\colon \calO^n\to \varphi^*\calE$ over $\calY_{[0,\frac{1}{N}]}$.
	Now let $r=\frac{1}{N}$. 
	Since $(\calE,\Phi)$ is anti-effective, we can think of $(\calE,\Phi)$ through $\beta$ and $\varphi^*\beta$ as a matrix $M\in \on{GL}_n(B^R_{[0,r]})$ such that 
	\[M^{-1}\in \on{GL}_n(B^R_{[0,r]}[\frac{1}{\pi}])\cap M_{n\times n}(\bbW R)\,.\] 
	A map $f\in \on{Hom}_{\IC^\lozenge}(\calO,\gamma(\calE))$ can then be thought of as a vector $v\in \bbW(R)[\frac{1}{\pi}]^n$ satisfying the equation
	\[M\varphi v=v\,.\]
	On the other hand, $v\in \on{Hom}_{\VEC}(\calO,\calE)$ if and only if $v\in B_{[0,s]}[\frac{1}{\pi}]$ for some $s>0$. 
	Indeed, we can use $\varphi$-equivariance to extend this map along $\calY_{(\frac{s}{2},\infty)}$. 
	Replacing $v$ by $\pi^N\cdot v$, we may assume $v\in \bbW(R)^n$. 

	We fix a norm of $|\cdot|\colon R\to \bbR$ inducing the topology of $R$ with $|\varpi|=\frac{1}{q}$ and define a function $|\cdot|_k\colon \bbW R\to \bbR$ by the formula
	\[\sum_{i=0}^\infty [a_i]\pi^i\mapsto \on{sup}_{0\leq i\leq k} |a_i|.\]
	This definition extends to $M_{n\times n}(\bbW R)$ and $(\bbW R)^n$ by taking supremum over the entries.
	By the strong triangle inequality, and because $M^{-1}\in M_{n\times n}(\bbW R)$, we see that for every $k\in \bbN$ the inequality $|M^{-1}\cdot v|_k\leq |M^{-1}|_k\cdot |v|_k$ holds and by inspection $|\varphi v|_k=|v|_k^q$.
	From this we deduce that $|v|^{{q-1}}_k\leq |M^{-1}|_k$. 
	Let $m_{ij}\in B^R_{[0,r]}$ denote the $(i,j)$ entry of $M^{-1}$ and write $m_{ij}=\sum_{l=0}^\infty [m_{ijl}] \pi^l$. 
	The sequences $m_{ijl}$ all satisfy that $\lim_{l\mapsto \infty} |m_{ijl}|\cdot (\frac{1}{q})^{N\cdot l}=0$.
	Now, \Cref{maximum-of-a-sequence-drops} shows that $\lim_{l\mapsto \infty} |M^{-1}|_l\cdot (\frac{1}{q})^{N\cdot l}=0$ and in particular that $\lim_{l\mapsto \infty} |v|_l\cdot (\frac{1}{q})^{N\cdot(q-1)\cdot l}=0$, which implies that $v\in (B^R_{[0,\frac{1}{N\cdot(q-1)}]})^n$ as we needed to show. 

	By \Cref{DM1p-is-well-bheaved-subcategory-old}, the last claim can be verified at the level of geometric points. 
	Consider the ideal $I$ in $B^C_{[0,1]}$ generated by the entries of $v$.
	By \cite[Corollary II.1.12]{FS24}, $B^C_{[0,1]}$ is a principal ideal domain, so the zero locus of $I$ consists of finitely many closed points in $\Spec B^C_{[0,1]}$. 
	Moreover, the zero locus is $\varphi$-equivariant so it is at worst the ideal cut out by $\pi$, but then it avoids $\Spec B^C_{[0,1]}[\frac{1}{\pi}]$.  
\end{proof}

    \begin{lemma}
	\label{maximum-of-a-sequence-drops}
	    Let $I$ be a finite set and $\rho$ a number with $0<\rho<1$. 
		For each $i \in I$, let $(b_{i,j})_{j\geq 0}$ be a sequence in $\mathbb{R}_{\geq 0}$ such that $\lim_{j\to \infty}b_{i,j}\cdot \rho^j=0$. 
		For each $j\geq 0$, let $B_j = {\rm max}_{i \in I, j'\leq j}\{b_{i,j'}\}$. 
		Then $\lim_{j\to \infty}B_j\cdot \rho^j=0$.
    \end{lemma}
    \begin{proof}
	    This easily reduces to the case $I=\{1\}$. Fix $\varepsilon>0$. 
		By assumption, there is some $j_{\varepsilon,0}>0$ such that for all $j \geq j_{\varepsilon,0}$, we have $b_j \rho^j < \varepsilon$. 
		Put 
    	\[\lambda = {\rm max}_{j'<j_{\varepsilon,0}} b_{j'}\rho^{j'}\,.\]
    	We now pick a big enough $j_{\varepsilon}$, such that $\rho^{j_{\varepsilon} - j_{\varepsilon,0}} \lambda < \varepsilon$. Then for any $j\geq j_{\varepsilon}$ we have
    	\begin{align*}
    		B_j \rho^j &= {\rm max}_{j'\leq j} \{b_{j'}\rho^j\} \\
    		&= {\rm max}\{ {\rm max}_{j'<j_{\varepsilon,0}}\{b_{j'}\rho^{j'}\rho^{j-j'}\}, {\rm max}_{j_{\varepsilon,0}\leq j' \leq j}\{b_{j'}\rho^{j'}\rho^{j-j'}\} \} < \varepsilon\,.
    	\end{align*}
    	Indeed, if $j'<j_{\varepsilon,0}$, then $b_{j'}\rho^{j'}\rho^{j-j'} \leq \lambda \rho^{j-j'} \leq \lambda \rho^{j_\varepsilon - j_{\varepsilon,0}} < \varepsilon$ (as $\rho < 1$ and $j-j' \geq j_\varepsilon - j_{\varepsilon,0}$); and if $j' > j_{\varepsilon,0}$, then $b_{j'}\rho^{j'} < \varepsilon$ and $\rho^{j-j'} < 1$. 
    \end{proof}

\begin{definition}
	\label{definition-semistable-of-slope}
	Let $S\in \Perf$. 
	For a fixed $\lambda\in \bbQ$, we let $(\VEC)_{\lambda}(S)\subseteq (\VEC)^{\on{loc}}(S)$, (resp.\ $(\IC^\Diamond)_{\lambda}(S)\subseteq (\IC^\Diamond)^{\on{loc}}(S)$, resp.\ $\Bun_\FF(S)_\lambda\subseteq \Bun_\FF$) denote the full subcategories of objects whose generic Newton polygon function $\gamma_{(-)}$ (resp.\ Newton polygon function $\nu$) attaches to each geometric point of $S$ a constant polygon of slope $\lambda$. 
	We call objects in these subcategories semi-stable of slope $\lambda$.
\end{definition}

\begin{proposition}
	\label{semistable-gives-equivalences}
	For all $\lambda \in \bbQ$ the maps $(\IC^\lozenge)_\lambda \xleftarrow{\gamma} (\VEC)_\lambda \xrightarrow{\sigma} (\Bun_{\on{FF}})_\lambda$ are exact equivalences of sheaves of $E$-linear exact categories.
\end{proposition}
\begin{remark}
We note that the categories $(\IC^\lozenge)_\lambda,\, (\VEC)_\lambda,\,  (\Bun_{\on{FF}})_\lambda$ are not stable under $\otimes$-products.
\end{remark}
\begin{proof} 
To prove that $\gamma$ and $\sigma$ are equivalences, it suffices to show that they are fully-faithful.
Indeed, by \cite[Proposition 4.3.13]{CS17} (resp.\ \cite[Theorem I.3.4]{FS24}) every object of $(\IC^\lozenge)_\lambda$ (resp.\ $(\Bun_\FF)_\lambda$) is pro-\'etale locally isomorphic to $\calO(\lambda)^m$ which is already in $\VEC(\Spd \overline{\bbF}_q)$.
Then, \Cref{lm:ess_surjectivity_local} allows us to conclude. 

To show that $\gamma$ and $\sigma$ are fully-faithful, it suffices to deal with the case in which $\lambda=0$.
Indeed, for general $\lambda$ we may pass to internal $\underline{\on{Hom}}$-objects and take global sections which reduces us to prove that the maps 
\[ \on{Hom}_{\IC^\Diamond}(\calO,\gamma \calE)\xleftarrow{\gamma} \on{Hom}_{\VEC}(\calO,\calE)  \xrightarrow{\sigma} \on{Hom}_{\VEC}(\calO,\sigma \calE) \] 
are isomorphisms for the internal $\underline{\on{Hom}}$-bundle $\calE\in \VEC(S)_0$.

Let us show $\gamma$ is fully-faithful in the case $\lambda=0$.  
Since both 
\[S'\mapsto \on{Hom}_{\IC^\Diamond}(\calO,\gamma (\calE_{S'})) \text{ and } S'\mapsto \on{Hom}_{\VEC}(\calO,\calE_{S'}) \]
are v-sheaves on $\Perf/S$, and by \Cref{Lifting-to-antieffective}, we can assume that $\calE$ is the image of an anti-effective crystalline shtuka. In this case, we conclude via \Cref{key-lemma-filtration-extends}. 

Let $\underline{E}-\on{Loc}$ denote the category of pro-\'etale $\underline{E}$-local systems. 
By \cite[Proposition 4.3.13]{CS17}, \cite[Theorem I.3.4]{FS24} and the argument above we get to a commutative diagram in $\CatE$ of the form 
\begin{center}
\begin{tikzcd}
	\underline{E}-\on{Loc}(S) \ar[bend left]{rrd}{\simeq} \ar[bend right]{ddr}{\simeq} \ar{dr} &  &  \\
										  & \VEC(S)_0  \arrow{r} \arrow{d}{\simeq}  & \Bun_\FF(S)_0 \,. \\
				 & \IC^\Diamond(S)_0  & 
\end{tikzcd}
\end{center}
This is enough to conclude that the remaining two arrows are also equivalences.
So far we have shown that the maps $(\IC^\lozenge)_\lambda \xleftarrow{\gamma} (\VEC)_\lambda \xrightarrow{\sigma} (\Bun_{\on{FF}})_\lambda$ are equivalences in $\CatE$, and it is left to argue that these equivalences are exact.

By \Cref{analytic-vector-bundles-exact-geometric}, \Cref{check-exact-isoc-points_2} and \Cref{exactness-bungmer}, exactness of the equivalences can be checked on geometric points. 
By \Cref{lemma-sanity-check}, the three categories are equivalent to a semi-simple category over a geometric point, so the exact structure is the one inherited from the additive structure and any equivalence preserves it. 
\end{proof}


We consider $\bbQ$-filtered meromorphic vector bundles (resp.\ vector bundles, resp.\ analytic isocrystals). 
That is, we consider sequences of the form $\{\calE_{\leq r}\}_{r\in \bbQ}\in \VEC(S)$ (resp.\ $\{\calE_{\leq r}\}_{r\in \bbQ}\in \Bun_{\on{FF}}(S)$, resp.\ $\{\calE_{\leq r}\}_{r\in \bbQ}\in \IC^\lozenge(S)$) with $\calE_{\leq r}\subseteq \calE_{\leq s}$ when $r<s$ such that the inclusion fits in an exact sequence 
\[0\to \calE_{\leq r}\to \calE_{\leq s}\to \calE_{\leq s}/\calE_{\leq r}\to 0\] 
where $\calE_{\leq s}/\calE_{\leq r}$ denotes an object in $\VEC(S)$ (resp.\ $\Bun_\FF(S)$, resp.\ $\IC^\Diamond(S)$) determined (up to unique isomorphism) by the inclusion map $\calE_{\leq r}\hookrightarrow \calE_{\leq s}$ whenever it exists.\footnote{Note that since $\VEC(S)$ (resp.\ $\Bun_{\on{FF}}(S)$, resp.\ $\IC^\lozenge(S)$) is only an exact category and not an abelian category, it is not automatic that the object $\calE_{\leq s}/\calE_{\leq r}$ completing the exact triangle exists.}. 
We also ask that $\calE_{\leq r}/\calE_{<r}=0$ for all but finitely many $r\in \bbQ$. 
By hypothesis, there is a $N \gg 0$ such that $\calE_{\leq s}=\calE_{\leq N}$ for every $s>N$, and we call $\calE_{\leq N}$ \textit{the underlying vector bundle} of $\{\calE_{\leq r}\}_{r\in \bbQ}$.

\begin{definition}
	We say that a $\bbQ$-filtered meromorphic vector bundle (resp.\ a vector bundle, resp.\ analytic isocrystal) is a \textit{semi-stable filtration} if $\calE_{\leq r}/\calE_{<r}$ is semi-stable of slope $r$ in the sense of \Cref{definition-semistable-of-slope}.
	We let $\Fil_{\on{ss}}^\mer(S)$ (resp.\ $\Fil^\sigma_{\on{ss}}(S)$, resp.\ $\Fil_{\on{ss}}^\gamma(S))$ denote the categories whose objects are semi-stable filtrations and whose morphisms are maps in $\VEC(S)$ (resp.\ $\Bun_{\on{FF}}(S)$, resp.\ $\IC^\lozenge(S)$) that respect the filtration.  
	We endow these categories with the exact structure inherited from their underlying vector bundle.
	Moreover, the tensor product of underlying vector bundles inherits a filtration with 
	\[(\calE\otimes \calF)_{\leq r}=\sum_{r_\calE+r_\calF=r} \calE_{\leq r_\calE}\otimes \calF_{\leq r_\calF}\]
	such that 
	\[(\calE\otimes \calF)_{\leq r}/(\calE\otimes \calF)_{<r}=\bigoplus_{r_\calE+r_\calF=r} \calE_{\leq r_\calE}/\calE_{<r_\calE}\otimes \calF_{\leq r_\calF}/\calF_{<r_\calF}\,.\]
\end{definition}

\begin{proposition}
	\label{two-types-of-semistable-fils}
	The natural map $\Fil_{\on{ss}}^\mer\to \Fil^\sigma_{\on{ss}}$ is a $\otimes$-exact equivalence of v-stacks.  	
\end{proposition}
\begin{proof}
	\textit{Full-faithfulness:} Let $\{\calE_{\leq r}\}_{r\in\bbQ}$ and $\{\calF_{\leq r}\}_{r\in\bbQ}$ be in $\Fil_{\on{ss}}^\mer(S)$,	
	with underlying meromorphic vector bundles $\calE$ and $\calF$.
	The internal $\underline{\on{Hom}}$-bundle 
	\[\calH\coloneqq\underline{\on{Hom}}(\calE,\calF)=\calE^\vee\otimes \calF\]
	is naturally endowed with a $\bbQ$-filtration $\{\calH_{\leq r}\}_{r\in \bbQ}$.
	It is not hard to verify that $\{\calH_{\leq r}\}_{r\in \bbQ}$ is a semi-stable filtration.  
	Moreover, we have an identification
	\[\on{Hom}_{\Fil_{\on{ss}}^\mer}(\{\calE_{\leq r}\}_{r\in\bbQ},\{\calF_{\leq r}\}_{r\in\bbQ})=\on{Hom}_{\VEC}(\calO,\calH_{\leq 0})\,.\] 
	Analogously,
	\[\on{Hom}_{\Fil_{\on{ss}}}(\{\calE_{\leq r}\}_{r\in\bbQ},\{\calF_{\leq r}\}_{r\in\bbQ})=\on{Hom}_{\Bun_{\on{FF}}}(\calO,\calH_{\leq 0})\,.\] 
	Since $\{\calH_{\leq r}\}_{r\in \bbQ}$ is semistable, by applying induction with respect to $\{r\in \bbQ\mid \calH_{\leq r}/\calH_{<r}\neq 0\}$ one can show that 
	\[\on{Hom}_{\VEC}(\calO,\calH_{\leq r})=0=\on{Hom}_{\Bun_{\on{FF}}}(\calO,\calH_{\leq r})\] 
	for all $r<0$. 
	To prove fully-faithfulness it suffices to show that
	\[\on{Hom}_{\VEC}(\calO,\calH_{\leq 0}/\calH_{<0})\cong \on{Hom}_{\Bun_{\on{FF}}}(\calO,\calH_{\leq 0}/\calH_{<0})\,.\]
	But since $\calH_{\leq 0}/\calH_{<0}$ is semi-stable of slope $0$, the result follows directly from \Cref{semistable-gives-equivalences}. 
	
\textit{Essential surjectivity:}
Let $\{\calE_{\leq r}\}\in \Fil_{\on{ss}}^\sigma$ with underlying vector bundle $\calE$ of rank $n$. 
	If $E_s$ is the degree $s$ unramified extension of $E$, then objects in $\Bun_{\on{FF}}$ can be constructed by descent from objects in $\Bun_{\on{FF},E_s}$, and by fully-faithfulness a descent datum in $\Fil_{\on{ss}}^\sigma$ agrees with a descent datum in $\Fil^\mer_{\on{ss}}$.
	This reduces us to show that for a suitable $s$, the pullback of the filtered vector bundle $\{\calE_{\leq r}\}$ to the Fargues--Fontaine curve associated with $E_s$ lies in the essential image of $\Fil_{\on{ss}}^\mer\to \Fil^\sigma_{\on{ss}}$. 
	This pullback has the effect of multiplication by $s$ on the slopes (see \Cref{lemma_unramifieds}). 
	Picking $s$ to be the product of all denominators appearing on the slopes of $\{\calE_{\leq r}\}$, we may thus assume that the support of the filtration is contained in $\bbZ$.
	Since essential surjectivity can now be proved v-locally by Lemma \ref{lm:ess_surjectivity_local}, we may think of every bundle $\calE_{\leq r}$ as a free module $M_r$ over $B^R_{[1,q]}$ with $\varphi$-glueing data over $B^R_{[1,1]}$.
	We may even assume that the graded pieces $\calE_{\leq N}/\calE_{<N}$ are isomorphic to $\calO(N)^{m_N}$.
	We may choose bases for the $M_r$ over $B^R_{[1,q]}$ compatible with the filtration and in such a way that after transferring the Frobenius structure to $\calO^n$, the induced $N$-graded pieces are given by diagonal matrices of the form $\pi^{-N}$. 
	Thus $\{\calE_{\leq r}\}$ is represented by an upper block-diagonal matrix $A \in M_{n\times n}(B^R_{[1,1]})$, with diagonal blocks of the form $\pi^{-N}\cdot \on{Id}_{m_N,m_N}$, describing the Frobenius structure. 
	Let $P \subseteq \GL_n$ denote the parabolic subgroup (containing the upper triangular matrices) corresponding to the block-diagonal shape of $A$.

	Now it suffices to prove the claim that there is a matrix $A_\infty \in P(B^R_{[0,1]}[\frac{1}{\pi}])$ and a matrix $U\in P(B^R_{[1,q]})$ with
	$U^{-1}A_\infty\varphi(U)=A$. Indeed, such $U$ defines a (not necessarily meromorphic) isomorphism of the filtered vector bundle $\{\calE_{\leq r}\}$ with the filtered vector bundle represented by $A_\infty$; now, $\pi$ is only meromorphically inverted in $B_{[0,1]}^R[\frac{1}{\pi}]$, hence $A_\infty$ in fact defines a meromorphic filtered vector bundle.
    Now our claim follows from Lemma \ref{lm:conjugation_to_meromorphic_upper_triangular} below.
\end{proof}

Before proving the remaining Lemma \ref{lm:conjugation_to_meromorphic_upper_triangular}, we need some preparations.

\begin{lemma}
	\label{lemma_unramifieds}
	Let $E_s$ is the degree $s$ unramified extension of $E$, for $C$ a fixed algebraically closed non-Archimedean field consider the map of Fargues--Fontaine curves attached to $C$, $\pi_s:X_{C,E_s}\to X_{C,E}$, then for all $m$ and $n$ with $(m,n)=1$ the slope of $\pi_s^*\calO(\frac{m}{n})$ is $\frac{s\cdot m}{n}$. 
\end{lemma}
\begin{proof}
	The standard presentation of $\calO(\frac{m}{n})$ is through the matrix $n\times n$ matrix with $\varphi(e_i)=e_{i+1}$ for $i\in \{1,\dots n-1\}$ and $\varphi(e_n)=\pi^m e_1$. 
	We may write $s=k\cdot s'$ and $n=k\cdot n'$ with $(s',n')=1$.
	Then $\pi^*_k\calO(\frac{m}{n})$ has matrix $\varphi^k(e_i)=e_{i+k}$ for $i\in \{1,\dots, n-k\}$ and $\varphi^k(e_i)=\pi^m e_{i+k-n}$ for $i\in \{n-k+1,\dots, n\}$.
	We can now rewrite $\pi^*_k\calO(\frac{m}{n})$ as $\calO(\frac{m}{n'})^{k}$. 
	Finally, if $s'=j\cdot n'+r$ with $(r,n')=1$ and $r<n'$ then $\pi^*_{s'}\calO(\frac{m}{n'})$ has matrix $\varphi^{s'}(e_i)=\pi^{m\cdot j}e_{i+r}$ if $i\in \{1,\dots n'-r\}$ and $\varphi^{s'}(e_i)=\pi^{m\cdot (j+1)}e_{i+r-n'}$ if $i\in \{n'-r+1,\dots n'\}$.
	Relabeling the basis $e'_1=e_1$ and $e'_{i+1}=\varphi^s(e_i)$ for $i\in \{1,\dots, n'\}$ we can compute that $\varphi^s(e'_{n'})=\pi^{m\cdot (j\cdot n'+r)}e_1$, which is precisely the standard matrix for $\calO(\frac{m\cdot s'}{n'})$.
\end{proof}

\begin{lemma}\label{lm:decomposition_into_mero_and_varpi_part} 
We have $B_{[1,1]}^R = B_{[0,1]}^R[\frac{1}{\pi}] + [\varpi]B_{[1,\infty]}^R$\,.
\end{lemma}
\begin{proof}
Let $A_1 = \bbW(R^+)[\frac{\pi}{[\varpi]}]$, $A_2 = \bbW(R^+)[\frac{[\varpi]}{\pi}]$ and $A_{12} = \bbW(R^+)[\frac{\pi}{[\varpi]},\frac{[\varpi]}{\pi}]$. 
We have $B_{[1,1]}^R = (A_{12})_\pi^\wedge[\frac{1}{\pi}]$, $B_{[0,1]}^R = (A_1)_{[\varpi]}^\wedge[\frac{1}{[\varpi]}]$ and $B_{[1,\infty]}^R = (A_2)_\pi^\wedge[\frac{1}{\pi}]$. 
After multiplication with a big enough power of $\pi$, it suffices to show that any element of $(A_{12})_\pi^\wedge$ can be written as a sum of an element of $(A_1)^\wedge_{[\varpi]}$ and an element of $\frac{[\varpi]}{\pi} \cdot (A_2)_\pi^\wedge$. 

For any $n \geq 1$, let $I_n = \{(i,j) \in \bbZ^2 \colon 0\leq i < n \}$ and let 
\[
S_n \subseteq \!\prod_{(i,j) \in I_n}\! R^+
\] 
be the subset of all sequences $a = (a_{ij})_{ij}$ for which $a_{ij} = 0$ except for finitely many $(i,j) \in I_n$. 
Let also $S_n^+ \subseteq S_n$ (resp. $S_n^- \subseteq S_n$) be the subset of all sequences for which $a_{ij} = 0$ unless $j\geq 0$ (resp. $a_{ij} = 0$ unless $j < 0$). 
There is a commutative diagram $D_n$ of sets
\begin{center}
\begin{tikzcd}
S_n^+ \ar{r} \ar{d} & S_n \ar{d} & S_n^- \ar{l} \ar{d} \\
A_1/[\varpi]^n A_1 \ar{r} & A_{12}/\pi^n A_{12} & \frac{[\varpi]}{\pi}\cdot (A_2/\pi^n A_2) \ar{l}
\end{tikzcd}
\end{center}
(note that $A_{12}/\pi^n A_{12} = A_{12}/[\varpi]^nA_{12}$), where the upper horizontal maps are the defining inclusions, the lower horizontal maps are induced by the natural ring maps $A_1 \rightarrow A_{12} \leftarrow A_2$ (and the inclusion of the ideal $\frac{[\varpi]}{\pi}A_2 \subseteq A_2$) and the vertical maps are given by sending $(a_{ij})_{ij}$ to $\sum_{ij} [a_{ij}]\pi^i\cdot (\frac{\pi}{[\varpi]})^j$. 

We make three observations, which immediately follow from the explicit definition of the vertical maps: first, the middle vertical map is surjective. 
Second, there is an obvious map $D_{n+1} \rightarrow D_n$ of commutative diagrams and the resulting diagram is commutative. 
Third, when we define the map $+ \colon S_n^+ \times S_n^- \rightarrow S_n$ by $(a+b)_{ij} = a_{ij}$ if $j\geq 0$ and $(a+b)_{ij} = b_{ij}$ if $j<0$, then the resulting diagram
\begin{center}
\begin{tikzcd}
S_n^+ \times S_n^- \ar{r}{+} \ar{d} & S_n \ar{d} \\
A_1/[\varpi]^n A_1 \times \frac{[\varpi]}{\pi}\cdot(A_2/\pi^n A_2) \ar{r}{+} & A_{12}/\pi^n A_{12}
\end{tikzcd}
\end{center}
is commutative.

Let now $S = \lim_n S_n$ and $S^{\pm} = \lim_n S_n^{\pm}$. 
Explicitly, $S \subseteq \prod_{(i,j) \in \bbZ_{\geq 0} \times \bbZ}\! R^+$ is the subset of all sequences $(a_{ij})_{ij}$ satisfying the following condition: for each $i$ there is some $j(i)\geq 0$ such that $a_{ij} = 0$ unless $|j|<j(i)$ and $S^+$ and $S^-$ are corresponding subsets of $S$. 
Passing to the limit over all $n > 0$, we obtain a commutative diagram 
\begin{center}
\begin{tikzcd}
S^+ \ar{r} \ar{d} & S \ar{d} & S^- \ar{l} \ar{d} \\
(A_1)_{[\varpi]}^\wedge \ar{r} & (A_{12})^\wedge_\pi & \frac{[\varpi]}{\pi}\cdot (A_2)^\wedge_\pi \ar{l}\,.
\end{tikzcd}
\end{center}
Moreover, we also get the commutative diagram
\begin{center}
\begin{tikzcd}
S^+ \times S^- \ar{r}{+} \ar{d} & S \ar{d} \\
(A_1)_{[\varpi]}^\wedge \times \frac{[\varpi]}{\pi}\cdot (A_2)_\pi^\wedge \ar{r}{+} & (A_{12})_\pi^\wedge\,,
\end{tikzcd}
\end{center}
where the lower horizontal map is the restriction of the addition map $B_{[0,1]}\times \frac{[\varpi]}{\pi}\cdot B_{[1,\infty]}\to B_{[1,1]}$ and the upper horizontal map is defined in the same way as $S_n^+ \times S_n^- \rightarrow S_n$. 
Now, one can concretely verify that $S^+ \times S^- \rightarrow S$ and $S \rightarrow (A_{12})_\pi^\wedge$ are surjective using \Cref{seth-theretic-lemma}. 
This implies that the lower horizontal map in the diagram is surjective as well, which is precisely what we had to show.\qedhere
\end{proof}

\begin{lemma}
	\label{seth-theretic-lemma}
	Treat $\bbN$ together with is partial order as a category. 
	Let $A_{(-)},B_{(-)}\colon \bbN^\op\to \Sets$ denote two functors and denote by $g^A_{n}\colon A_n\to A_{n-1}$ and $g^B_{n}\colon B_n\to B_{n-1}$ the image of the unique morphism $(n-1)\to (n)$.
Let $f\colon A\to B$ be a natural transformation that is pointwise surjective.
Suppose that for all triples $(n,a_n,b_{n+1})$ with $n\in \bbN$, $a_n\in A_n$, $b_{n+1}\in B_{n+1}$ and $f_n(a_n)=g^B_{n+1}(b_{n+1})$ there exists $a_{n+1}\in A_{n+1}$ with $f_{n+1}(a_{n+1})=b_{n+1}$ and $g^A_{n+1}(a_{n+1})=a_n$.
Then $f\colon \varprojlim A_n\to \varprojlim B_n$ is surjective. 
\end{lemma}
\begin{proof}
	An element $b\in \varprojlim B_n$ is given by a sequence $(b_n)_{n\in \bbN}$ with $g^B_n(b_n)=b_{n-1}$.	
	The hypothesis of the lemma allows us to inductively define a sequence $(a_n)_{n\in \bbN}$ such that $a_n\in A_n$, $g^A_n(a_n)=a_{n-1}$ and $f_n(a_n)=b_n$. 
	The sequence $(a_n)_{n\in \bbN}$ defines an element of $\varprojlim A_n$ whose image under the limit map is $b_n$. 
\end{proof}

Recall that the restriction of functions defines an inclusion $B_{[\frac{1}{q},\infty]}^R \subseteq B_{[1,\infty]}^R$ and Frobenius induces an isomorphism $\varphi \colon B_{[1,\infty]}^R \stackrel{\sim}{\rightarrow} B_{[\frac{1}{q},\infty]}^R \subseteq B_{[1,\infty]}^R$.

\begin{lemma}\label{lm:Frob_minus_res_surjective}
Let $k \in \bbZ_{\geq 0}$. The image of the map
\[
\psi_k \colon B_{[1,\infty]}^R \rightarrow B_{[1,\infty]}^R, \quad a \mapsto \pi^{-k}a - \varphi(a)
\]
contains $[\varpi]B_{[1,\infty]}^R$. If $k>0$, it contains $B_{[1,\infty]}^R$.
\end{lemma}

\begin{proof}
Let $A = \bbW(R^+)[\frac{\varpi}{\pi}]$. 
Recall that $B_{[1,\infty]}^R = A^\wedge_\pi[\frac{1}{\pi}]$. 
Thus, as $\psi_k(\pi^nx) =\pi^n\psi_k(x)$, it suffices to show that the image contains $[\varpi]A^\wedge_\pi$ (resp. $A^\wedge_\pi$ if $k>0$). 
Let $x \in [\varpi]A_\pi^\wedge$ if $k=0$ (resp. $x \in A^\wedge_\pi$ if $k>0$). 
Note that the sequence $(\pi^{i\cdot k} \varphi^{(i-1)}(x))_{i\geq 1}$ in $A^\wedge_\pi$ converges $\pi$-adically to $0$. (Use that $\varphi(A_\pi^\wedge) \subseteq A_\pi^\wedge$ and $\varphi([\varpi]) = [\varpi]^q$.) 
Thus $y = \sum_{i=1}^\infty \pi^{ik}\varphi^{(i-1)}(x)$ exists in $A_\pi^\wedge$. 
By $\pi$-adic continuity of Frobenius and hence of $\psi_k$, it is immediate that $\psi_k(y) = x$.
\end{proof}

\begin{lemma}\label{lm:conjugation_to_meromorphic_upper_triangular}
Let $n\geq 1$ and let $A \in \GL_n(B_{[1,1]}^R)$ be upper triangular with $i$-th diagonal entry $\pi^{s_i}$ for some $s_i \in \bbZ$ (with $1\leq i\leq n$). 
Assume that $s_1 \geq s_2 \geq \dots \geq s_n$ holds. 
Then there exists a unipotent upper triangular matrix $U \in \GL_n(B_{[1,\infty]}^R)$ such that $U^{-1}A\varphi(U)$ is upper triangular with entries in $B_{[0,1]}^R[\frac{1}{\pi}]$. 
\end{lemma}

\begin{proof}
We argue by induction on $n$. If $n=1$, there is nothing to show. 
Assume $n$ is fixed and we know the claim for all matrices of size $(n-1) \times (n-1)$. 
Let $a_{ij}$ denote the $(i,j)$-th entry of $A$. 
Exploiting the induction hypothesis for the lower right $(n-1)\times(n-1)$-minor of $A$, we may assume that $a_{ij} \in B_{[0,1]}^R[\frac{1}{\pi}]$ for all $i>1$. 
Let now $1< j\leq n$. 
Suppose, by induction, that for all $1<j'<j$, one has $a_{1j'} \in B_{[0,1]}[\frac{1}{\pi}]$. 
It suffices to find, in this situation, a unipotent upper triangular matrix $U \in \GL_n(B_{[1,\infty]}^R)$ such that $U^{-1}A\varphi(U)$ has all the above properties of $A$ and additionally its $(1,j)$-th entry lies in $B_{[0,1]}^R[\frac{1}{\pi}]$. 
Therefore, write $a_{1j} = a_{1j}^{\rm mer} + a_{1j}'$ with some $a_{1j}^{\rm mer} \in B_{[0,1]}^R[\frac{1}{\pi}]$ and $a_{1j}' \in [\varpi]B_{[1,\infty]}^R$, according to \Cref{lm:decomposition_into_mero_and_varpi_part}. 
By \Cref{lm:Frob_minus_res_surjective}, there exists some $y \in B_{[1,\infty]}^R$ with $\psi_{s_1 - s_j}(y) = a_{1j}'$ (here we use that $s_j \leq s_1$). 
Let $U = (U_{\ell m})_{\ell m} \in \GL_n(B_{[1,\infty]}^R)$ be such that $U_{\ell m} = \delta_{\ell m}$ (where $\delta$ denotes the Kronecker-delta), unless $(\ell,m)=(1,j)$ and $U_{1 j} = y$. 
Then it is immediate to compute that $U^{-1}A\varphi(U)$ satisfies all the claimed conditions.
\end{proof}

\begin{proposition}
	\label{semistablefiltrationsplit-isoc}
	The forgetful functor $\Fil^\gamma_{\sss}\to \IC^\lozenge$ factors through $(\IC^\lozenge)^{\on{loc}}$ and defines a $\otimes$-exact equivalence
	\[\Fil^\gamma_{\sss}\to (\IC^\lozenge)^{\on{loc}}\,.\]
\end{proposition}
\begin{proof}
	On points, any filtration splits since the category of isocrystals is semi-simple. 
	In particular, the Newton polygon can be computed on the graded pieces.
	By the definition of semi-stable filtrations the Newton polygon is constant on the graded isocrystal.

	We now prove fully-faithfulness. Let $\{\calE_{\leq r}\}_{r\in \bbQ}$ and $\{\calF_{\leq r}\}_{r\in \bbQ}$ be two semi-stable filtrations with underlying analytic isocrystals $\calE$ and $\calF$.
	Let $\calH$ denote the $\underline{\on{Hom}}$-bundle endowed with its induced semi-stable filtration $\{\calH_{\leq r}\}_{r\in \bbQ}$. 
	We need to show that 
	\[\on{Hom}_{\IC^\lozenge}(\calO, \calH)=\on{Hom}_{\IC^\lozenge}(\calO, \calH_{\leq 0}).\]

	By applying induction with respect to $\{r\in \bbQ\mid \calH_{\leq r}/\calH_{<r}\neq 0\}$ one can show that 
	\[\on{Hom}_{\IC^\lozenge}(\calO, \calH_{\leq r}/\calH_{\leq 0})=0\]
	for all $r>0$ since the graded pieces all have slope larger than $0$.

	Since essential surjectivity can be proved v-locally by Lemma \ref{lm:ess_surjectivity_local}, it suffices to show that the standard objects can be endowed with a semi-stable filtration, but this is clear.
\end{proof}

\begin{proposition}
	\label{meromorphicbundles-remember-fils}
	The forgetful functor $\Fil^\mer_{\sss}\to \VEC$ factors through $(\VEC)^{\on{loc}}$ and defines a $\otimes$-exact equivalence 
	\[\Fil^\mer_{\sss}\to (\VEC)^{\on{loc}}\,.\]
\end{proposition}
\begin{proof}
It is automatic that the map respects the monoidal structure and exactness, since they are defined in terms of those of $\VEC$. 	
It follows from \Cref{semistablefiltrationsplit-isoc} that the map factors through $(\VEC)^{\on{loc}}$.
To show fully-faithfulness, we may again pass to $\underline{\on{Hom}}$-bundles $\calH$ with semi-stable filtration $\{\calH_{\leq r}\}_{r\in \bbQ}$ as in the proof of \Cref{semistablefiltrationsplit-isoc}.
We need to show that
	\[\on{Hom}_{\VEC}(\calO, \calH)=\on{Hom}_{\VEC}(\calO, \calH_{\leq 0}),\]
but as in the proof of \Cref{semistablefiltrationsplit-isoc} we can prove that $\on{Hom}_{\VEC}(\calO, \calH_{\leq r}/\calH_{\leq 0})=0$ for all $r>0$.

Essential surjectivity can now be proved v-locally by \Cref{lm:ess_surjectivity_local}, hence it suffices to show that every isoshtuka $\calE\in (\SHT[\frac{1}{\pi}])^{\on{loc}}(S)$ can be endowed with a semi-stable filtration. 
In other words, we need to show that the unique semi-stable filtration of $\gamma(\calE)$ lifts to a filtration in $\VEC$. 
Replacing $E$ by its degree $s$ field extension $E_s$, as in the proof of \Cref{two-types-of-semistable-fils}, and since we have already proved fully-faithfulness, we may assume that the generic Newton polygon only takes values in $\bbZ$ (see also \Cref{lemma_unramifieds}).
Twisting by a line bundle, we may even assume that the smallest slope $\calE$ is $0$. 
We can now apply \Cref{Lifting-to-antieffective} and \Cref{key-lemma-filtration-extends} to find a sub-bundle $\calO^k\subseteq \calE$, where $k$ is the rank of $\gamma(\calE)_0$ and such that $\gamma(\calE)/\gamma(\calO^k)$ has all slopes greater than $0$. 
By induction on the rank, $\calE/\calO^k$ can be endowed with a semi-stable filtration $\{(\calE/\calO^k)_{\leq r}\}_{r\in\bbQ}$ and we can lift this filtration to $\calE$.
\end{proof}

\section{$G$-bundles with meromorphic structure}
    \subsection{$\calG$-structure}
    Let $\calG$ be a smooth affine group scheme over $\Spec O_E$. We denote by $G$ its generic fiber over $\Spec E$, which we assume to be reductive.
    Later on, we will assume that $\calG$ is a parahoric group scheme. 
    We denote by $\Rep_\calG$ (resp.\ $\Rep_G$) the Tannakian category of algebraic representations of $\calG$ over $O_E$ (resp.\ of $G$ over $E$).
    
    \begin{definition}
	    We denote by $\DMG\in \calP(\PSch\!,\Grps)$ the presheaf valued in groupoids defined by
    \[S\mapsto \on{Fun}_{\on{ex}}^\otimes(\Rep_\calG,\DM(S))\,,\]
    where $\on{Fun}_{\on{ex}}^\otimes$ denotes the $\otimes$-compatible $O_E$-linear exact functors.
  	Analogously, we denote by $\ICG\in \calP(\PSch\!,\Grps)$ the presheaf valued in groupoids with 
  \[S\mapsto \on{Fun}_{\on{ex}}^\otimes(\Rep_G,\IC(S))\,.\]
    \end{definition}

    Recall the loop group and positive loop group functors $LG,L^+\calG\colon\PSch\to \Sets$ given on affine schemes $S=\Spec A$ by the formulas
    \[LG(S)\coloneqq G(\bbW(A)[\frac{1}{\pi}])\]
    and 
    \[L^+\calG(S)\coloneqq \calG(\bbW(A))\,.\]
	We let $LG$ and $L^+\calG$ act on $LG$ by $\varphi$-conjugation. 
    \begin{proposition}
    $LG$ and $L^+\calG$ are arc-sheaves.	
    \end{proposition}
    \begin{proof}
    Note that $L^+\calG$ is a scheme, $LG$ is an ind-scheme (see \cite[Section 3]{Ivanov_arc_descent}) and that the arc-topology is subcanonical (in fact, canonical) on perfect $\bbF_p$-schemes by \cite[Theorem 5.16]{BhattMathew_21}. 
	Now the claim follows directly for $L^+\calG$, and it also follows for $LG$ since filtered colimits of sets commute with finite limits.
    \end{proof}

    \begin{proposition}\label{DMisvstack}
	    The following statements hold:
	    \begin{enumerate}
		    \item $\DMG$ and $\ICG$ are scheme-theoretic small v-stacks.
		    \item The natural maps $LG \to \DMG$ and $LG \to \ICG$ are v-covers. 
		    \item We have identities $\DMG=[LG/\!\!/_{\!\!\varphi} L^+\calG]$ and $\ICG=[LG/\!\!/_{\!\!\varphi} LG]$. 
	    \end{enumerate}
    \end{proposition}
    \begin{proof}
	    The first statement follows from \Cref{formal-properties-IC-DM}. 
	    Indeed, $\on{Fun}_{\on{ex}}^\otimes(\Rep_\calG,-)$ (resp.\ $\on{Fun}_{\on{ex}}^\otimes(\Rep_G,-)$) is simply the representable presheaf in the category $\CatexOE$ (resp.\ $\CatexE$) and this preserves the condition of being a v-sheaf.

	    To prove surjectivity of $LG \to \DMG$ in the v-topology, it suffices by \Cref{rem:scheme_v_covers} to show that for a product comb $\Spec R$, any $\calG$-torsor on $\Spec \bbW R$ is trivial. 
		Such a torsor becomes trivial after some \'etale cover of $\Spec \bbW R$, so it suffices to show that any \'etale cover of $\Spec \bbW R$ splits. 
		As $\Spec \bbW R$ is henselian along the closed subscheme $\Spec R$, this follows from the same statement for $\Spec R$, which holds true by Remark \ref{rem:scheme_v_covers}. 
		The surjectivity of $LG \to \ICG$ in the v-topology follows from \Cref{triviality-on-combs} (see \cite[Theorem 6.1]{Ivanov_arc_descent} for the case $G=\GL_n$).
	    The last claim follows directly from the second one by computing the fiber products $LG\times_{\DMG}LG$ and  $LG\times_{\ICG}LG$.
    \end{proof}

    The following slight generalization of \cite[Theorem 11.5]{Ans18} will be useful for our purposes.

    \begin{lemma}
	    \label{triviality-on-combs}
	    If $\Spec(A)$ is a comb, then every $G$-torsor over $\Spec \bbW(A)[\frac{1}{\pi}]$ is trivial.
    \end{lemma}
    \begin{proof}
	    We can follow the proof of \cite[Proposition 11.5]{Ans18}, by noting that the reduction method in \cite[Section 6.1.1]{Ivanov_arc_descent} (which is also used in \cite[Proposition 11.5]{Ans18}) works for general combs.
    \end{proof}

    \begin{definition}
	    \label{defi-groupoids-of-interest}
	    We define the following four presheaves over $\Perf$ with values in groupoids: 	
	    \begin{enumerate}
		    \item $\Sht_{\calY,\calG} \colon S\mapsto \on{Fun}_{\on{ex}}^\otimes(\Rep_\calG,\SHT(S))$.
		    \item $\Isoc_G  \colon S\mapsto \on{Fun}_{\on{ex}}^\otimes(\Rep_\calG,\IC^\lozenge(S))$.
		    \item $\Bun^\mer_G  \colon S\mapsto \on{Fun}_{\on{ex}}^\otimes(\Rep_\calG,\VEC(S))$.
		    \item $\DmG  \colon S\mapsto \on{Fun}_{\on{ex}}^\otimes(\Rep_\calG,\DMan(S))$.
	    \end{enumerate}
    \end{definition}
    \begin{theorem}
	    \label{theorem-Isoc-is-what-it-is}
    The following statements hold:	
    \begin{enumerate}
	    \item $\Sht_\calG$, $\DmG$, $\Isoc_G$ and $\Bun^\mer_G$ are small v-stacks. 
	    \item We have a Cartesian diagram
		    \begin{center}
		    \begin{tikzcd}
		    \Sht_\calG \arrow{r} \arrow{d}  & \Bun^\mer_G \arrow{d} \\
		     \DmG\arrow{r} & \Isoc_G
		    \end{tikzcd}
		    \end{center}
			in $\calP(\Perf\!, \Grps)$.
	    \item We have identifications 
		    \[ \DmG=(\DMG)^\lozenge=[LG^\lozenge\!/\!\!/_{\!\!\varphi} L^+\calG^\lozenge]\] 
		    and 
		    \[\Isoc_G=(\ICG)^\lozenge=[LG^\lozenge\!/\!\!/_{\!\!\varphi} LG^\lozenge]\,.\] 
	    \item The maps $\DmG\to \Isoc_G$ and $\Sht_\calG\to \Bun^\mer_G$ are v-covers. 
    \end{enumerate}
    \end{theorem}
    \begin{proof}
	    Since the application $\on{Fun}_{\on{ex}}^\otimes(\Rep_\calG,-)$ (resp.\ $\on{Fun}_{\on{ex}}^\otimes(\Rep_G,-)$ commutes with $2$-limits within $\CatexOE$ (resp.\ $\CatexE$) and all of $\SHT$, $(\DM)^\lozenge$, $\IC^\lozenge$ and $\VEC$ are v-stacks in $\CatexOE$ or $\CatexE$, all of the presheaves of \Cref{defi-groupoids-of-interest} are v-sheaves.
	    For the same reason, the second claim follows directly from \Cref{Beauville-Laszlo-presheaves}.
	    Furthermore, $\on{Fun}_{\on{ex}}^\otimes(\Rep_\calG,-)$ commutes with sheafification by \Cref{funct-RepG}, which implies directly that $\DmG=(\DMG)^\lozenge$ and $\Isoc_G=\ICG^\lozenge$.
	    Since the functor $(-)^\lozenge$ commutes with finite limits, it suffices to prove that the maps $LG^\lozenge\to \DmG$ and $LG^\lozenge\to \Isoc_G$ are surjective to deduce the formulas from the third assertion.
	    Let $\calF\in \Isoc_G(S)$, the argument for $\DmG$ being analogous. 
	    By part (1) of the theorem surjectivity can be shown v-locally, so we may assume $S=\Spa(R,R^+)$ is a product of points.
	    By \Cref{DM1p-is-well-bheaved-subcategory}, we may even assume that for all objects $V\in \Rep_\calG$, the object $\calF(V)\in \IC^\lozenge(S)$ is isomorphic to one in $\DMan[\frac{1}{\pi}](S)$.
	    We obtain a $\otimes$-exact functor from $\Rep_\calG$ to the category of projective $\bbW(R)[\frac{1}{\pi}]$-modules, which we interpret as a $G$-torsor over $\Spec \bbW(R)[\frac{1}{\pi}]$.
	    By \Cref{triviality-on-combs}, such torsors are trivial over combs, and by \Cref{combs_and-products}, $\Spec R$ is a comb.
	    After choosing a trivialization of $\calF$, the $\varphi$-structure corresponds to an element $LG(\Spec R)$ which gives precisely a point $LG^\lozenge(S)$ lifting our original point.
	    The final claim follows by base change from the third claim and the second claim.
    \end{proof}

\subsection{Newton strata on $\Isoc_G$}\label{sec:newton_strata}
We now wish to study the geometry of $\Isoc_G$ and $\Bun_G^\mer$. 
Recall the Kottwitz set $B(G)$, which classifies isocrystals with $G$-structure over algebraically closed fields, see \cite[\S 3]{KottwitzII}. 
Recall that $B(G)$ is naturally endowed with a partial order, see for example \cite{RR_96} or \cite[\S 3]{Viehmann_20}.
For $G = \GL_n$, $\calN(G) = \calN$ with $\calN$ from \Cref{sec:semistable_filtrations}. 

\begin{definition}
	Let $\calS=\Spec A\in \PSch$, and let $b\in B(G)$. 
	We let $\ICG_{\leq b}(\calS)\subseteq \ICG(\calS)$ denote the full subcategory of objects $\calE\in \ICG(\calS)$ whose Newton polygon is bounded by $b$ at geometric points of $\calS$. 
	We let $\ICG_{b}(\calS)\subseteq \ICG_{\leq b}(\calS)$ denote the full subcategory of objects $\calE\in \ICG_{\leq b}(\calS)$ whose Newton polygon is exactly $b$ at geometric points of $\calS$. 
\end{definition}

The following theorem due to work of various authors summarizes what we will need about the geometry of $\ICG$.
\begin{theorem}
	\label{summary-of-properties-BG}
	 For any $b\in B(G)$ the map $\ICG_{\leq b}\to \ICG$ is a perfectly finitely presented closed immersion. Moreover, $\ICG_b=[\ast\!/\underline{G_b(\bbQ_p)}]$ as scheme-theoretic v-stacks.
\end{theorem}
\begin{proof}
The first statement follows from \cite[Theorem 3.6\,(ii)]{RR_96}, whose proof carries over the characteristic $p$ setting, see \cite[Theorem 7.3]{HartlViehmann}. 
The last statement follows from \cite[Theorem 2.11]{HamacherKim_22}.
\end{proof}

\begin{proposition}
	The elements of $|\Isoc_G|$ are in bijection with $B(G)$. 
\end{proposition}
\begin{proof}
	By definition, points in $|\Isoc_G|$ are in bijection with equivalence classes of $\Spa(C,C^+)$-valued points of $\Isoc_G$.
	By \Cref{lemma-sanity-check}, these are the same as isocrystals over $\Spec\bar{\bbF}_q$ with $G$-structure which are classified by $B(G)$, see \cite[\S 3]{KottwitzII}.
\end{proof}


\begin{definition}
	Let $S=\Spa(R,R^+) \in \Perf$. We let $\Isoc_G^{\leq b}(S)\subseteq \Isoc_G(S)$ denote the full subcategory of objects $\calE\in \Isoc_G(S)$ whose Newton polygon is bounded by $b$ at geometric points of $S$. 
	We let $\Isoc_G^{b}(S)\subseteq \Isoc_G^{\leq b}(S)$ denote the full subcategory of objects $\calE\in \Isoc_G^{\leq b}(\calS)$ whose Newton polygon is exactly $b$ at geometric points of $S$. 
\end{definition}

\begin{proposition}
	\label{structure-of-isoc}
	For any $b\in B(G)$, the map $\Isoc_G^{\leq b}\to \Isoc_G$ is a closed immersion and agrees with $\ICG_{\leq b}^\lozenge$. The map $\Isoc_G^b\to \Isoc_G^{\leq b}$ is an open immersion. Moreover, $\Isoc_G^b=\ICG_b^\lozenge=[\ast\!/\underline{G_b(\bbQ_p)}]$ as v-stacks. 	
\end{proposition}
\begin{proof}
	Since $\lozenge$ preserves open and closed immersions, it suffices to identify $\Isoc_G^{\leq b}$ and $\Isoc_G^b$ with $\ICG_{\leq b}^\lozenge$ and $\ICG_b^\lozenge$, respectively. 
	Let $S=\Spa(R,R^+)$. 
	By definition, $\Isoc_G^{\leq b}(S)$ is the subcategory of objects $\calE\in \Isoc_G(S)$ whose Newton polygon is pointwise bounded by $b$ at every geometric point $S$.
	On the other hand, $\ICG_{\leq b}^{\lozenge_{\on{pre}}}(S)$ corresponds to $G$-isocrystals over $\Spec R$ whose polygon is bounded by $b$ at every geometric point of $\Spec R$.
	To prove $\ICG_{\leq b}^\lozenge=\Isoc_G^{\leq b}$, it suffices to show that v-locally having a Newton polygon bounded by $b$ for $\Spa(R,R^+)$ or for $\Spec R$ agree.
	Of course, the scheme-theoretic condition is stronger than the analytic one, since on the analytic side a condition is imposed only on those ideals of $\Spec R$ that support a continuous valuation. 
	Now, over product of points the two conditions agree. 
	Indeed, principal connected components of a product of points support a continuous valuation.
	Moreover, these components are dense in $\Spec R$.

	A similar argument shows $\ICG_b^\lozenge=\Isoc_G^b$. Indeed, if $S$ is a product of points, all of the maximal ideals of $\Spec R$ support a continuous valuation and the map $\ICG_b\to \ICG_{\leq b}$ is open.  

The last claim follows directly from \Cref{quotient-by-groups}.
\end{proof}

\subsection{Newton strata on $\Bun_G^\mer$}
Recall the moduli stack $\calM$ of \cite[Definition V.3.2]{FS24}. 
The connected components of $\calM$ are indexed by $b\in B(G)$ and the maps $\calM_b\to \Bun_G$ are cohomologically smooth charts of $\Bun_G$, which uniformize an open substack containing $\Bun_G^b \hookrightarrow \Bun_G$. 
\begin{proposition}
	The v-stack $\calM$ is the moduli stack given by the formula
\[\calM \colon S\mapsto \on{Fun}^{\otimes}_{\on{ex}}(\on{Rep}_G,\Fil^\sigma_\sss(S))\,.\]
\end{proposition}
\begin{proof}
	It follows directly from the definition.	
\end{proof}

\begin{theorem}
	\label{maintheorem}
 The moduli stack $\calM$ fits in the  Cartesian diagram  
\begin{center}
\begin{tikzcd}
	\calM \arrow{r} \arrow{d}  & \Bun_G^\mer \arrow{d}{\gamma} \\
 \coprod_{b\in B(G)} \Isoc^b_G \arrow{r} & \Isoc_G
\end{tikzcd}
\end{center}
of small v-stacks.
\end{theorem}

\begin{remark}
	\label{rmrk-Zw}
	While this article was in preparation, we learned from a private communication with Z.\ Wu that he had proven independently a version of \Cref{maintheorem} in the language of relative Robba rings.
\end{remark}

\begin{proof}
	Observe that we have the identification
	\[ \coprod_{b\in B(G)}\Isoc^b_G(S) =\on{Fun}^{\otimes}_{\on{ex}}(\on{Rep}_G,(\IC^\lozenge)^{\on{loc}}(S))\,.\]
	Since $\on{Fun}^\otimes_{\on{ex}}(\on{Rep}_G,-)$ commutes with limits, it suffices to show that $\Fil^\sigma_\sss(S)$ fits in the Cartesian diagram
\begin{center}
\begin{tikzcd}
	\Fil^\sigma_\sss(S) \arrow{r} \arrow{d}  & \VEC(S) \arrow{d} \\
	(\IC^\lozenge)^{\on{loc}}(S) \arrow{r} & \IC^\lozenge(S)\,.
\end{tikzcd}
\end{center}
By definition, $(\Bun_{\on{FF}}^\mer)^{\on{loc}}$ fits as the upper-left entry of the above Cartesian diagram. 
By \Cref{two-types-of-semistable-fils} and \Cref{meromorphicbundles-remember-fils}, 
\[(\Bun_{\on{FF}}^\mer)^{\on{loc}}(S)\cong \Fil_{\on{ss}}^\mer \cong \Fil^\sigma_\sss(S)\,.\]
\end{proof}

\begin{corollary}
	Let $S=\Spa(R,R^+)$ and let $\calE\in \Bun_{\on{FF}}(S)$. The following hold:
	\begin{enumerate}
		\item After replacing $S$ by a v-cover, $\calE$ can be lifted to $\VEC(S)$.
		\item After replacing $S$ by a v-cover, $\calE$ can be lifted to $\SHT(S)$.
		\item The map of small v-stacks $\Bun_G^\mer\to \Bun_G$ is surjective.
		\item The map of small v-stacks $\Sht_\calG\to \Bun_G$ is surjective.
	\end{enumerate}
\end{corollary}
\begin{proof}
	The first and second claims are consequences of the third and fourth claim applied in the case $G=\on{GL}_n$.
	For the third claim, the map $\calM\to \Bun_{G}$ is formally and $\ell$-cohomologically smooth and surjects onto its image. 
	In particular, it is a surjection of small v-stacks.
	The result follows since this map factors through $\Bun_G^\mer\to \Bun_G$.
	The fourth claim follows from \Cref{theorem-Isoc-is-what-it-is} (4) and the third claim.
\end{proof}

\begin{definition}
	\label{definition-the-M-space}
	Given two subsets $U_1,U_2\subseteq B(G)$
	we let $\calM^{\sigma\in U_2} _{\gamma\in U_1}$ denote $\gamma^{-1}(\Isoc_G^{U_1})\cap \sigma^{-1}(\Bun_G^{U_2})$.
	Whenever $U_i=B(G)$, we omit the subscript or superscript as an abbreviation.
\end{definition}

We will mostly use \Cref{definition-the-M-space} when $U_1$ or $U_2$ are given by Newton polygon inequalities. 
In this case, we use more intuitive notation: for example, $\calM^{\sigma=b} \coloneqq\sigma^{-1}(\Bun_G^b)$ and $\calM_{\gamma=b}\coloneqq\gamma^{-1}(\Isoc_G^b)=\calM_b$ (where the last equality follows from \Cref{maintheorem}). 

\section{Three comparison theorems}
\label{section-3-compar}
\subsection{The meromorphic comparison}
Let $C$ be a non-Archimedean algebraically closed field and let $S=\Spa(C,O_C)$.
One interesting consequence of the classification theorem of vector bundles on the Fargues--Fontaine curve is that every such vector bundle extends at $\infty$, i.e.\ it is isomorphic to one obtained from a $\varphi$-module over $Y_{(0,\infty],S}$ for such $S$.
In what follows, we will prove that this statement holds in families when one is allowed to work v-locally.
\begin{definition}
	Let $S=\Spa(R,R^+)\in \Perf$ and $T=\Spa(R,R^\circ)$.
	\begin{enumerate}
		\item We let $\Bun^+_{\on{FF}}(S)\in \calP(\Perf\!,\CatexE)$ be given by the rule that attaches to $S$ the category of pairs $(\calE,\Phi)$, where $\calE$ is a vector bundle over $Y_{(0,\infty],T}$ and $\Phi\colon \varphi^*\calE\to \calE$ is an isomorphism.
		\item We say that $\calE\in \Bun_{\on{FF}}(S)$ \textit{extends} at $\infty$ if it is in the essential image of the map $\Bun^+_{\on{FF}}(S)\to \Bun_{\on{FF}}(T)\cong \Bun_{\on{FF}}(S)$.
		\item We denote by $(\DM)^{\dagger_{\on{pre}}}\in \calP(\Perf\!,\CatexOE)$, the presheaf given by the rule  
		\[(R,R^+)\mapsto \DM(\Spec R^\circ).\]
	\item We say that $\calE\in \SHT(S)$ is a \textit{BKF-shtuka} if it is in the essential image of the map $(\DM)^{\dagger_{\on{pre}}}(S)\to \SHT(T)\cong \SHT(S)$.
	\item We denote by $(\DM[\frac{1}{\pi}])^{\dagger_{\on{pre}}}\in \calP(\Perf\!,\CatexE)$, the presheaf given by the rule  
		\[(R,R^+)\mapsto \DM[\frac{1}{\pi}](\Spec R^\circ)\,.\]
	\item We denote by $\IC^{\dagger_{\on{pre}}}\in \calP(\Perf\!, \CatexE)$, the presheaf given by the rule  
		\[(R,R^+)\mapsto \IC(\Spec R^\circ)\,.\]
	\end{enumerate}


\end{definition}

We can pass to $\calG$-objects for all the above.

\begin{definition}
	\label{def-extend-G-atinfty}
	Let $S=\Spa(R,R^+)\in \Perf$ and $T=\Spa(R,R^\circ)$.
	\begin{enumerate}
		\item We let $\Bun^+_{G}(S)\in \calP(\Perf\!,\Grps)$ be given by $\on{Fun}^\otimes_{\on{ex}}(\on{Rep}_G,\Bun^+_{\on{FF}})$. 
		\item We say that $\calE\in \Bun_{G}(S)$ \textit{extends} at $\infty$ if it is in the essential image of the map $\Bun^+_{G}(S)\to \Bun_{G}(S)$.
		\item We denote $(\DMG)^{\dagger_{\on{pre}}}=\on{Fun}^\otimes_{\on{ex}}(\on{Rep}_\calG,(\DM)^{\dagger_\pre})$.
	\item We say that $\calE\in \Sht_\calG(S)$ is a \textit{BKF-$\calG$-shtuka} if it is in the essential image of the map $(\DMG)^{\dagger_{\on{pre}}}(S)\to \Sht_\calG(S)$.
	\item We let $(\DMG[\frac{1}{\pi}])^{\dagger_{\on{pre}}}\in \calP(\Perf\!,\Grps)$ be given by $\on{Fun}^\otimes_{\on{ex}}(\on{Rep}_G,(\DM[\frac{1}{\pi}])^{\dagger_\pre})$. 
	\item We let $\ICG^{\dagger_{\on{pre}}}\in \calP(\Perf\!, \Grps)$ be given by $\on{Fun}^\otimes_{\on{ex}}(\on{Rep}_G,(\IC)^{\dagger_\pre})$.
	\end{enumerate}


\end{definition}

\begin{remark}
	\label{remark-Faruges-theorem}
	The purpose of this section is to formulate and prove \Cref{meromorphic-comparison-theorem}.
	This result can be regarded as a version of Fargues' theorem \cite[Theorem 1.12]{Fargues_courbe} in families. 
	Recall that Fargues' theorem states that the category of shtukas over $(C,O_C)$ is equivalent to the category of BKF-modules of $\bbW(O_C)$.
	Although this statement is not true for general families, \Cref{meromorphic-comparison-theorem} below shows that the statement is true v-locally. 
	Indeed, $\Sht_\calG(R,R^+)$ parametrizes $\calG$-shtukas over $\Spa(R,R^+)$ while $(\DMG)^\dagger$ is the sheafification of the functor attaching to $(R,R^+)$ the category of BKF-modules with $\calG$-structure over $\bbW(R^\circ)$.
\end{remark}

\begin{proposition}
	\label{fully-faithful-from-inftytobeyond}
	Let $S=\Spa(R,R^+)\in \Perf$, the following hold.
	\begin{enumerate}
		\item The map $\Bun^+_{\on{FF}}(S)\to \Bun_{\on{FF}}(S)$ is exact and fully-faithful.  
		\item The map $\Bun^+_{G}(S)\to \Bun_{G}(S)$ is fully-faithful.  
		\item If $S$ is a product of points, the diagrams
		\begin{center}
		\begin{tikzcd}
			(\DMG)^{\dagger_{\on{pre}}}(S) \arrow{r} \arrow{d}  & \DMG[\frac{1}{\pi}]^{\dagger_{\on{pre}}}(S) \ar{r} \ar{d} &\Bun^+_{G}(S)  \arrow{d} \\
			\Sht_\calG(S)	\arrow{r} & \Sht_\calG[\frac{1}{\pi}](S) \ar{r} & \Bun_{G}(S) 
		\end{tikzcd}
		\end{center}
		are Cartesian in $\Grps$.
		\item If $S$ is a product of points then $\DMG[\frac{1}{\pi}]^{\dagger_{\on{pre}}}(S)\cong \ICG^{\dagger_{\on{pre}}}(S)$.
		\item The sheafification of $\DMG[\frac{1}{\pi}]^{\dagger_{\on{pre}}}$ is $\ICG^\dagger$.
	\end{enumerate}
\end{proposition}
\begin{proof}
	The first claim is \cite[Proposition 2.1.3]{PR21}, whose proof generalizes to general $E$ (see \cite[Remark 2.1.10]{PR21}). 
	The second claim follows formally by passing to $\on{Fun}^\otimes_{\on{ex}}(\on{Rep}_G,-)$.
	For the third claim, note that by Kedlaya's GAGA \cite[Theorem 3.8]{Kedlaya_16_ringAinf}, we can identify the category $\Sht_\calG(S) \times_{\Bun_{G}(S)} \Bun^+_{G}(S)$ with the category of $\calG$-bundles over $\Spec\bbW(R^\circ)\setminus (\{\pi =0\} \cap \{[\varpi] = 0\})$ together with $\varphi$-action defined over $\Spec \bbW(R^\circ)[\frac{1}{\pi}]$.
	We note that to carry \cite[Theorem 3.8]{Kedlaya_16_ringAinf} over to the equal characteristic setting, it suffices to generalize \cite[Proposition 3.6]{Kedlaya_16_ringAinf} to ramified Witt vectors, which is straightforward since the proof strategy works in this generality.
	As $S$ is a product of points, by \cite[Theorem 1.1]{Ans18}) (see \cite[Proposition 2.1.17]{Gle22v2}), any such $\calG$-bundle extends uniquely to a $\calG$-bundle over $\Spec \bbW(R^\circ)$. 
	This proves that the outer diagram is Cartesian. 
	Moreover, the same argument also applies to the right square, so this is Cartesian as well.
	It then follows that the left square is Cartesian.

	For the fourth claim, write $S = \Spa(R,R^+)$. 
	In the case that $\calG=\on{GL}_n$, we need to show that any isocrystal $\calE$ over $\bbW(R^\circ)[\frac{1}{\pi}]$ contains a $\bbW(R^\circ)$-lattice. 
	As $S$ is a product of points, \Cref{combs_and-products} and \cite[Theorem 6.1]{Ivanov_arc_descent} imply that $\calE$ is free as a $\bbW(R^\circ)[\frac{1}{\pi}]$-module, but then an $\bbW(R^\circ)$-lattice obviously exists. 
	For more general $\calG$ this follows from \Cref{triviality-on-combs}.
	The fifth claim follows from the fourth since product combs are a basis for the v-topology.
\end{proof}
\begin{remark}
We warn the reader that the maps $(\DM)^{\dagger_{\on{pre}}}(S) \to \SHT(S)$ and $\Bun^+_{\on{FF}}(S) \to \Bun_{\on{FF}}(S)$ do not reflect exactness. 
For this reason one crucially relies on the in-depth analysis of $\calG$-torsors for parahoric $\calG$, see \cite[Theorem 1.1]{Ans18}.
\end{remark}
The advantage of working with $(\DM)^{\dagger_{\on{pre}}}$ is that its values on product of points are easy to describe.
\begin{proposition}
	\label{descriptionDMdaggeronproductofpoints}
	Let $S=\Spa(R,R^+)$ be a product of points with $R^+\!=R^\circ\!=\prod_{i\in I}\!O_{C_i}$, then the restriction functor
	\[(\DM)^{\dagger_{\on{pre}}}(S)\to \prod_{i\in I}\DM(\Spec O_{C_i})\]
	is fully-faithful, and its essential image is the collection of families of $\{(\calE_i,\Phi_i)\}_{i\in I}$ with uniformly bounded zeros and poles on $\pi$.
\end{proposition}
\begin{proof}
The fully-faithful functor is induced by the isomorphism $\bbW(\prod \!O_{C_i}) = \prod \bbW(O_{C_i})$. The pole (resp. zero) at each $i \in I$ of any object in the essential image is bounded by the pole (resp. zero) of its preimage. Conversely, if we have a uniform bound, then the Frobenius is represented by a matrix with entries in $\bbW(R^\circ)[\frac{1}{\pi}] = (\prod \bbW(O_{C_i}))[\frac{1}{\pi}] \subseteq \prod (\bbW(O_{C_i})[\frac{1}{\pi}])$, whose inverse also has entries in this subring. 
\end{proof}

Moreover, at the level of geometric points $\SHT$ is also easy to describe. 
Indeed, the following is the $\pi=\xi$ version of Fargues' theorem \cite[Theorem 14.1.1]{SW20}.
\begin{proposition}
	\label{concrete-shtukas-atpoints}	
	Let $C$ be a non-Archimedean field, then the following categories are equivalent: 
	\begin{enumerate}
		\item BKF-modules with $\xi=\pi$. In other words, the category pairs $(M,\Phi)$, where $M$ is a free $\bbW(O_C)$-module and $\Phi\colon M[\frac{1}{\pi}]\to \bbW(O_C) _\varphi\otimes_{\bbW(O_C)} M [\frac{1}{\pi}]$ is an isomorphism.
		\item $(\DM)^{\dagger_{\on{pre}}}(C,C^+)$
		\item $\SHT(C,C^+)$.
	\end{enumerate}
\end{proposition}
\begin{proof}
	By definition $(\DM)^{\dagger_{\on{pre}}}(C,C^+)=\DM(O_C)$, which is precisely the category of BKF-modules with $\xi=\pi$, so the first two categories are the same category.
	The equivalence with the third category is given in \cite[$\S12$-$14$]{SW20} when $\xi\neq \pi$.
	The same proof strategy applies and generalizes to general $E$ (see the proof of \Cref{fully-faithful-from-inftytobeyond}).
\end{proof}

We wish to extend \Cref{concrete-shtukas-atpoints} to the case of products of points.
It will be profitable to work with the stack of shtukas that have their leg away from the trivial untilt.
Let us set some notation.

We let $S\in \Perf$ be of the form $S=\Spa(R,R^+)$.
Recall that an $S$-point $S \to \Spd O_E$ corresponds to an untilt $S^{\sharp}$ of $S$, which in turn corresponds to a degree $1$ Cartier divisor $\infty\colon S^\sharp\hookrightarrow \calY_S$, and that it factors through $\Spd E$ if and only if the Cartier divisor factors through $Y_S$.
In this circumstance, we let $\xi_{S^\sharp}\in \bbW(R^+)$ be a generator of the kernel of $\bbW(R^+)\to R^{\sharp,+}$, cf.\ \cite[Proposition 11.3.1]{SW20} and \cite[Proposition II.1.4]{FS24}.

\begin{definition}
	Let $S\in \Perf$ be of the form $S=\Spa(R,R^+)$.
	\begin{enumerate}
		\item A \emph{$\calG$-shtuka} is a triple $(S^\sharp,\calE,\Phi)$, where $\infty \colon S^\sharp \hookrightarrow \calY_S$ is an untilt of $S$ over $\Spd(O_E)$, $\calE$ is a $\calG$-bundle over $\calY_S$ and $\Phi_\calE$ is an isomorphism \[\Phi_\calE\colon (\varphi^*\calE)_{\calY_S\setminus \infty} \rightarrow  \calE_{\calY_S\setminus \infty}\]
		that is meromorphic (cf. \cite[Definition 5.3.5]{SW20}) along $\infty=\{\xi=0\}$. 
		\item A \emph{BKF-$\calG$-shtuka} with leg at $\infty$ is a triple $(S^\sharp, \calE,\Phi)$, where $\calE$ is a $\calG$-bundle over $\Spec \bbW(R^\circ)$ and $\Phi_\calE$ is an isomorphism 
		\[\Phi_\calE\colon (\varphi^*\calE) \rightarrow  \calE\]
		defined over ${\Spec \bbW(R^\circ)[\frac{1}{\xi_{S^\sharp}}]}$.
		\item We let $\Sht_{\calG,O_E}$ denote the v-stack of $\calG$-shtukas and we let $\Sht^+_{\calG,O_E}$ denote the prestack of BKF-$\calG$-shtukas.
	\end{enumerate}
\end{definition}

Observe that we have a map 
\[\Sht_{\calG,O_E}\to \Bun_G\]
that sends $(S^\sharp, \calE,\Phi)$ to the unique $\calG$-$\varphi$-module over $Y_S$ that agrees with $(\calE,\Phi)$ over $Y_{[r,\infty),S}$ for sufficiently large $r$.

\begin{proposition}
	\label{Cartesian-grps-untilt}
	If $S$ is a product of points we have a Cartesian diagram 
	\begin{center}
	\begin{tikzcd}
		\Sht^+_{\calG,O_E} (S) \arrow{r} \arrow{d}  & \Bun_G^+ (S) \arrow{d} \\
		\Sht_{\calG,O_E}(S) \arrow{r} & \Bun_G(S)
	\end{tikzcd}
	\end{center}
	in $\Grps$.
\end{proposition}
\begin{proof}
The same proof as in \Cref{fully-faithful-from-inftytobeyond} works in this generality.	
\end{proof}

Recall that the stack of shtukas $\Sht_{\calG,E} \coloneqq \Sht_{\calG, O_E} \times_{\Spd O_E} \Spd E$ admits a different description.

\begin{proposition}
	\label{Second-description}
	Let $S\in \Perf$ and consider the groupoid $\Sht_{\calG,E}(S)$.
	It is equivalent to the groupoid of tuples $(S^\sharp, T, \calF, \alpha)$, where $S^\sharp$ is an untilt of $S$ over $E$, $T$ is a quasi-pro-\'etale $\underline{\calG(O_E)}$-torsor over $S$, $\calF\in \Bun_G(S)$ and $\alpha\colon \calE_T\to \calF$ is a modification over $X_{\FF,S}$, where $\calE_T\in \Bun_G^1(S)$ is the unique $G$-bundle over $X_{\FF,S}$ specified by the identification $\Bun_G^1\simeq [\ast\!/\underline{G(E)}]$.
	Under this equivalence the map $\Sht_{\calG,E}\to \Bun_G$ is given by
	\[(S^\sharp, T, \calF, \alpha)\mapsto \calF\,.\]
	Moreover, the Beauville--Laszlo uniformization map $\on{BL} \colon \on{Gr}_{G,E} \to \on{Bun}_G$ factors as $\on{Gr}_{G,E} \to \on{Sht}_{\calG,E} \to \on{Bun}_G$.
\end{proposition}
\begin{proof}
This is \cite[Proposition 11.16]{Zha23}, whose proof carries over to general $E$.
\end{proof}

\begin{lemma}
	\label{extend-at-infty-direct}
	Let $S \in \Perf$ be a product of points.
	Then $\Sht^+_{\calG,E}(S)\to \Sht_{\calG,E}(S)$ is an equivalence.
\end{lemma}
\begin{proof}
Let $S = \Spa(R,R^+)$ with $R^+\!=\prod_{i\in I}\!C_i^+$. 
	By \Cref{Cartesian-grps-untilt}, the map is fully-faithful.
	We wish to show it is essentially surjective. 
	Fix $(S^\sharp,\calE,\Phi)\in \Sht_{\calG,E}(S)$, with $\xi \in \bbW(\prod_{i \in I} O_{C_i})$ a generator of the Fontaine map as in the definition and its projections $\xi_i \in \bbW(O_{C_i})$.
	From now on we will omit the untilt from the notation.
	After fixing an embedding $\calG\to \on{GL}_n$, we see that $\Phi$ has uniformly bounded zeroes and poles along $\infty$.

	Let $S_i=\Spa(C_i,C_i^+)$. 
	Over geometric points, $\Sht^+_{\calG,E}(S_i)\simeq \Sht_{\calG,E}(S_i)$ by \cite[Theorem 1.1]{Ans18}.
	Let $(\calE_i,\Phi_i)$ be the BKF-$\calG$-shtuka obtained from $(\calE,\Phi)$ after restricting to $S_i$.
	We may find a trivialization $\calG\simeq \calE_i$ and by transfer of structure obtain a matrix $M_i\in \calG(\bbW(O_{C_i})[\frac{1}{\xi_i}])$. 
	We consider the product BKF-$\calG$-shtuka $(\calE_\infty,\Phi_\infty)$ obtained from the product matrix $M_\infty=\prod_{i\in I}\! M_i\in \prod_{i\in I}\!\calG(\bbW(O_{C_i})[\frac{1}{\xi_i}])$. 
	Since we have uniform bounds on poles and zeroes of $\xi$, this matrix lies in $\calG([\prod_{i\in I}\!\bbW(O_{C_i})][\frac{1}{\xi}])$.

	We claim that $(\calE_\infty,\Phi_\infty)$ is isomorphic to $(\calE,\Phi)$. 
	Let $\calI$ denote the groupoid of isomorphisms between $(\calE,\Phi)$ and $(\calE_\infty,\Phi_\infty)$, we claim that $\calI$ is proper and quasi-pro-\'etale over $S$. 

	Indeed, we reinterpret both shtukas in terms of \Cref{Second-description}, so that we have tuples $(T, \calF, \alpha)$ and $(T_\infty, \calF_\infty, \alpha_\infty)$.
	An isomorphism of this data consists of a pair $(\Theta_1,\Theta_2)$, where $\Theta_1\colon T\to T_\infty$ is an $\underline{\calG(O_E)}$-equivariant isomorphism while $\Theta_2\colon \calF\to \calF_\infty$ is an isomorphism in $\Bun_G(S)$. This data should fit in the following commutative diagram
	\begin{center}
	\begin{tikzcd}
		\calE_T \arrow{r}{\Theta_1} \arrow[d, dashed, "\alpha"]  & \calE_{T_\infty} \arrow[d, dashed, "\alpha_\infty"]  \\
	\calF \arrow{r}{\Theta_2} & \calF_\infty \,,
	\end{tikzcd}
	\end{center}
	where the solid arrows denote maps over $X_\FF$ while the dashed arrows denote modifications.
	Notice that $\Theta_1$ determines $\Theta_2$ (if they exist) and vice versa. 

	We make two observations. 
	Since $S$ is a product of points, every quasi-pro-\'etale $\underline{\calG(O_E)}$-torsor is trivial by \Cref{quasi-pro-etale-admitsection}.
	This implies that the space of isomorphism between $T$ and $T_\infty$ is isomorphic to $\underline{\calG(O_E)}_S$ and in particular it is proper and quasi-pro-\'etale over $S$ since $\calG(O_E)$ is a pro-finite group.
	The second observation is that any isomorphism $\Theta_1\colon T\to T_\infty$ gives rise to a unique commutative diagram 
	\begin{center}
	\begin{tikzcd}
		\calE_T \arrow{r}{\Theta_1} \arrow[d, dashed, "\alpha"]   & \calE_{T_\infty} \arrow[d, dashed, "\alpha_\infty"]  \\
	\calF \arrow[r, dashed, "\Theta_2"]  & \calF_\infty \,.
	\end{tikzcd}
	\end{center}
	Indeed, $\Theta_2=\alpha_\infty\circ \Theta_1 \circ \alpha^{-1}$. 
	Moreover, after fixing trivializations $\underline{\calG(O_E)}_S\simeq T$ and $\underline{\calG(O_E)}_S\simeq T_\infty$,
	we have a map 
	\[\underline{\calG(O_E)}_S\to \on{Hck}^{\on{loc}}_{\calG,S^\sharp} \quad \text{   with   } \quad  \Theta_1\mapsto [\calF^\loc\xrightarrow{\Theta_2} \calF^\loc_\infty] \]
	to the local Hecke stack of \cite[Definition VI.1.6]{FS24}. 
	$\Theta_1$ defines an isomorphism if and only if the induced map $\calF^\loc\xrightarrow{\Theta_2} \calF^\loc_\infty$ is an isomorphism instead of merely a modification. 
	Since this is a closed condition within the local Hecke stack, the v-sheaf $\calI$ of isomorphisms between $(\calE,\Phi)$ and $(\calE_\infty,\Phi_\infty)$ is a closed subsheaf of $\underline{\calG(O_E)}_S$.
	In particular, $\calI$ is proper and quasi-pro-\'etale over $S$.

	To show that $(\calE_\infty,\Phi_\infty)$ is isomorphic to $(\calE,\Phi)$ it suffices to show that $\calI\to S$ admits a section.
	We note that it is v-surjective since the map is qcqs and $|\calI|\to |S|$ is surjective. 
	Indeed, it is a closed map which by construction maps to every principal component of $\pi_0(S)$ and these are dense in $S$.
	Using \Cref{quasi-pro-etale-admitsection} again, we observe that the map $\calI\to S$ admits a section. 
\end{proof}

\begin{corollary}
	\label{lemma-extending-at-infty}
	Let $S\in \Perf$ be a product of points.
	Then, the following statements hold:
	\begin{enumerate}
		\item The map $\Bun^+_G(S)\to \Bun_G(S)$ is an equivalence. 
		\item The map $\Sht^+_{\calG,O_E}(S)\to \Sht_{\calG,O_E}(S)$ is an equivalence of groupoids.
	\end{enumerate}
\end{corollary}
\begin{proof}
	In light of \Cref{Cartesian-grps-untilt} it suffices to show the first statement.	
	Recall from \cite[\S III.3]{FS24} the Beauville--Laszlo uniformization map
	\[\on{BL}\colon \on{Gr}_{G,E}\to \Bun_{G}\]
\[\on{BL}(\alpha_{\on{dR}}\colon \calE_{\on{dR}}\to G)\coloneqq\alpha_{X_\FF}\colon \calE_{X_\FF}\to G\]
that sends a modification of the trivial $G$-torsor over $B_{\on{dR}}$ to a modification of the trivial $G$-torsor over $X_\FF$.
	This map is a pro-\'etale surjection by \cite[Proposition III.3.1]{FS24}. 

	Recall that under the notation of \Cref{Second-description}, this map factors as 
	\[\on{Gr}_{G,E}\to \Sht_{\calG,E}\to \Bun_G\,.\]
	By \Cref{quasi-pro-etale-admitsection}, any map $S\to \Bun_G$ lifts to a map $S\to \on{Gr}_{G,E}$ which induces $S\to \Sht_{\calG,E}$. 
	By \Cref{extend-at-infty-direct}, this comes from an object in $\Sht^+_{\calG,E}$ which induces an element in $\Bun_G^+$ as we wanted to show. 
\end{proof}

\begin{theorem}
	\label{meromorphic-comparison-theorem}
	The following statements hold:
	\begin{enumerate}
		\item We have an isomorphism of small v-stacks $(\DMG)^{\dagger}\cong \Sht_\calG$.
\item We have an isomorphism of small v-stacks $\ICG^{\dagger}\cong \Bun_G^\mer$.
\item The maps $(\DMG)^\diamond\to \Sht_\calG$ and $\ICG^\diamond\to \Bun_G^\mer$ are v-surjective. 
	\end{enumerate}
\end{theorem}

\begin{proof}
	\Cref{fully-faithful-from-inftytobeyond} shows that $(\DMG)^{\dagger_{\on{pre}}}(S) \to \Sht_\calG(S)$ is fully-faithful and \Cref{lemma-extending-at-infty} shows that it is v-locally surjective since it is an equivalence on products of points. 
	This proves the first claim.
	For the second claim, consider the map  
	\[\DMG[\frac{1}{\pi}]^{\dagger_{\on{pre}}}(S)\to \Sht_\calG[\frac{1}{\pi}](S)\,.\]
	Arguing as above, we see that it is fully-faithful in general and an equivalence when $S$ is a product of points.
	In particular, their v-sheafifications agree. 
	Nevertheless, by \Cref{fully-faithful-from-inftytobeyond} (5), the v-sheafification of $\DMG[\frac{1}{\pi}]^{\dagger_{\on{pre}}}$ is $\ICG^\dagger$ while the v-sheafification of $\Sht_\calG[\frac{1}{\pi}]$ is $\Bun_G^\mer$. 

	For the third claim, it suffices to prove that $(\DMG)^\diamond \to \Sht_\calG$ is surjective since $\Sht_\calG\to \Bun_G^\mer$ is surjective and the map $(\DMG)^\diamond\to \Bun_G^\mer$ factors through $\ICG^\diamond$. 
	Since $\Sht_\calG\cong (\DMG)^\dagger$, it suffices to prove that $\calE\in (\DMG)^{\dagger_{\on{pre}}}(S)$ lifts to an object in $(\DMG)^{\diamond_{\on{pre}}}(S')$ for some v-cover $S'\to S$. 
	We can reduce this to the case where $S=\Spa(R,R^+)$ is a product of points with $R^+=\prod_{i\in I}\!C_i^+$, and $\calE$ is given by a matrix $M\in \calG(\bbW(\prod_{i\in I}\! O_{C_i})[\frac{1}{\pi}])$.
	Any $\varphi$-conjugation by a matrix $N\in\calG(\bbW(\prod_{i\in I}\! O_{C_i}))=\prod_{i\in I}\!\calG(\bbW(O_{C_i}))$ defines an isomorphic object in $(\DMG)^{\dagger_{\on{pre}}}(S)$. 
	This allows us to reduce to the case where the set $I$ is a singleton and we must show that $M$ is $\varphi$-conjugate by some $N \in \calG(\bbW(O_C))$ to a matrix $M'\in\calG(\bbW(C^+)[\frac{1}{\pi}])$. 
	Let $k=O_C/C^{\circ \circ}$ and $k^+=C^+\!/C^{\circ \circ}$ be the residue rings. 
	As $\bbW(C^+)[\frac{1}{\pi}]$ is the preimage of $\bbW(k^+)[\frac{1}{\pi}]$ under $\bbW(O_C)[\frac{1}{\pi}] \to \bbW(k)[\frac{1}{\pi}]$, it suffices to prove the above claim after replacing $O_C$ and $C^+$ by $k$ and $k^+$, respectively. 
	Indeed, if $\overline{M'}\in \calG(\bbW(k^+)[\frac{1}{\pi}])$ can be conjugated to $\overline{M}\in \calG(\bbW(k)[\frac{1}{\pi}])$ through a matrix $\overline{N}\in \calG(\bbW(k))$, then any lift of $\overline{N}$ to $N\in \calG(\bbW(O_C))$ would give the desired conjugation.

	The claim is now that for every $\overline{M}\in \calG(\bbW(k)[\frac{1}{\pi}])$ there is some $\overline{N}\in \calG(\bbW(k)[\frac{1}{\pi}])$ such that
	\[\overline{N}^{-1}\overline{M}\varphi(\overline{N})\in \calG(\bbW(k^+)[\frac{1}{\pi}])\,.\]
	Note that $k$ is an algebraically closed field, so that by the classification of isocrystals
	there are $A\in \calG(\bbW(k)[\frac{1}{\pi}])$ and $T\in \calG(\bbW(\overline{\bbF}_q)[\frac{1}{\pi}])$
	such that $A^{-1} T \varphi(A)=\overline{M}$.
	By ind-properness of the Witt vector affine Grassmannian, we have $A=B\cdot C$ for some $B\in \calG(\bbW(k^+)[\frac{1}{\pi}])$ and $C\in \calG(\bbW(k))$. Letting $\overline{N}=C^{-1}$ we get
	\[B^{-1}T\varphi(B)=\overline{N}^{-1} \overline{M} \varphi(\overline{N})\,.\]
	Now $B^{-1}T\varphi(B)\in \calG(\bbW(k^+)[\frac{1}{\pi}])$ as we wanted to show.
\end{proof}

\subsection{The scheme-theoretic comparison}
\begin{theorem}
	\label{comparison-reduction}
	Let $G$ be a reductive group over $E$ and $\calG$ a parahoric $O_E$-model of $G$.
	\begin{enumerate}
		\item The natural map $\ICG\xrightarrow{\cong} (\Bun_G)^\red$ is an isomorphism of scheme-theoretic v-sheaves valued in groupoids.
		\item The natural map $\DMG\xrightarrow{\cong} (\Sht_\calG)^\red$ is an isomorphism of scheme-theoretic v-sheaves valued in groupoids.
	\end{enumerate}
\end{theorem}
\begin{proof}
	Let $X\in \PSch$. For the first claim we write:
	\begin{align*}
		\ICG(X) & = \on{Fun}_{\on{ex}}^\otimes(\Rep_G,\IC(X))\\	
			& \cong  \on{Fun}_{\on{ex}}^\otimes(\Rep_G,\Bun_\FF(X^\diamond))\\
			& =  \Bun_G(X^\diamond)\\
			& = (\Bun_G)^\red(X)\,.
	\end{align*}
Here, the second isomorphism is \Cref{isomorphism-isocrystals-bung}. 	

For the second claim, since $(\Sht_\calG)^\red$ and $\DMG$ are v-sheaves (the latter by \Cref{DMisvstack}), it suffices to prove that $\DMG(X)\xrightarrow{\cong} \Sht_\calG(X^\diamond)$ when $X=\Spec A$ is a comb.
In this case, $\DMG(X)$ is equivalent to the category where the objects are elements $M_\calE\in \calG(\bbW(A)[\frac{1}{\pi}])$, and morphisms between $M_{\calE_1}$ and $M_{\calE_2}$ are elements $N\in \calG(\bbW(A))$ with $N^{-1} M_{\calE_1} \varphi(N)=M_{\calE_2}$.
On the other hand, by \Cref{meromorphic-comparison-theorem} (3), an isomorphism between $M_{\calE_1}$ and $M_{\calE_2}$ in $\Sht_\calG(X^\diamond)$ corresponds to a functorial choice of elements $N_R\in \calG(\bbW(R^\circ))$ with $N_R^{-1} M_{\calE_1} \varphi(N_R)=M_{\calE_2}$ ranging over maps $\Spa(R,R^+)\to X^\diamond$, with $\Spa(R,R^+)$ a product of points.
Recall that the functor 
\[(R,R^+)\mapsto R^\circ\]
is represented by the closed subsheaf $(\bbA^1)^\dagger\subseteq \Spd(\bbF_q[T],\bbF_q)=(\bbA^1_{\bbF_q})^\Diamond$.
In particular, $H^0(X^\diamond,\calO^\circ)\subseteq \{X^\diamond\to \bbA^1_{\bbF_q}\}$.
By \cite[Theorem 2.32]{Gle22}, we can conclude that $H^0(X^\diamond,\calO^\circ)=A$. 
Since $H^0(X^\diamond,\calO^\circ)=A$, such a collection of $N_R$ uniquely comes from an element $N\in \calG(\bbW(A))$. 
This shows that $\DMG(X)\to \Sht_\calG(X^\diamond)$ is fully-faithful.

To prove essential surjectivity, fix $\calE\in \Sht_\calG(X^\diamond)$. This induces elements $\calE_\Bun\in \Bun_G(X^\diamond)$ and $\calE_\IC\in \ICG(X)$ unique up to isomorphism.
Objects in $\DMG(X)$ lifting $\calE_\IC$ correspond to sections of $\DMG\times_{\ICG} X\to X$, whereas objects in $\Sht_\calG(X^\diamond)$ lifting $\calE_\Bun$ correspond to sections $\Sht_\calG\times_{\Bun_G} X^\diamond\to X^\diamond$. 
The result then follows from \Cref{reductionoffamilyofshtukas} below.
\end{proof}

\begin{lemma}
	\label{reductionoffamilyofshtukas}	
	Let $X\in \PSch$ be a comb and $X\to \ICG$ be a map, then 
	the natural map
	\[\DMG\times_{\ICG} X \xrightarrow{\simeq} (\Sht_\calG\times_{\Bun_G}X^\diamond)^\red\]
	induces an isomorphism.
\end{lemma}
\begin{proof}
	The argument given in \cite[Proposition 2.30]{Gle21} works in this generality.	
	To orient the reader, we give a summary of the proof of \cite[Proposition 2.30]{Gle21} in the more general case we propose here.
	Since $X=\Spec A$ is a comb, every map $X\to \ICG$ is given by a matrix $M\in \calG(\bbW(A)[\frac{1}{\pi}])$, cf.\ \cite[Theorem 6.1]{Ivanov_arc_descent}. 
	The functor 
	\[\DMG\times_{\ICG} X\colon (\PSch)^\op_{\!/X}\to \Sets\]
	parametrizes on $T=\Spec R$ triples $(\calE^\sch,\Phi^\sch,\rho^\sch)$, where $\calE^\sch$ is a $\calG$-torsor over $\Spec \bbW(R)$, $\Phi\colon \varphi^*\calE^\sch\to \calE^\sch$ is an isomorphism defined over $\Spec \bbW(R)[\frac{1}{\pi}]$ and $\rho\colon \calE^\sch\to \calG$ is an isomorphism defined over $\Spec \bbW(R)[\frac{1}{\pi}]$.
	Moreover, these data have to fit in the following commutative diagram
	\begin{center}
	\begin{tikzcd}
		\varphi^*\calE^\sch \arrow{r}{\varphi^*\rho^\sch} \arrow{d}{\Phi^\sch}  & \calG \arrow{d}{M_T} \\
		\calE^\sch \arrow{r}{\rho^\sch} & \calG\,.
	\end{tikzcd}
	\end{center}
	We notice at this point that the data of $\Phi$ are completely determined by $\rho$, and for this reason $\DMG\times_{\ICG} X\simeq \on{Gr}_X$, where the latter is the Witt vector Grassmannian which is an ind-scheme.

	On the other hand, the functor 
	\[\Sht_\calG\times_{\Bun_G}X^\diamond\colon (\Perf)^\op_{\!/X^\diamond}\to \Sets\]
	parametrizes on $S=\Spa(R,R^+)$ triples 
	$(\calE,\Phi,\rho)$, where $\calE$ is a $\calG$-torsor over $\calY_S$, the map $\Phi\colon \varphi^*\calE\to \calE$ is an isomorphism defined over $\calY_{(0,\infty),S}$ such that $\Phi$ is meromorphic along $V(\pi)\subseteq \calY_S$, the map $\rho\colon \calE\to \calG$ is an isomorphism defined over $\calY^{R^+}_{(0,\infty)}$ (not necessarily meromorphic), and the data fit in the commutative diagram
	\begin{center}
	\begin{tikzcd}
		\varphi^*\calE \arrow{r}{\varphi^*\rho} \arrow{d}{\Phi}  & \calG \arrow{d}{M_T} \\
		\calE \arrow{r}{\rho} & \calG\,. 
	\end{tikzcd}
	\end{center}
	Note that in this case, although the data of $\rho$ determine $\Phi$ (since $\Phi=\rho^{-1}\circ M_T\circ \varphi^*\rho$), it is no longer true in this context that every choice of $\rho$ defines a $\Phi$ which is meromorphic. 

	Now, 
	\[(\Sht_\calG\times_{\Bun_G}X^\diamond)^\red\colon (\PSch)^\op\to \Sets\]
	has the formula 
	\[(\Sht_\calG\times_{\Bun_G}X^\diamond)^\red(T)=(\Sht_\calG\times_{\Bun_G}X^\diamond)(T^\diamond)\]
	which is the same as giving a compatible system of maps 
	\[\Spa(R,R^+)\to \Sht_\calG\times_{\Bun_G}X^\diamond\]
	as $\Spa(R,R^+)$ ranges over over $\Perf_{\!/T^\diamond}$.

	Given $(\calE^\sch,\Phi^\sch,\rho^\sch)\in \DMG\times_{\ICG} X (T)$ one can construct a compatible systems of maps
	by pulling back all data along $\calY_{\Spa(R,R^+)}\to \Spa \calY_T$, and what \cite[Proposition 2.30]{Gle21} shows is that every system of compatible data arises uniquely in this way.
	The proof of \cite[Proposition 2.30]{Gle21} treats the case in which $X=\Spec \bar{\bbF}_p$, but this hypothesis is not essential to the argument, the only thing that is really used is that the map $X\to \ICG$ has trivial underlying $\calG$-torsor. 

	The proof of \cite[Proposition 2.30]{Gle21} goes as follows. 
	One first shows that $\DMG\times_{\ICG} X\to (\Sht_\calG\times_{\Bun_G}X^\diamond)^\red$ is injective, or that the triple $(\calE^\sch,\Phi^\sch,\rho^\sch)\in (\DMG\times_{\ICG} X) (T)$ is determined by the compatible system of data that it induces.
	To prove surjectivity of $\DMG\times_{\ICG} X\to (\Sht_\calG\times_{\Bun_G}X^\diamond)^\red$ the harder part of the argument is showing that given a compatible system of data 
	\[(\calE_R,\Phi_R,\rho_R)\in  (\Sht_\calG\times_{\Bun_G}X^\diamond)(R,R^+)\]
	with $\Spa(R,R^+)$ ranging over $\Perf_{\!/T^\diamond}$, each of the $\rho_R$ defined over $Y^{R^+}_{(0,\infty)}$ is meromorphic along $V(\pi)\subseteq Y^{R^+}_{[0,\infty)}$.
	Once one knows that each $\rho_R$ is meromorphic, it is not hard to show it comes from a unique map 
	\[\rho_T\colon \calE^\sch\to \calG\]
	defined over $Y_T$.
\end{proof}

\subsection{The topological comparison}
Recall from \S\ref{sec:newton_strata} that $B(G)$ is a partially ordered set. We equip $B(G)$ with the topology induced by the partial order, i.e.\ $[b] \in \overline{\{[b']\}}$ if and only if $[b] \leq [b']$ in the partial order. Recall that by results of Rapoport--Richartz \cite{RR_96} and He \cite[Theorem 2.12]{He_hecke_padic},
\begin{equation}
	\label{eq;ICGvsB(G)}
	|\ICG|\cong B(G)\,.
\end{equation}
Alternatively, by the results of Fargues--Scholze \cite[Theorem I.4.1.(iii)]{FS24} and of Viehmann \cite[Theorem 1.1]{Viehamann_Newton_BdR}, we also have
\begin{equation}
	\label{eq:B(G)vsBunG}	
|\Bun_G|^{\on{op}}\cong B(G)\,.
\end{equation}
Here $|\Bun_G|^\op$ is the topological space where a subset is open in $|\Bun_G|^\op$ if and only if it is closed in $|\Bun_G|$ (\Cref{alltopologieshavequotientopology}.(4) explains that this rule really defines a topological space).
Combining these two sets of references, we obtain that 
\begin{equation}
	\label{isocvsbunGtopologyeq}
	|\ICG|\cong |\Bun_G|^{\on{op}}\,.
\end{equation}
In this section we give a direct and new proof of the identity \eqref{isocvsbunGtopologyeq}. 
As a consequence, we prove that the identities \eqref{eq;ICGvsB(G)} and \eqref{eq:B(G)vsBunG} are equivalent statements.

Let us clarify. 
The statements of the form 
\begin{equation}
	B(G)\cong |\ICG| \quad \text{ and } \quad B(G)\cong |\Bun_G|^\op
\end{equation}
are naturally broken into two complementary statements. 
The first statement is Grothendieck's specialization theorem (as generalized in \cite{RR_96}), which says that the Newton polygon of a family of isocrystals over a scheme decreases as the family degenerates along a closed subscheme. The analogous statement for $\Bun_G$ is Kedlaya--Liu's semi-continuity theorem \cite[Theorem 22.2.1]{SW20} (as generalized in \cite[Theorem I.4.1.(iii)]{FS24}). 
One way of formulating Grothendieck's theorem (respectively Kedlaya--Liu's theorem) is to say that we have continuous maps  
\begin{equation}
	\label{continuous-equation}
|\ICG|\to B(G) \text{ (respecitvely } |\Bun_G|^\op\to B(G) \text{)}.
\end{equation}
The second statement is Grothendieck's conjecture \cite[\S 3.13]{RR_96}, which asks if every specialization of the combinatorially defined partial order in $B(G)$ can be realized by a geometric family. 
This is what \cite[Theorem 2.12]{He_hecke_padic} and \cite[Theorem 1.1]{Viehamann_Newton_BdR} show in the $\ICG$ case and $\Bun_G$ case, respectively. 
Given \eqref{continuous-equation}, Grothendieck's conjecture becomes equivalent to the statement that we have homeomorprhisms
\begin{equation}
	|\ICG|\xrightarrow{\simeq} B(G) \text{ (respecitvely } |\Bun_G|^\op\xrightarrow{\simeq} B(G) \text{)}.
\end{equation}

Our proof shows that if we assume that Grothendieck's specialization theorem (in its $\ICG$ and $\Bun_G$ form) already holds, then Grothendieck's conjectures in both setups are equivalent. \\

Let us set some notation.
\begin{definition}
	\label{partial-order}
	Let $b_1,b_2\in B(G)$. 
	\begin{enumerate}
		\item We say that $b_1 \preceq_{\ICG} b_2$ if $b_1\in \overline{\{b_2\}}$ in $\ICG$.
		\item We say that $b_1 \preceq_{\Isoc_G} b_2$ if $b_1\in \overline{\{b_2\}}$ in $\Isoc_G$.
		\item We say that $b_1 \preceq_{\Bun^{\on{op}}_G} b_2$ if $b_2\in \overline{\{b_1\}}$ in $\Bun_G$.
	\end{enumerate}
	Moreover, we write $b_1 \preceqdot_{\ICG} b_2$, $b_1 \preceqdot_{\Isoc_G} b_2$ or $b_1 \preceqdot_{\Bun^\op_G} b_2$ whenever $b_2$ covers $b_1$ in the respective partial order (i.e.\ $b_1 \preceq b_2$ and there is no element $b^{\prime} \in B(G) \setminus\{b_1,b_2\}$ such that $b_1 \preceq b^{\prime} \preceq b_2$ in the respective orders).
\end{definition}

The following lemma says that the information of the topological spaces in question is completely determined by the partial order that the closure relations define. In particular, it will suffice to compare the corresponding partial orders.

\begin{lemma}
	\label{alltopologieshavequotientopology}
	Let $U\subseteq B(G)$. For $b\in B(G)$ we let $U_{\leq b}\coloneqq U\cap B(G)_{\leq b}$.
	\begin{enumerate}
		\item $U$ is closed in $\ICG$ $\iff$ $U_{\leq b}$ is closed in $\ICG$ for all $b\in B(G)$. 
		\item $U$ is closed in $\Isoc_G$ $\iff$ $U_{\leq b}$ is closed in $\Isoc_G$ for all $b\in B(G)$. 
		\item $U$ is open in $\Bun_G$ $\iff$ $U_{\leq b}$ is open in $\Bun_G$ for all $b\in B(G)$. 
		\item $|\Bun_G|^\op$ is a topological space. 
		\item The topologies on $\ICG$, $\Isoc_G$ and $\Bun_G$ are determined by the partial order they induced as in \Cref{partial-order}. 
			More precisely a set $U\subset \ICG$ is closed if and only if for all $b\in U$ the sets $\ICG_{\leq b}=\{b'\in \ICG\mid b'\leq b\}$ is a subset of $U$ (analogously for $\Isoc_G$ and $\Bun_G$).
		\item Let $b_1 \preceq_{\ICG} b_2$ (resp.\ $b_1 \preceq_{\Bun^{\on{op}}_G} b_2$ ). 
			Then the set $\{b \in |\ICG|| b_1 \preceq_{\ICG} b \preceq_{\ICG} b_2\}$ (resp.\ $\{b \in |\Bun^{\on{op}}_G|| b_1 \preceq_{\Bun^{\on{op}}_G} b \preceq_{\Bun^{\on{op}}_G} b_2\}$) is finite.
			In particular, whenever $b_1 \preceq_{\ICG} b_2$ (resp.\ $b_1 \preceq_{\Bun^{\on{op}}_G} b_2$) we may find a finite chain of immediate specializations $b_1=b^{1}\preceqdot_{\ICG} b^2\preceqdot_{\ICG} \dots \preceqdot_{\ICG} b^{n-1}\preceqdot_{\ICG} b^n=b_2$ (resp.\ $b_1=b^{1}\preceqdot_{\Bun^{\on{op}}_G} b^2\preceqdot_{\Bun^{\on{op}}_G} \dots \preceqdot_{\Bun^{\on{op}}_G} b^{n-1}\preceqdot_{\Bun^{\on{op}}_G} b^n=b_2$) connecting $b_1$ to $b_2$.
	\end{enumerate}
\end{lemma}
\begin{proof}
	We prove the first claim, the second and third claim being analogous.
	The forward implication is evident since $\ICG_{\leq b}\subseteq \ICG$ is a closed immersion.
	Let $f\colon \Spec R\to \ICG$ be any map. As $\Spec R$ is quasi-compact, there are a finite number of elements $b^f_i\in B(G)$ such that $f$ factors through $\bigcup_{i=1}^n \ICG_{\leq b^f_i}$.
	By assumption $U\cap \bigcup_{i=1}^n \ICG_{\leq b^f_i}$ is closed in $\ICG$.
	Since $f$ factors through the set above, base change of $f$ along $U$ defines a closed immersion.

	The fourth claim follows from the third. 
	Indeed, the only part that needs justification is that an arbitrary union of open subsets in $|\Bun_G|^\op$ is open. 
	This is equivalent to the preservation of open subsets of $|\Bun_G|$ under arbitrary intersections,
    but arbitrary intersections can be expressed as finite intersections when we restrict them to $\Bun_G^{\leq b}$.

	The fifth claim follows from the first three. 
	Indeed, $\ICG$, $\Isoc_G$ and $\Bun_G$ have the strong topology along the inclusion maps from $\coprod_{b\in B(G)} \ICG_{\leq b}$, $\coprod_{b\in B(G)} \Isoc_G^{\leq b}$ and $\coprod_{b\in B(G)} \Bun_G^{\leq b}$. 
	Moreover, since these latter ones are finite topological spaces, they are determined by their closure relations. 
	
	The last claim follows by the respective versions of Grothendieck's specialization theorem and the fact that the analogous claim holds for $B(G)$ with its partial order.\qedhere
\end{proof}

The following implications are reformulations of Grothendieck's specialization theorem in the respective setup.
\begin{enumerate}
	\item $b_1 \preceq_\ICG b_2 \implies b_1 \leq_{B(G)} b_2$
	\item $b_1 \preceq_{\Isoc_G} b_2 \implies b_1 \leq_{B(G)} b_2$
	\item $b_1 \preceq_{\Bun_G^\op} b_2 \implies b_1 \leq_{B(G)} b_2$
\end{enumerate}

\begin{theorem}
	\label{topologicalcomparisonmainthext}
The partial orders $\preceq_\ICG$, $\preceq_{\Isoc_G}$, $\preceq_{\Bun_G^\op}$ agree. In particular, we have natural homeomorphisms 
\[|\ICG| \cong |\! \Isoc_G\! | \cong|\Bun_G|^\op\,.\]
\end{theorem}
\begin{proof}
	For the rest of the proof we fix $b_1,b_2\in B(G)$ with $b_1\leq_{B(G)} b_2$.
	We first prove $|\ICG|\cong |\Isoc_G|$. Recall that $\lozenge$ preserves closed immersions, consequently:
	\[b_1\preceq_{\Isoc_G} b_2 \implies b_1\preceq_{\ICG} b_2\,.\]
	
	Now, suppose that $b_1\preceqdot_{\ICG} b_2$ (which we may assume by the last part of \Cref{alltopologieshavequotientopology}).
	We claim that there is a perfect rank $1$ valuation ring $V$ and a map $\Spec V\to \ICG$ such that the induced maps on $\Spec k_V$ (the residue field) and $\Spec K_V$ (the fraction field) factor through $\ICG_{b_1}$ and $\ICG_{b_2}$, respectively.	
	We do this in steps. 

	\textit{Step 1:}
	We first find a map $f\colon \Spec R\to \ICG$ with the property that for all $x\in \Spec R$ the induced map $\Spec k_x\to \ICG$ factors through either $\ICG_{b_1}$ or $\ICG_{b_2}$ and with the property that $\overline{f^{-1}(\ICG_{b_2})}\cap f^{-1}(\ICG_{b_1})\neq \emptyset$. 
	We do this as follows.
By definition, the underlying topological space of $\ICG$ has the quotient topology along all maps $\Spec R\to \ICG$ as we range over $\Spec R\in \PSch$.
In particular, whenever $b_1\preceq_{\ICG} b_2$ there must be a map $g:\Spec R'\to \ICG$, and a point $y\in \Spec R'$ such that $g(y)=b_1$ and such that every open in $\Spec R'$ containing $y$ meets $g^{-1}(\ICG_{b_2})$. We may replace $\Spec R'$ with its intersection with $\ICG_{\leq b_2}$ to ensure that every $x\in \Spec R'$ factors through $\ICG_{\leq b_2}$. 
If additionally $b_1\preceqdot_\ICG b_2$, we may remove the closed subsets of the form $\ICG_{\leq b'}$ for all $b'\prec_\ICG b_2$ and $b'\neq b_1$.
This gives rise to an open subset $U\subseteq \Spec R'$ containing $y$ and with the property that for all $x\in U$ the induced map $\Spec k_x$ factors through either $\ICG_{b_2}$ or $\ICG_{b_1}$. Letting $\Spec R\subseteq U$ be any affine open subset containing $y$ fits the bill.

\textit{Step 2:}
	We may replace $\Spec R$ by a v-cover, so we may assume that $R=\prod_{i\in I}\! V_i$ is a product comb.
	Since the inclusion $\ICG_{\leq b_1}\to \ICG$ is a perfectly finitely presented closed immersion and $R$ is a B\'ezout ring (see \Cref{product-combs-bezout}), there is $r\in R$ such that $\Spec R/(r)^{\on{perf}}\subseteq \Spec R$ is equal to $f^{-1}(\ICG_{b_1})$.
	We may write $R=R_1\times R_2$, where $R_1=\prod_{\{i\in I\mid r_i=0\}}\!V_i$ and $R_2=\prod_{\{i\in I\mid r_i\neq 0\}}\!V_i$, and replace $R$ by $R_2$.
	Let $K_{V_i}$ denote the fraction field of $V_i$.
	Now, $\Spec \prod_{i\in I}\! K_{V_i}\subseteq \Spec R$ is a pro-open subset lying in $f^{-1}(\ICG_{b_2})$. 
	Let $\beta I$ denote the Stone--\v{C}ech compactification of $I$.
	It is a standard fact that $\pi_0(\Spec R)\simeq \beta I$.
	Since $f^{-1}(\ICG_{b_1})$ is non-empty, there is a connected component in $x\in \beta I$ with associated valuation ring $V_x$ such that that the image of $r$ in $V_x$, which we denote by $r_x$, is not identically $0$, but is also not a unit.
	The largest prime ideal contained in $\langle r_x\rangle$ and the smallest prime ideal containing $\langle r_x\rangle$ define a rank $1$ valuation ring with the desired properties. This finishes our construction of the map $\Spec V\to \ICG$ which we had claimed exists.

	The map $\Spec V\to \ICG$ induces a map $\Spd(V,V)\to\ICG^\diamond\to \Isoc_G$ such that the corresponding maps on $\Spd(k_V,k_V)$ and $\Spd(K_V,K_V)$ factor through $\Isoc_G^{b_1}$ and $\Isoc_G^{b_2}$, respectively.
	This implies that $\Spd(K_V,V)\to \Isoc_G$ factors through $\Isoc_G^{b_2}$, but $\Spd(K_V,V)\subseteq \Spd(V,V)$ is dense. 
	This proves: 
	\[b_1\preceq_{\ICG} b_2 \implies  b_1\preceq_{\Isoc_G} b_2\,.\]
	
	In the same fashion, the map $\Spec V\to \ICG$ induces a map $\Spd(V,V)\to \ICG^\diamond \to \Bun_G$ that if restricted to $\Spd(k_V,k_V)$ and $\Spd(K_V,K_V)$, factors through $\Bun_G^{b_1}$ and $\Bun_G^{b_2}$, respectively.
	Let $\pi\in V$ be a pseudo-uniformizer, let $\hat{V}_\pi$ be the $\pi$-adic completion of $V$ and let $K=\hat{V}_{\pi}[\frac{1}{\pi}]$.
	We note that $\Spa(K,\hat{V}_\pi)$ is a perfectoid field and that $\Spd(\hat{V}_\pi,\hat{V}_\pi)$ has two points, one corresponding to $\Spd(K,\hat{V}_\pi)$ and one corresponding to $\Spd(k_V,k_V)$.
	By \Cref{theorem-evaluating-opensubs}, the map $\Spd(\hat{V}_\pi,\hat{V}_\pi)\to \Bun_G$ corresponds to a $\otimes$-exact functor from $\Rep_G$ to the category of $\varphi$-equivariant objects in $\on{Vect}(Y_{(0,\infty]}^K)$.
	Using \cite[Proposition 2.1.3]{PR21} and \cite[Theorem II.2.14]{FS24}, we conclude that the map $\Spd(\hat{V}_\pi,\hat{V}_\pi)\to \Bun_G$ factors through $\Bun_G^{b_1}$ as $\Spd(k_V,k_V)\to \Bun_G$ does. 
	Moreover, $\Spd(\hat{V}_\pi,\hat{V}_\pi)\subseteq \Spd(V,V)$ is an open subsheaf whose v-sheaf-theoretic closure is $\Spd(V,V)$.
	This allows us to conclude:
	\[b_1\preceq_{\ICG} b_2 \implies  b_1\preceq_{\Bun^{\op}_G} b_2\,.\]

	Finally, suppose that $b_1\preceqdot_{\Bun^{\op}_G} b_2$ (which we again may assume by the last part of \Cref{alltopologieshavequotientopology}).
	Arguing as in the \textit{Step 1} of the above argument, we may use our assumptions to find a map $\Spa(R,R^+)\to \Bun_G$ with the property that for all $x\in \Spa(R,R^+)$, the induced map $\Spa (k(x),k(x)^+)\to \Bun_G$ factors through either $\Bun_G^{b_1}$ or $\Bun_G^{b_2}$ and with the property that $\overline{f^{-1}(\Bun_G^{b_1})}\cap f^{-1}(\Bun_G^{b_2})\neq \emptyset$. 
	Replacing $\Spa(R,R^+)$ by a v-cover, we may assume that it is a product of points with $R^+=\prod_{i\in I}\!C^+_i$.
	By shrinking $\Spa(R,R^+)$ and ignoring some factors if necessary, we may assume that the principal components of $\Spa(R,R^+)$ all factor through $\Bun_G^{b_1}$ without changing the condition that $\overline{f^{-1}(\Bun_G^{b_1})}\cap f^{-1}(\Bun_G^{b_2})\neq \emptyset$.
	This forces at least one non-principal component to factor through $\Bun_G^{b_2}$. 
	Moreover, we may assume $C^+_i=O_{C_i}$ for all $i$ so that $R^+=R^\circ$.
	By \Cref{meromorphic-comparison-theorem}, we may assume that our map $\Spa(R,R^+)\to \Bun_G$ is induced from a map $\Spec R^+\to \ICG$.
	Let $k_i$ denote the residue field of $O_{C_i}$.
	By assumption, the map $\Spa(C_i,O_{C_i})\to \Bun_G$ factors through $\Bun_G^{b_1}$. 
	In particular, $\Spec k_i\to \ICG$ factors through $\ICG_{b_1}$,
	which implies that $\Spec \prod_{i\in I} k_i\to \ICG$ also factors through $\ICG_{b_1}$.
	Indeed, it certainly factors through $\ICG_{\leq b_1}$, and the locus where it factors through $\ICG_{b'}$ with $b'<b_1$ is finitely presented and contains no principal component of $\Spec \prod_{i\in I} k_i$, which implies that it is empty.   
	We see that the closed point of every connected component of $\Spec R^+$ factors through $\ICG_{b_1}$. 
	Furthermore, there is at least one point $x\in \Spec R^+$ mapping to $\ICG_{b_2}$. 
	The connected component containing $x$ defines a valuation ring $V_x$ and a map $\Spec V_x\to \ICG$ such that the closed point factors through $\ICG_{b_1}$ and at least one point of $\Spec V_x$ factors through $\ICG_{b_2}$. 
	This allows us to conclude that
	\[b_1\preceq_{\Bun^\op_G} b_2 \implies b_1\preceq_{\ICG} b_2\,. \qedhere\] 
\end{proof}

Recall that the groupoid of maps $X\to \ICG$ and the groupoid of maps $X^\diamond\to \Bun_G$ are equivalent. 
The following statement explains the relation between the stratifications that such data induces in $X$.

\begin{proposition}
	\label{stratifying-BunG}
	Let $X = \Spec A$ and let $X\to \ICG$ be a map. Let $Z_b\subseteq X$ denote the closed subscheme that factors through $\ICG_{\leq b}$ and let $U_b\subseteq X^\diamond$ denote the open subsheaf that factors through $\Bun_G^{\leq b}$.
	Then $U_b=\widehat{X}_{\!/Z_b}$ with notation as in \cite[Definition 4.18]{Gle22} (i.e.\ $U_b$ is the formal neighborhood around $Z_b$).
\end{proposition}
\begin{proof}
	By definition $U_b$ is an open subsheaf and since $Z_b\to X$ is a finitely presented closed immersion, $\widehat{X}_{\!/Z_b}$ is also an open subsheaf by \cite[Proposition 4.22]{Gle22} .
	It suffices to show that $U_b$ and $\widehat{X}_{\!/Z_b}$ have the same geometric points and since both of these subsheaves are partially proper over $X^\diamond$, it even suffices to show that they have the same rank $1$ geometric points.
	Every map of the form $\Spa(C,O_C)\to X^\diamond$ factors uniquely through a map $\Spd O_C\to X^\diamond$. 
	We let $k=O_C/C^{\circ \circ}$ and note that $|\Spd O_C|$ consists of two points, where one corresponds to $\Spa(C,O_C)$ while the other corresponds to $\Spd k$.
	Let $b_k\in B(G)$ be the unique isomorphism class of the induced map $\Spd k\to X^\diamond\to \Bun_G$ 
	We claim that $\Spd O_C\to \Bun_G$ factors through $\Bun_G^{b_k}$.
	Indeed, if $S=\Spa(C,O_C)$, by \Cref{theorem-evaluating-opensubs} the map $\Spd O_C\to \Bun_G$ corresponds to a $\otimes$-exact functor from $\Rep_G$ to the category of $\varphi$-equivariant objects in $\calV(Y_{(0,\infty],S})$.
	The claim then follows from \cite[Theorem 13.2.1, Theorem 13.4.1]{SW20}.

	In particular, a rank $1$ geometric point $\Spa(C,O_C)\to X^\diamond$ lies over $U_b$ if and only if its induced $\Spd k\to X^\diamond$ point lies over $\Bun_G^{b_k}$ for $b_k\in B(G)_{\leq b}$. 
	This is the same as saying that its image under the specialization map 
	\[\on{sp}\colon |X^\diamond|\to |X|\]
	lies on $|Z_b|\subseteq |X|$. 
This in turn is the definition of $\widehat{X}_{\!/Z_b}$.
\end{proof}

\begin{corollary}
	Let $X= \Spec A$ and let $X\to \ICG$ be a map. 
	Let $U\subseteq B(G)$ the intersection of a finite closed subset with an open subset. 
	Let $Z_U\subseteq X$ denote the locally closed subscheme that factors through $U\subseteq |\ICG|$.
	Let $U_b\subseteq X^\diamond$ denote the locally closed subsheaf that factors through $\Bun_G^{b\in U}$.    
	Then $U_b$ is the smallest subsheaf of $X^\diamond$ containing $\widehat{X^\diamond}_{\!\!/Z_b}$ and stable under vertical specialization.
\end{corollary}
\begin{proof}
	The same argument as in \Cref{stratifying-BunG} will show that $\widehat{X^\diamond}_{\!\!/Z_b}$ and $U_b$ agree on rank $1$ points.  
	Since we allow $U$ to be more general, it is no longer true that $\widehat{X^\diamond}_{\!\!/Z_b}$ is partially proper over $X^\diamond$, while it is still true that $U_b$ is. 
	The description given above takes this into account.
\end{proof}

\appendix

\section{Sheaves of categories}
\label{Appendix-sect}
Throughout the body of the text, we used the sheafification construction in $\CatexE$. 
This appendix has two purposes: 1) to justify why the $\infty$-category $\CatexE$ is presentable and compactly generated so that sheafification is well defined and 2) to collect some general categorical statements pertaining separated presheaves that were used in the above.
We start with the second purpose. 
For the rest of the appendix we let $\calT\in \{\Perf\!, \PSch\}$.
More generally, we can work with a site $\calT$ that has all finite limits and such that every covering sieve $\calU\subseteq \calT_{/X}$ contains a covering sieve generated by a single element $U\to X$. 
Clearly both sites, $\Perf$ and $\PSch$, satisfy the hypothesis. 

Recall that given a complete $\infty$-category $\calC$ and a presheaf $\calF\in \calP(\calT,\calC)$ we say that $\calF$ is a sheaf if for every object $X\in \calT$ and every covering sieve $\calU\subseteq \calT_{\!/X}$ the natural map
\[\calF(X)\to \!\varprojlim_{U\in \calU}\!\calF(U)\]
is an equivalence.
The following statement says that in our categories $\{\Perf\!,\PSch\}$, it suffices to verify that descent holds for specific covering sieves.

\begin{lemma}
	\label{sheaf-case}
	A presheaf $\calF\in \calP(\calT,\calC)$ is a sheaf if and only if for every surjective map $[U\to X]\in \calT$ the map 
	\[\calF(X)\to \on{Desc.}(\calF,U/X)\]
	is an equivalence.
\end{lemma}
Here $\on{Desc.}(\calF,U/X)$ denotes the descent category defined in terms of the \v{C}ech resolution $U_\bullet\to X$, where $U_i$ denotes the $i$-fold product $U_i\coloneqq U\times_X \dots \times_X U$.
\begin{proof}
	This is classical. For a reference, see \cite[Lemma A.4.8, A.4.6]{heyer20246}.
\end{proof}

\begin{lemma}\label{lm:ess_surjectivity_local}
Let $\calF_1,\calF_2 \colon \calT \to {\rm Cat}_1$ be two sheaves of categories on $\calT$ and suppose that $f \colon \calF_1 \to \calF_2$ is a morphism of sheaves. If $f$ is fully-faithful and locally on $\calT$ essentially surjective, then it is essentially surjective.
\end{lemma}
\begin{proof}
Let $X \in \calT$ and let $(X_i \to X)_i$ be a covering in $\calT$, such that $f_{X_i}$ is essentially surjective. For any $B \in \calF_2(X)$, let $A_i \in F_1(X_i)$ be a preimage under $f_{X_i}$ of $B|_{X_i}$. By the sheaf property of $\calF_1$  and the fully-faithfulness of $f_{X_i}$ and $f_{X_i \times_X X_j}$ it follows that the $A_i$ glue to a unique object of $F(X)$ mapping to $B$.
\end{proof}

\begin{definition}
	\label{separated-presheaves}
	Let $\calC$ be a presentable category. 
	Let $\calF \in  \calP(\calT,\calC)$ be a presheaf and let $\tilde \calF$ denote its sheafification. 
	We call $\calF$ \emph{separated} if the natural map $\calF \to \calF \times_{\tilde\calF} \calF$ in $\calP(\calT,\calC)$ is an equivalence.
\end{definition}

Recall the $+$-construction (or $\dagger$-construction) involved in the explicit construction of sheafification \cite[6.2.2.9-6.2.2.13]{lurie2009htt}. 
The functor $\calF\mapsto \calF^+$ comes with a natural transformation $\eta_{(-)}\colon\on{id}\to (-)^+$ which satisfies the property that for all $\calG\in \calS(\calT,\calC)$ and all $\calF\in \calP(\calT,\calC)$ the maps
\[\on{Hom}_{\calP(\calT,\calC)}(\calF^+,\calG)\to \on{Hom}_{\calP(\calT,\calC)}(\calF,\calG)\]
are equivalences \cite[Lemma 6.2.2.14]{lurie2009htt}.
In particular, if $\calF^+$ is a sheaf the map $\eta_\calF\colon\calF\to \calF^+$ exhibits $\calF^+$ as the sheafification of $\calF$. 
We use the following criterion to determine when a presheaf is separated.

\begin{proposition}
\label{pass-to-anima}
	Let $\calC$ be a presentable and compactly generated category and let $\calF\in \calP(\calT,\calC)$.
	Given a covering map $[f\colon U\to X]\in \calT$, we let $\calD_{U/X}\coloneqq\on{Desc.}(\calF,U/X)$.
	The following are equivalent.
	\begin{enumerate}
		\item $\calF$ is a separated presheaf.
		\item For all $f$ as above, the diagonal $\Delta_{\calF,f}\colon \calF(X)\to \calF(X)\times_{\calD_{U/X}}\calF(X)$ is an equivalence.
	\end{enumerate}
	Moreover, if any of these statements hold, then $\calF^+$ is a sheaf. 
\end{proposition}
\begin{proof}
	We start by showing that $(2)$ already implies that $\calF^+$ is a sheaf.
	Given a compact generator $A\in \calC^\omega$, we let $h_A\in \calP(\calT,\on{Ani})$ denote the representable presheaf $h_A(X)\coloneqq\on{Hom}_\calC(A,\calF(X))$.
	We observe that since $A\in \calC^\omega$, we have functorial equivalences $(h_A)^+\simeq \on{Hom}_\calC(A,\calF^+)$.
	By the hypothesis that $\calC$ is compactly generated, the family of functors $\{\on{Hom}(A,-)\colon \calC\to \on{Ani}\}_{A\in \calC^\omega}$ is conservative. 
	In particular, we may test the sheaf condition on $\calF^+$ by showing that $h_A^+$ is a sheaf for all $A\in \calC^\omega$. 
	The hypothesis show that for every map $[U\to X]\in \calT$ the diagonal $h_A(X)\to h_A(X)\times_{\on{Desc.}(h_A,U/X)} h_A(X)$ is an equivalence.  
	This suffices to conclude that $h_A^+$ is a sheaf (see \cite[Proposition 3.4.22]{anel20}), and consequently $\calF^+$ also is.

	Let us show that the two statements are equivalent.
	We observe that since $\calC$ is compactly generated, the sheafification of $h_A$ agrees with $\on{Hom}_\calC(A,\tilde{\calF})$. 
	Since both assertions, the first being $\calF$ being separated, and the second being the condition that $\Delta_{\calF,f}\colon \calF(X)\to \calF(X)\times_{\calD_{U/X}}\calF(X)$ is an equivalence, can be verified on the representable presheaves $\{h_A\}_{A\in \calC^\omega}$, we may and do reduce to the case $\calC=\on{Ani}$ and $\calF=h_A$.
	In one direction, arguing as above, if the diagonal $h_A(X)\to h_A(X)\times_{\on{Desc.}(h_A,U/X)} h_A(X)$ is an equivalence, then the sheafification of $h_A$ is $h^+_A$, one can use the concrete expression of $h_A^+(X)$ and that in $\on{Ani}$ filtered colimits commute with finite limits to show that $h_A\to h_A\times_{\tilde{h}_A} h_A$ is also an equivalence.

	For the converse, fix a cover $[U\to X]\in \calT$ and consider the diagram
	\begin{center}
		\begin{tikzcd}
			\on{Desc.}(\calF,U/X) \ar[r, "s^{\prime}"]  & \on{Desc.}(\widetilde{\calF}, U/X) \\
			\calF(X) \ar[u, "\alpha"] \ar[r, "s"'] & \widetilde{\calF}(X)\,. \ar[u, "\widetilde{\alpha}"']
		\end{tikzcd}
	\end{center} 
	We note that for a map $f \colon c \to d$ in an $\infty$-category, the diagonal $\Delta_f \colon c \to c \times_d c$ is an equivalence if and only if $f$ is $(-1)$-truncated if and only if $f$ is a monomorphism. 
	In general, $n$-truncated morphisms have the left cancellation property, and we know that $s$ and $s^{\prime}$ are monomorphisms by assumption and that $\widetilde{\alpha}$ is even an equivalence.
	This implies that $\alpha$ is also $(-1)$-truncated.

\end{proof}

\begin{proposition}
	\label{separated-presheaves-finitelimits}
Suppose that $\calC$ is presentable and compactly generated category. Then the subcategory of $\calP(\calT,\calC)$ spanned by separated presheaves $\calF$ is stable under finite limits.
\end{proposition}
\begin{proof}
	Arguing as in the proof of \Cref{pass-to-anima} we may reduce to the case $\calC=\on{Ani}$. This case follows from the fact that sheafification is exact. 
\end{proof}

\begin{lemma}
	\label{check-on-categories}
	Suppose that $\calF\in \calP(\calT,\CatexE)$. The following statements are equivalent.
	\begin{enumerate}
		\item $\calF$ is separated as an object in $\calP(\calT,\on{Cat}_1)$.
		\item $\calF$ is separated as an object in $\calP(\calT,\CatE)$.
		\item $\calF$ is separated as an object in $\calP(\calT,\CatexE)$.
	\end{enumerate}
\end{lemma}
\begin{proof}
	We can reduce the claim to showing that for a map $[A\to B]\in \CatexE$, the diagonal $\Delta_f\colon A\to A\times_B A$ is an equivalence if and only if $\Delta_{\calF(f)}\colon \calF(A)\to \calF(A)\times_{\calF(B)} \calF(A)$ is an equivalence.
	Here, we denote by $\calF\colon \CatexE\to \on{Cat}_1$ (resp.\ $\calF\colon \CatexE\to \CatE$) the forgetful functor.

	As the forgetful functor commutes with limits, we have $\Delta_{\calF(f)}\simeq \calF(\Delta_f)$. 
	Moreover, the inverse of $\calF(\Delta_f)$ is necessarily given by $\calF(\pi_1)$, where $\pi_1\colon A\times_B A\to A$ is the projection onto the first factor. 
	This shows that $\calF(\Delta_f)^{-1}$ is automatically $E$-linear and exact as we wanted to show.
\end{proof}


The previous considerations allow us to give a complicated, but concrete description of the categories that one obtains from applying sheafification to a separated presheaf with values in $\CatexE$.

\begin{lemma}
	\label{Describing-separated-presheaves}
	Suppose that $\calF\in \calP(\calT,\CatexE)$ is a separated presheaf.
	Let $\widetilde{\calF}\in \calS(\calT,\CatexE)$ be its sheafification.
	For all $S\in \calT$, the category $\widetilde{\calF}(X)\in \CatexE$ can be described as follows:
	Objects $V\in \widetilde{\calF}(S)$ are tuples $(S',V',\alpha)$, where $[S'\to S]\in \calT$ is a v-cover and $V'\in \on{Desc.}(\calF,S'\!/S)$.
	If we have a map $S''\xrightarrow{f} S'\to S$ with $S''\to S$ a v-cover, then $(S',V',\alpha)$ is isomorphic to $(S'',f^*V',f^*\alpha)$. 
	Given two objects, $(S',V'_1,\alpha_1)$ and $(S',V'_2,\alpha_2)$ the set of morphisms between them agrees with the set of morphisms in $\on{Desc.}(\calF,S'\!/S)$.
	A sequence $\Sigma\coloneqq[(S',V'_1,\alpha_1)\to (S',V'_2,\alpha_2)\to (S',V'_3,\alpha_3)]$ is exact if and only if there is
	$S''\xrightarrow{f} S' \to S$ with $S''\to S$ a v-cover such that the sequence $f^*\Sigma\coloneqq[f^*V'_1\to f^*V'_2\to f^*V'_3]$ is exact in $\calF(S^{\prime \prime})$.
\end{lemma}

Let $G$ be a reductive group over $E$ and let $\on{Rep}_G\in \CatexE$ denote the category of finite dimensinal algebraic representations of $G$ with its usual $\otimes$-structure and its usual exact structure.

\begin{lemma}
	\label{funct-RepG}
	Let $\calF\in \calP(\calT,\CatexE)$ be a separated presheaf with sheafification $\widetilde{\calF}$. 
	Then the sheafification of ${\rm Fun}_E^{\otimes,{\rm ex}}(\Rep_G, \calF)$ is ${\rm Fun}_E^{\otimes,{\rm ex}}({\rm Rep}_G, \widetilde{\calF})$.
\end{lemma}

\begin{proof}
	The category $\Rep_G$ is generated under tensor products by finitely many objects, say $\{\calO_i\}^n_{i=1}$. 
	We claim that for any $S \in {\rm Perf}^{\rm aff}$ and $F \in {\rm Fun}^{\otimes,{\rm ex}}({\rm Rep}_G, \widetilde{\calF})(S)$, there exists a v-cover $S' \to S$ such that $F$ factors through a (exact monoidal) functor ${\rm Rep}_G \to {\rm Desc.}(\calF,S'\!/S)$. 
	Indeed, according to \Cref{Describing-separated-presheaves} each of the tensor generators $\calO_i$ maps to an object of the form $(S_i,V_i,\alpha_i)$ and we can take $S'=S_1\times_S\dots \times S_n$. 
	Furthermore, since for every other $S''\to S'\to S$ the map 
	\[{\rm Desc.}(\calF,S'\!/S)\to {\rm Desc.}(\calF,S''\!/S)\]
	is fully-faithful, all morphisms between the tensor powers of the $\calO_i$ must lie in ${\rm Desc.}(\calF,S'\!/S)$. 

	With other words, we get the first equality in
	\begin{align*}
		{\rm Fun}^{\otimes,{\rm ex}}({\rm Rep}_G, \widetilde{\calF})(S) &= \colim_{S'\to S} {\rm Fun}^{\otimes,{\rm ex}}({\rm Rep}_G, {\rm Desc.}(\calF,S'\!/S)) \\
		&= \colim_{S'\to S} {\rm Desc.}({\rm Fun}^{\otimes,{\rm ex}}({\rm Rep}_G, \calF), S'\!/S)\,,
	\end{align*}
	where the second equality holds as ${\rm Desc.}(-,S'\!/S)$ is a limit and ${\rm Fun}$ commutes with limits. The last expression is the value of the sheafification of ${\rm Fun}^{\otimes,{\rm ex}}({\rm Rep}_G, \calF)$ on $S$, so we are done.
\end{proof}

\begin{remark}
	One can be much more general in the formulation of \Cref{funct-RepG}. 
	Indeed, given appropriate cut-off cardinals $\kappa<\lambda$, the categories $\Perf_\lambda$ and $\PSch_\lambda$ have enough $\kappa$-cofiltered limits. 
	Consequently, ${\rm Fun}^{\otimes,{\rm ex}}(\Rep_G,-)$ commutes with sheafification for general $\calF$ as long as $\Rep_G$ is $\kappa$-compact. 
	There is always a $\kappa$ for which this holds.
	To avoid discussing further technicalities, we have settled with the formulation and proof given above. 
\end{remark}

\subsection{The category $\CatexE$}
We verify that $\CatexE$ is presentable and compactly generated in several steps.
We will also define $\CatexE$ in steps, and at each step we will show that the relevant category is presentable and compactly generated.
Throughout, we will use the following statement.
\begin{theorem}[{\cite[Theorem 1.1]{ragimov2022inftycategoricalreflectiontheoremapplications}}]
	\label{reflection-thm}
	Let $\calC$ be a presentable $\infty$-category and let $\calD\subseteq \calC$ be a full subcategory which is closed under limits and $\kappa$-filtered colimits for some regular cardinal $\kappa$. Then, $\calD$ is presentable.
\end{theorem}
We can combine this with the following useful statement. 

\begin{proposition}
	\label{bound-on-compacts}	
	Let $\calC$ be a compactly generated presentable $\infty$-category, let $\calD$ be a presentable $\infty$-category, let $\calF\colon \calD\to \calC$ be a conservative functor which commutes with limits and filtered colimits. 
	Then, $\calD$ is compactly generated.
\end{proposition}
\begin{proof}
	Using the adjoint functor theorem \cite[Corollary 5.5.2.9]{lurie2009htt} and the presentability assumption, we conclude that $\calF\colon \calD\to \calC$ admits a left adjoint $L\colon \calC\to \calD$. 
	Let $\{K_i\}_{i\in I}\subseteq \calC$ be a family of compact generators for $\calC$.
	We claim that $\{L(K_i)\}_{i\in I}$ is a family of compact generators for $\calD$.
	Indeed, they are compact since $\calF$ preserves filtered colimits.
	To show they generate $\calD$, consider the full subcategory $\calD_0\subseteq \calD$ generated under finite colimits by objects of the form $L(K_i)$.
	By \cite[Corollary 5.3.4.15]{lurie2009htt}, $\calD_0$ consists of compact objects.
	We obtain a fully-faithful functor $\iota\colon \!\on{Ind} \calD_0\to \calD$ that commutes with filtered colimits, and hence commutes with all colimits by \cite[Proposition 5.5.1.9]{lurie2009htt}. 
	By the adjoint functor theorem \cite[Corollary 5.5.2.9]{lurie2009htt}, we get a right adjoint $G\colon \calD\to \on{Ind} \calD_0$. 
	To show that $\iota$ is an equivalence, it suffices to show that $G$ is conservative. 
	We observe that $L$ factors along $\iota$ and consequently $\calF$ factors along $G$. 
	Since $\calF$ was assumed to be conservative, $G$ is also conservative.
\end{proof}

\begin{corollary}
	\label{useful-cor}
	Let $\calC$ be a compactly generated presentable $\infty$-category and let $\calD\subseteq \calC$ be a full subcategory which is closed under limits and filtered colimits. Then, $\calD$ is presentable and compactly generated.
\end{corollary}
\begin{proof}
	By \Cref{reflection-thm} above, $\calD$ is presentable.
	The claim now follows from \Cref{bound-on-compacts}. 
\end{proof}

\subsubsection{Additive categories}
\begin{definition}
	An $\infty$-category $\calC$ is \textit{semiadditive} if it is pointed, admits finite products and coproducts, and for every finite collection $\{C_s\in \calC\}_{s\in S}$ the norm map 
	\[\coprod_{s\in S} C_s\to \prod_{t\in S} C_t\]
	described in \cite[Example 6.1.6.11]{lurie2017ha} is an equivalence.
An $\infty$-category $\calC$ is additive if it is semiadditive and its homotopy category is additive in the classical sense.
\end{definition}

Since being semiadditive (or additive) is a property of an $\infty$-category, we could consider the full subcategory of $\on{Cat}_\infty$ spanned by these objects. 
Nevertheless, this $\infty$-category will contain too many functors. 
Indeed, one is only interested in those functors that respect the semiadditive structure.
As it turns out it suffices to look at functors that preserve finite coproducts.

\begin{definition}
We denote by $\mathrm{SemiAdd}$ the $\infty$-category of semiadditive $\infty$-categories (which is defined through e.g. \cite[Corollary 5.4]{Ambidexterity}) and by $\mathrm{Add}$ its full subcategory of additive categories.
\end{definition}

\begin{proposition}
	The $\infty$-categories $\mathrm{Add}$ and $\mathrm{SemiAdd}$ are compactly generated.
\end{proposition}
\begin{proof}
Since $\mathrm{Add}$ is stable under limits and filtered colimits in $\mathrm{SemiAdd}$, we are reduced to showing that the $\infty$-category of semiadditive $\infty$-categories $\mathrm{SemiAdd}$ together with additive functors is compactly generated by \Cref{useful-cor}. 
By \cite[Corollary 5.4]{Ambidexterity}, we are reduced to showing that $\on{Cat}(K_0)$ (the category of categories with finite coproducts and functors preserving them) is compactly generated. 
By \cite[Lemma 4.8.4.2]{lurie2017ha}, $\on{Cat}(K_0)$ is presentable. 
We can apply \Cref{bound-on-compacts} to the forgetful functor $\on{Cat}(K_0)\to \on{Cat}_\infty$ to conclude that $\on{Cat}(K_0)$ is compactly generated.  
Indeed, by \cite[Corollary 5.3.6.10]{lurie2009htt} the forgetful functor admits a left adjoint.
Further, given a filtered diagram $I\to \on{Cat}(K_0)$, we wish to show that $\calC=\varinjlim_{i\in I}\!\!\calC_i \in \on{Cat}_\infty$ stays in $\on{Cat}(K_0)$ and that for every $i\in I$ the projection map $\calC_i\to \calC$ is functor in $\on{Cat}(K_0)$. 
This can be done as in the proof of \cite[Proposition 5.5.7.11]{lurie2009htt}.
More precisely, given a finite collection $\{o_s\}_{s\in S}\subseteq \calC$, we may lift it to a finite collection $\{o_{s,i}\}_{s\in S}$ for some $i\in I$.
The image of the coproduct of the $\{o_{s,i}\}_{s\in S}$ is a coproduct in $\calC$.
\end{proof}

Within the category of additive $\infty$-categories, we have the subcategory $\on{Add}_{\leq 1}\subseteq \on{Add}$ of classical additive categories.
It can be realized as the full subcategory of $1$-truncated objects. 
Since being $1$-truncated is stable under limits and filtered colimits, it follows that $\on{Add}_{\leq 1}$ is compactly generated and that the inclusion $\on{Add}_{\leq 1}\subseteq \on{Add}$ has a left adjoint.

\subsubsection{Symmetric monoidal additive categories.}
As justified in \cite[Proposition 5.6]{Ambidexterity} and \cite[Proposition 4.8.2.7]{lurie2017ha}, there is a symmetric monoidal structure $\on{SemiAdd}^\otimes\to N(\on{Fin}_\ast)$ on the $\infty$-category $\mathrm{SemiAdd}$.
This product captures the following phenomena: 
Given categories $\calC_1,\calC_2,\calD\in \on{SemiAdd}$, then $\on{Fun}(\calC_1\otimes \calC_2, \calD)$ captures functors $G\in \on{Fun}(\calC_1\times \calC_2,\calD)$ that are additive in both variables (i.e.\ $G(c_1\oplus c'_1,c_2)\simeq G(c_1,c_2)\oplus G(c'_1,c_2)$ and $G(c_1,c_2\oplus c_2)\simeq G(c_1,c_2)\oplus G(c_1,c'_2)$).
One can verify that both $\on{Add}_{\leq 1}$ and $\on{Add}$ are preserved under $\otimes$. 
In particular, we obtain symmetric monoidal structures $\on{Add}^\otimes_{\leq 1}$ and $\on{Add}^\otimes$ on $\on{Add}_{\leq 1}$ and $\on{Add}$, respectively.

We know from \cite[Lemma 4.8.4.2]{lurie2017ha} that $\otimes\colon \on{Add}_{\leq 1}\times \on{Add}_{\leq 1}\to \on{Add}_{\leq 1}$ preserves colimits in each variable. 
It follows from \cite[Corollary 3.2.3.5]{lurie2017ha} that the categories of commutative algebra objects $\on{CAlg}(\on{Add}^\otimes)$ and $\on{CAlg}(\on{Add}^\otimes_{\leq 1})$ are again presentable.  
Moreover, by \cite[Lemma B.2.4]{heyer20246}, the forgetful functor $\on{CAlg}(\on{Add}^\otimes_{\leq 1})\to \on{Add}_{\leq 1}$ is conservative and commutes with limits and filtered colimits. 
Using \Cref{bound-on-compacts}, we conclude that $\on{CAlg}(\on{Add}^\otimes_{\leq 1})$ is compactly generated.

Objects in $\on{CAlg}(\on{Add}^\otimes_{\leq 1})$ capture the data of an additive category $\calC$ together with a symmetric monoidal operation $\otimes\colon \calC\times \calC \to \calC$ that is additive on both variables, and a $\otimes$-unit $\mathbb{1}\in \calC$ satisfying the usual further compatibilities. 
Functors in $\on{CAlg}(\on{Add}^\otimes_{\leq 1})$ capture symmetric monoidal additive functors between such categories. 
We will now isolate in $\on{CAlg}(\on{Add}^\otimes_{\leq 1})$ those categories that are rigid.
\begin{definition}
	Given $(\calC,\otimes, \mathbb{1})$ a symmetric monoidal $\infty$-category, we say that an object $o\in \calC$ is \textit{dualizable} if there exists a dual $o^\vee$ together with an evaluation map $\on{ev}_o\colon o\times o^\vee\to \mathbb{1}$, a coevaluation map $\on{coev}_o\colon \mathbb{1}\to o\otimes o^\vee$ and commutative diagrams
	\begin{center}
	\begin{tikzcd}
		o \arrow{r}{{\on{coev}}_o\otimes {\on{id}}}  \arrow{rd}{{\on{id}}}  & o\otimes o^\vee \otimes o \arrow{d}{{\on{id}}\otimes {\on{ev}_o}} & 
		o^\vee \arrow{r}{{\on{id}}\otimes{\on{coev}}_o }  \arrow{rd}{{\on{id}}}  & o^\vee \otimes o \otimes o^\vee \arrow{d}{{\on{ev}_o}\otimes {\on{id}}} \\
										    &  o &
	  &  o^\vee\,.
	\end{tikzcd}
	\end{center}
	We say that $(\calC,\otimes, \mathbb{1})$ is \textit{rigid} if every object of $\calC$ is dualizable.
\end{definition}

One can easily verify that symmetric monoidal functors preserve dualizable objects, in particular we can form the full subcategory of $\Cat\subseteq \on{CAlg}(\on{Add}^\otimes_{\leq 1})$ spanned by those symmetric monoidal categories that are rigid.
It follows from \Cref{useful-cor} that $\Cat$ is presentable and compactly generated. 

Given a ring $R$, for example $R\in \{E, O_E\}$, we let $\on{Proj}(R)\in \Cat$ denote the rigid symmetric monoidal additive category of finite projective $R$-modules.
We let $\on{Cat}^\otimes_{1,R}$ denote the coslice category of $\Cat$ under $\on{Proj}(R)$.
Since the forgetful functor $\on{Cat}^\otimes_{1,R}\to \on{CAlg}(\on{Add}^\otimes_{\leq 1})$ commutes with limits and filtered colimits, applying \Cref{reflection-thm} and \Cref{bound-on-compacts}, we get that $\on{Cat}^\otimes_{1,R}$ is also compactly generated.  
Given a map of rings $R\to R'$ and a category $\calC\in \on{Cat}^\otimes_{1,R}$, we let $\calC\otimes_R R'$ denote the short hand for $\calC\otimes_{\on{Proj}(R)} \on{Proj}(R')$.

\subsubsection{Exact categories.}
In the above, we have justified that $\CatOE$ and $\CatE$ are presentable and compactly generated. 
In what follows, we will justify that $\CatexOE$ and $\CatexE$ are also presentable and compactly generated. 

Recall that Quillen's definition of an exact category can be rephrased in $\infty$-categorical language in terms of Waldhausen and coWaldhausen categories, cf.\ \cite[Example 3.3]{Barwick_Thm_Heart}.
\begin{definition}{\cite[Definition 3.1]{Barwick_Thm_Heart}}
An exact $\infty$-category consists of a triple $(\calC,\calC_{\dagger},\calC^{\dagger})$ such that:
\begin{enumerate}
\item The underlying $\infty$-category $\calC$ is additive.
\item The pair $(\calC,\calC_{\dagger})$ is a Waldhausen $\infty$-category.
\item The pair $(\calC,\calC^{\dagger})$ is a coWaldhausen $\infty$-category.
\item A square in $\calC$ is an ambigressive pullback if and only if it is an ambigressive pushout.
\end{enumerate}
Here we use the following terminology:
\begin{itemize}
	\item $\calC_{\dagger}\subseteq \calC$ is a wide subcategory and a morphism of $\calC_{\dagger}$ is called ingressive or a cofibration.
	\item $\calC^{\dagger}\subseteq \calC$ is a wide subcategory and a morphism of $\calC^{\dagger}$ is called egressive or a fibration.
\item A pullback square
\begin{center}
	\begin{tikzcd}
		X \ar[r] \ar[d] & Y \ar[d, twoheadrightarrow] \\
		X^{\prime} \ar[r, rightarrowtail] & Y^{\prime} 
	\end{tikzcd}
\end{center}
is said to be ambigressive if $X^{\prime} \rightarrowtail Y^{\prime}$ is ingressive and $Y \twoheadrightarrow Y^{\prime}$ is egressive. 
Dually, a pushout square 
\begin{center}
	\begin{tikzcd}
		X \ar[r, rightarrowtail] \ar[d, twoheadrightarrow] & Y \ar[d] \\
		X^{\prime} \ar[r] & Y^{\prime} 
	\end{tikzcd}
\end{center}
is said to be ambigressive if $X \rightarrowtail Y$ is ingressive and $X \twoheadrightarrow X^{\prime}$ is egressive.
\item An exact sequence is fiber/cofiber sequence
	\begin{center}
		\begin{tikzcd}
			X^{\prime} \ar[r, rightarrowtail] \ar[d, twoheadrightarrow] & X \ar[d, twoheadrightarrow] \\
			0 \ar[r, rightarrowtail] & X^{\prime \prime}\,.
		\end{tikzcd}
	\end{center}
\end{itemize}
\end{definition}

\begin{proposition}
	The $\infty$-category $\mathrm{Exact}$ of exact $\infty$-categories is compactly generated, and the forgetful functor $\on{Exact}\to \on{Add}$ commutes with limits and filtered colimits.
\end{proposition}
\begin{proof}
By definition, we have a fully faithful embedding $\mathrm{Exact} \hookrightarrow \mathrm{Wald}_{\infty} \times_{\mathrm{Cat}_{\infty}} \mathrm{coWald}_{\infty}$.
The categories $\mathrm{Wald}_{\infty}$ and $\mathrm{coWald}_{\infty}$ are compactly generated by \cite[Proposition 4.8]{Barwick_Higher_K}. 
Moreover, by \cite[Propositions 4.4, 4.5]{Barwick_Higher_K} and the adjoint functor theorem, it follows that the maps $\mathrm{Wald}_{\infty} \to \mathrm{Cat}_{\infty}$ and $\mathrm{CoWald}_{\infty} \to \mathrm{Cat}_{\infty}$ are maps in $\mathrm{Pr}^R$.
By \cite[Proposition 5.5.7.6]{lurie2009htt}, it follows that $\mathrm{Wald}_{\infty} \times_{\mathrm{Cat}_{\infty}} \mathrm{coWald}_{\infty}$ is also compactly generated.
Hence by \Cref{useful-cor}, it suffices to show the above embedding preserves limits and filtered colimits.

Let $p \colon I \to \mathrm{Exact}$ be a diagram with the compositions $\pi_1 \circ p \colon I \to \mathrm{Wald}_{\infty}$ and $\pi_2 \circ p \colon I \to \mathrm{coWald}_{\infty}$, where $\pi_1 (\mathrm{resp.}\ \pi_2) \colon \mathrm{Exact} \to \mathrm{Wald}_{\infty} (\mathrm{resp.}\ \mathrm{coWald}_{\infty})$ are the forgetful functors.
Let $(\calC, \calC_{\dagger})=\varprojlim \pi_1 \circ p$ and $(\calC, \calC^{\dagger})=\varprojlim \pi_2 \circ p$. 
We want to show that $(\calC, \calC_{\dagger}, \calC^{\dagger})$ is an exact category. 
By \cite[Proposition 4.5]{Barwick_Higher_K}, it follows that $(\calC,\calC_{\dagger})=(\underset{i \in I}{\varprojlim} \ \calC_i, \underset{i \in I}{\varprojlim} \ \calC_{\dagger, i})$ and dually $(\calC,\calC^{\dagger})=(\underset{i \in I}{\varprojlim} \ \calC_i, \underset{i \in I}{\varprojlim} \ \calC^{\dagger}_i)$.
Since additive categories are stable under limits, it suffices to show that a square is an ambigressive pullback if and only if it is an ambigressive pushout.
We know this is true for the categories $(\calC_i,\calC_{\dagger, i}, \calC^{\dagger}_i)$.
We have an isomorphism of morphism $\infty$-categories 
$$\mathrm{Mor}_{\underset{i \in I}{\varprojlim} \ \calC_i}(\underset{i \in I}{\varprojlim} \ x_i,\underset{i \in I}{\varprojlim} \ y_i) \simeq \underset{i \in I}{\varprojlim} \ \mathrm{Mor}_{\calC_i}(x_i,y_i)\,.$$ 
But then, a square is a pullback/pushout square if and only if its projections to the categories $\calC_i$ are pullback/pushout squares. 
To show this, we use \cite[Proposition 4.4.2.6]{lurie2009htt} to reduce to the case of products and fiber products. 
For products, this follows from \cite[Corollary 5.1.2.3]{lurie2009htt} and for fiber products this follows from \cite[Proposition 5.4.5.5]{lurie2009htt}.
Similarly, for a filtered $\infty$-category $I$ with a functor $p^{\prime} \colon I \to \mathrm{Exact}$ with compositions $\pi_1 \circ p^{\prime} \colon I \to \mathrm{Wald}_{\infty}$ and $\pi_2 \circ p^{\prime} \colon I \to \mathrm{coWald}_{\infty}$, we let $(\calC,\calC_{\dagger}) = \varinjlim \pi_1 \circ p^{\prime}$ and $(\calC, \calC^{\dagger})=\varinjlim \pi_2 \circ p^{\prime}$. 
We want to show that $(\calC,\calC_{\dagger},\calC^{\dagger})$ is an exact category.
Again using \cite[Proposition 4.5]{Barwick_Higher_K}, it follows that $(\calC,\calC_{\dagger})=(\underset{i \in I}{\varinjlim} \ \calC_i, \underset{i \in I}{\varinjlim} \ \calC_{\dagger, i})$ and dually $(\calC,\calC^{\dagger})=(\underset{i \in I}{\varinjlim} \ \calC_i, \underset{i \in I}{\varinjlim} \ \calC^{\dagger}_i)$.
Now let $x_{i_0} \in \calC_{i_0}$ and $y_{i_1} \in \calC_{i_1}$ with their images $x,y \in \calC$. 
By \cite{rozenblyum2012filtered}, there is an isomorphism of morphism $\infty$-categories
\begin{equation}
	\label{formula}
	\mathrm{Mor}_{\calC}(x,y) \simeq \underset{j \in I_{(i_0,i_1)/}}{\varinjlim} \mathrm{Mor}_{\calC_j}(x_j,y_j)\,.
\end{equation}
We consider an ambigressive pushout 
\begin{center}
	\begin{tikzcd}
		X \ar[r, rightarrowtail] \ar[d, twoheadrightarrow] & Y \ar[d] \\
		X^{\prime} \ar[r] & Y^{\prime} 
	\end{tikzcd}
\end{center}
in $\calC$. 
We can pick a finite level $i_0 \in I$, and lifts $[x_{i_0}\to x'_{i_0}], [x_{i_0}\to y_{i_0}]\in \calC_{i_0}$ of the maps $[X\to Y], [X\to X']$ such that $[x_{i_0}\to y_{i_0}]$ is ingressive and $x_{i_0}\to x'_{i_0}$ is egressive. 
By \cite[Definition 2.7]{Barwick_Higher_K}, pushouts of ingressive morphisms exist.
We let $y'_{i_0}$ denote the pushout.
Since pushout diagrams of ingressive morphisms remain ingressive under any functor in $\mathrm{Wald}_{\infty}$, we can show that the image of $y'_{i_0}$ under $\calC_{i_0}\to \calC$ is equivalent to $Y'$ by using \Cref{formula}.
Since $\calC_{i_0}$ is exact, we have an ambigressive pushout diagram which is also an ambigressive pullback diagram
\begin{center}
	\begin{tikzcd}
		x_{i_0} \ar[r, rightarrowtail] \ar[d, twoheadrightarrow] & y_{i_0} \ar[d, twoheadrightarrow] \\
		x_{i_0}^{\prime} \ar[r, rightarrowtail] & y_{i_0}^{\prime} \,.
	\end{tikzcd}
\end{center}
Using \Cref{formula} and that functors in $\mathrm{coWald}_\infty$ preserve egressive pullbacks, one can show that the image of the square above in $\calC$ remains an egressive pullback.
This argument shows that if a square in $\calC$ is an ambigressive pushout, then it is an ambigressive pullback.
One can use the dual argument to show the converse.
\end{proof}

We can finally define $\CatexE$ by considering the Cartesian square in $\widehat{\on{Cat}}_\infty$
\begin{center}
\begin{tikzcd}
	\CatexE \arrow{r} \arrow{d}  & \on{Exact} \arrow{d}{\calF_2} \\
	\CatE \arrow{r}{\calF_1} & \on{Add}\,.
\end{tikzcd}
\end{center}

\begin{proposition}
	The category $\CatexE$ is presentable and compactly generated.
\end{proposition}

\begin{proof}
	Both forgetful functors $\calF_1$ and $\calF_2$ are maps in ${\rm Pr}^R$ as in \cite[Definition 5.5.3.1]{lurie2009htt}, since they commute with small limits and filtered colimits.
	By \cite[Theorem 5.5.3.18]{lurie2009htt}, ${\rm Pr}^R$ admits small limits and they can be computed in $\widehat{{\rm Cat}}_{\infty}$.
	Moreover, by \cite[Proposition 5.5.7.6]{lurie2009htt} for $\kappa = \omega$, fiber products in ${\rm Pr}^R$ preserve the property of being compactly generated. 
\end{proof}

	\bibliography{biblio.bib}

\begin{thebibliography}{BZCHN24}

\bibitem[ABM24]{anschuetz20246functorformalismsolidquasicoherent}
Johannes Anschütz, Arthur-César~Le Bras, and Lucas Mann.
\newblock {A 6-functor formalism for solid quasi-coherent sheaves on the
  Fargues-Fontaine curve}.
\newblock {\em preprint}, 2024.
\newblock \href{https://arxiv.org/abs/2412.20968}{arXiv:2412.20968}.

\bibitem[ABV92]{ABV92}
Jeffrey Adams, Dan Barbasch, and David~A. Vogan, Jr.
\newblock {\em The {L}anglands classification and irreducible characters for
  real reductive groups}, volume 104 of {\em Progress in Mathematics}.
\newblock Birkh\"{a}user Boston, Inc., Boston, MA, 1992.

\bibitem[AG15]{GA15}
D.~Arinkin and D.~Gaitsgory.
\newblock Singular support of coherent sheaves and the geometric {L}anglands
  conjecture.
\newblock {\em Selecta Math. (N.S.)}, 21(1):1--199, 2015.

\bibitem[AGLR22]{anschütz2022padictheorylocalmodels}
Johannes Ansch\"utz, Ian Gleason, Jo\~ao Louren{\c{c}}o, and Timo Richarz.
\newblock On the $p$-adic theory of local models.
\newblock {\em preprint}, 2022.
\newblock \href{https://arxiv.org/abs/2201.01234}{arXiv:2201.01234}.

\bibitem[ALB25]{AB21}
Johannes Ansch{\"u}tz and Arthur-C{\'e}sar Le~Bras.
\newblock A {Fourier} transform for {Banach}-{Colmez} spaces.
\newblock {\em J. Eur. Math. Soc. (JEMS)}, 27(9):3651--3712, 2025.

\bibitem[And21]{And21}
Grigory Andreychev.
\newblock Pseudocoherent and perfect complexes and vector bundles on analytic
  adic spaces.
\newblock {\em preprint}, 2021.
\newblock \href{https://arxiv.org/abs/2105.12591}{arXiv:2105.12591}.

\bibitem[Ans19]{Ans19}
Johannes Ansch{\"u}tz.
\newblock Reductive group schemes over the {Fargues}-{Fontaine} curve.
\newblock {\em Math. Ann.}, 374(3-4):1219--1260, 2019.

\bibitem[Ans22]{Ans18}
Johannes Ansch{\"u}tz.
\newblock Extending torsors on the punctured {S}pec
  \({A}_{\operatorname{inf}}\).
\newblock {\em J. reine angew. Math. (Crelle)}, 2022(783):227--268, 2022.

\bibitem[Ans23]{Anschuetz_absFF}
J.~Ansch\"utz.
\newblock {$G$-bundles on the absolute Fargues--Fontaine curve}.
\newblock {\em Acta Arith.}, 207:351--363, 2023.

\bibitem[AS20]{anel20}
Mathieu Anel and Chaitanya~Leena Subramaniam.
\newblock Small object arguments, plus-construction, and left-exact
  localizations.
\newblock {\em preprint}, 2020.
\newblock \href{https://arxiv.org/abs/2004.00731}{arXiv:2004.00731}.

\bibitem[Bar15]{Barwick_Thm_Heart}
Clark Barwick.
\newblock On exact {{\(\infty\)}}-categories and the theorem of the heart.
\newblock {\em Compos. Math.}, 151(11):2160--2186, 2015.

\bibitem[Bar16]{Barwick_Higher_K}
Clark Barwick.
\newblock On the algebraic {{\(K\)}}-theory of higher categories.
\newblock {\em J. Topol.}, 9(1):245--347, 2016.

\bibitem[Bez16]{Bez16}
Roman Bezrukavnikov.
\newblock On two geometric realizations of an affine {H}ecke algebra.
\newblock {\em Publ. Math. Inst. Hautes \'Etudes Sci.}, 123(1):1--67, 2016.

\bibitem[BGR84]{BGR84}
S.~Bosch, U.~G\"{u}ntzer, and R.~Remmert.
\newblock {\em Non-{A}rchimedean analysis}, volume 261 of {\em Grundlehren der
  mathematischen Wissenschaften [Fundamental Principles of Mathematical
  Sciences]}.
\newblock Springer-Verlag, Berlin, 1984.
\newblock A systematic approach to rigid analytic geometry.

\bibitem[BM21]{BhattMathew_21}
Bhargav Bhatt and Akhil Mathew.
\newblock The arc-topology.
\newblock {\em Duke Math. J.}, 170(9):1899--1988, 2021.

\bibitem[BMS18]{BMS18}
Bhargav Bhatt, Matthew Morrow, and Peter Scholze.
\newblock Integral {{\(p\)}}-adic {Hodge} theory.
\newblock {\em Publ. Math., Inst. Hautes {\'E}tud. Sci.}, 128:219--397, 2018.

\bibitem[BS15]{BS15}
Bhargav Bhatt and Peter Scholze.
\newblock The pro-{\'e}tale topology for schemes.
\newblock {\em Ast{\'e}risque}, 369:99--201, 2015.

\bibitem[BS17]{bhatt_scholze_projectivity_of_the_witt_vector_affine_grassmannian}
Bhargav Bhatt and Peter Scholze.
\newblock Projectivity of the {W}itt vector affine {G}rassmannian.
\newblock {\em Invent. Math.}, 209(2):329--423, 2017.

\bibitem[BZCHN24]{benzvi2022coherent}
David Ben-Zvi, Harrison Chen, David Helm, and David Nadler.
\newblock Coherent {Springer} theory and the categorical {Deligne}-{Langlands}
  correspondence.
\newblock {\em Invent. Math.}, 235(2):255--344, 2024.

\bibitem[CI23]{CI23}
Charlotte Chan and Alexander~B. Ivanov.
\newblock On loop {D}eligne-{L}usztig varieties of {C}oxeter-type for inner
  forms of {${\rm GL}_n$}.
\newblock {\em Camb. J. Math.}, 11(2):441--505, 2023.

\bibitem[CS17]{CS17}
Ana Caraiani and Peter Scholze.
\newblock On the generic part of the cohomology of compact unitary {S}himura
  varieties.
\newblock {\em Ann. of Math. (2)}, pages 649--766, 2017.

\bibitem[DHKM25]{DHKM20}
Jean-Fran{\c{c}}ois Dat, David Helm, Robert Kurinczuk, and Gilbert Moss.
\newblock Moduli of {L}anglands parameters.
\newblock {\em J. Eur. Math. Soc.}, 27(5):1827--1927, 2025.

\bibitem[Far18]{Fargues_courbe}
Laurent Fargues.
\newblock La courbe.
\newblock {\em Proc. Int. Cong. Math. - 2018 Rio de Janeiro}, 1, 2018.

\bibitem[Far20]{fargues_g_torseurs_en_theorie_de_hodge_p_adique}
Laurent Fargues.
\newblock {$G$-torseurs en th\'eorie de Hodge $p$-adique}.
\newblock {\em Comp. Math.}, 156(10):2076--2110, 2020.

\bibitem[FS24]{FS24}
Laurent Fargues and Peter Scholze.
\newblock {Geometrization of the local Langlands correspondence}.
\newblock {\em preprint}, 2024.
\newblock \href{https://arxiv.org/abs/2102.13459}{arXiv:2102.13459}.

\bibitem[GL24]{GL22}
Ian Gleason and Jo{\~a}o Louren{\c{c}}o.
\newblock Tubular neighborhoods of local models.
\newblock {\em Duke Math. J.}, 173(4):723--743, 2024.

\bibitem[Gle21]{Gle22v2}
Ian Gleason.
\newblock Specialization maps for {S}cholze's category of diamonds, v2.
\newblock {\em preprint}, 2021.
\newblock \href{https://arxiv.org/abs/2012.05483v2}{arXiv:2012.05483v2}.

\bibitem[Gle22]{Gle21}
Ian Gleason.
\newblock {On the geometric connected components of moduli spaces of p-adic
  shtukas and local Shimura varieties}.
\newblock {\em preprint}, 2022.
\newblock \href{https://arxiv.org/abs/2107.03579}{arXiv:2107.03579}.

\bibitem[Gle24]{Gle22}
Ian Gleason.
\newblock {Specialization maps for Scholze’s category of diamonds}.
\newblock {\em Mathematische Annalen}, pages 1--69, 2024.

\bibitem[Gü24]{Gut23}
Anton Güthge.
\newblock Perfect-prismatic {F}-crystals and p-adic {S}htukas in families.
\newblock {\em preprint}, 2024.
\newblock \href{https://arxiv.org/abs/2307.01108}{arXiv:2307.01108}.

\bibitem[Har20]{Ambidexterity}
Yonatan Harpaz.
\newblock Ambidexterity and the universality of finite spans.
\newblock {\em Proc. Lond. Math. Soc. (3)}, 121(5):1121--1170, 2020.

\bibitem[He16]{He_hecke_padic}
Xuhua He.
\newblock Hecke algebras and $p$-adic groups.
\newblock In {\em Current developments in mathematics 2015}, pages 73--135.
  Int. Press, 2016.

\bibitem[Hel23]{hellmann2021derived}
Eugen Hellmann.
\newblock {On the derived category of the Iwahori-Hecke algebra}.
\newblock {\em Comp. Math.}, 159(5):1042--1110, 2023.

\bibitem[Hen00]{Henniart_LLC_GLn}
Guy Henniart.
\newblock {Une preuve simple des conjectures de Langlands pour $GL(n)$ sur un
  corps $p$-adique}.
\newblock {\em Invent. Math.}, 139(2):439--455, 2000.

\bibitem[HK25]{HamacherKim_22}
Paul Hamacher and Wansu Kim.
\newblock Point counting on igusa varieties for function fields.
\newblock {\em Adv. Math.}, 480:110517, 2025.

\bibitem[HM24]{heyer20246}
Claudius Heyer and Lucas Mann.
\newblock 6-functor formalisms and smooth representations.
\newblock {\em preprint}, 2024.
\newblock \href{https://arxiv.org/abs/2410.13038}{arXiv:2410.13038}.

\bibitem[HS23]{HS21}
David Hansen and Peter Scholze.
\newblock Relative perversity.
\newblock {\em Commun. Am. Math. Soc.}, 3:631--668, 2023.

\bibitem[HT01]{HT01}
Michael Harris and Richard Taylor.
\newblock {\em The geometry and cohomology of some simple {S}himura
  varieties.(AM-151), Volume 151}.
\newblock Princeton University press, 2001.

\bibitem[HV11]{HartlViehmann}
Urs Hartl and Eva Viehmann.
\newblock The {Newton} stratification on deformations of local
  {{\(G\)}}-shtukas.
\newblock {\em J. Reine Angew. Math.}, 656:87--129, 2011.

\bibitem[Iva23]{Ivanov_arc_descent}
Alexander~B. Ivanov.
\newblock {Arc-descent for the perfect loop functor and $p$-adic
  Deligne--Lusztig spaces}.
\newblock {\em J. reine angew. Math. (Crelle)}, 2023(794):1--54, 2023.

\bibitem[Kal16]{Kal16}
Tasho Kaletha.
\newblock The local {L}anglands conjectures for non-quasi-split groups.
\newblock In {\em Families of automorphic forms and the trace formula}, Simons
  Symp., pages 217--257. Springer, [Cham], 2016.

\bibitem[Ked05]{Kedlaya_slope_revisited}
Kiran~S. Kedlaya.
\newblock Slope filtrations revisited.
\newblock {\em Doc. Math.}, 10:447--525, 2005.
\newblock errata, ibid. {\bf 12} (2007), 361-362.

\bibitem[Ked20]{Kedlaya_16_ringAinf}
Kiran~S. Kedlaya.
\newblock {Some ring-theoretic properties of ${A}_{\rm inf}$}.
\newblock In {\em $p$-adic Hodge theory}, Simons Symposia, pages 129--141.
  Springer, 2020.

\bibitem[KL15]{kedlaya_liu_relative_p_adic_hodge_theory_foundations}
Kiran~S. Kedlaya and Ruochuan Liu.
\newblock {\em Relative {{\(p\)}}-adic {Hodge} theory: foundations}, volume 371
  of {\em Ast{\'e}risque}.
\newblock Paris: Soci{\'e}t{\'e} Math{\'e}matique de France (SMF), 2015.

\bibitem[Kot85]{Kottwitz}
Robert~E. Kottwitz.
\newblock Isocrystals with additional structure.
\newblock {\em Compositio Math.}, 56(2):201--220, 1985.

\bibitem[Kot97]{KottwitzII}
Robert~E. Kottwitz.
\newblock Isocrystals with additional structure. {II}.
\newblock {\em Compositio Math.}, 109(3):255--339, 1997.

\bibitem[LRS93]{MR1228127}
G.~Laumon, M.~Rapoport, and U.~Stuhler.
\newblock {${D}$}-elliptic sheaves and the {L}anglands correspondence.
\newblock {\em Invent. Math.}, 113(2):217--338, 1993.

\bibitem[Lur09]{lurie2009htt}
Jacob Lurie.
\newblock {\em Higher Topos Theory}, volume 170 of {\em Annals of Mathematics
  Studies}.
\newblock Princeton University Press, Princeton, NJ, 2009.

\bibitem[Lur17]{lurie2017ha}
Jacob Lurie.
\newblock Higher algebra, 2017.
\newblock Available at \url{https://www.math.ias.edu/~lurie}.

\bibitem[Lus95]{Lus95}
George Lusztig.
\newblock Classification of unipotent representations of simple {$p$}-adic
  groups.
\newblock {\em Internat. Math. Res. Notices}, (11):517--589, 1995.

\bibitem[MW23]{MW23}
Lucas Mann and Annette Werner.
\newblock Local systems on diamonds and {{\(p\)}}-adic vector bundles.
\newblock {\em Int. Math. Res. Not.}, 2023(15):12785--12850, 2023.

\bibitem[PR24]{PR21}
Georgios Pappas and Michael Rapoport.
\newblock {$ p $-adic shtukas and the theory of global and local Shimura
  varieties}.
\newblock {\em Cambridge Journal of Mathematics}, 12(1):1--164, 2024.

\bibitem[Roz12]{rozenblyum2012filtered}
Nick Rozenblyum.
\newblock Filtered colimits of $\infty$-categories.
\newblock 2012.

\bibitem[RR96]{RR_96}
Micheael Rapoport and Melanie Richartz.
\newblock {On the classification and specialization of $F$-isocrystals with
  additional structure}.
\newblock {\em Comp. Math.}, 103(2):153--181, 1996.

\bibitem[RS22]{ragimov2022inftycategoricalreflectiontheoremapplications}
Shaul Ragimov and Tomer~M. Schlank.
\newblock The $\infty$-categorical reflection theorem and applications.
\newblock {\em preprint}, 2022.
\newblock \href{https://arxiv.org/abs/2207.09244}{arXiv:2207.09244}.

\bibitem[Sch22]{Sch17}
Peter Scholze.
\newblock {Étale cohomology of diamonds}.
\newblock {\em preprint}, 2022.
\newblock \href{https://arxiv.org/abs/1709.07343}{arXiv:1709.07343}.

\bibitem[SW20]{SW20}
Peter Scholze and Jared Weinstein.
\newblock {\em Berkeley Lectures on $p$-adic Geometry}, volume 389 of {\em
  AMS-207}.
\newblock Princeton University Press, 2020.

\bibitem[SZ18]{SZ18}
Allan~J. Silberger and Ernst-Wilhelm Zink.
\newblock Langlands classification for {$L$}-parameters.
\newblock {\em J. Algebra}, 511:299--357, 2018.

\bibitem[Vie20]{Viehmann_20}
Eva Viehmann.
\newblock On the geometry of the {N}ewton stratification.
\newblock In Thomas~J. Hanies and Michael Harris, editors, {\em Stabilization
  of the trace formula, Shimura varieties, and arithmetic applications, Volume
  II: Shimura varieties and Galois representations}, pages 192--208. Cambridge
  Univ. Press, 2020.

\bibitem[Vie24]{Viehamann_Newton_BdR}
Eva Viehmann.
\newblock On {Newton} strata in the {{\(B_{\mathrm{dR}}^+\)}}-{Grassmannian}.
\newblock {\em Duke Math. J.}, 173(1):177--225, 2024.

\bibitem[Vog93]{Vog93}
David~A. Vogan, Jr.
\newblock The local {L}anglands conjecture.
\newblock In {\em Representation theory of groups and algebras}, volume 145 of
  {\em Contemp. Math.}, pages 305--379. Amer. Math. Soc., Providence, RI, 1993.

\bibitem[XZ17]{XiaoZhu}
Liang Xiao and Xinwen Zhu.
\newblock {Cycles on Shimura varieties via geometric Satake}.
\newblock {\em preprint}, 2017.
\newblock \href{https://arxiv.org/abs/1707.05700}{arXiv:1707.05700}.

\bibitem[Zha23]{Zha23}
Mingjia Zhang.
\newblock A {PEL}-type {I}gusa stack and the $p$-adic geometry of {S}himura
  varieties.
\newblock {\em preprint}, 2023.
\newblock \href{https://arxiv.org/abs/2309.05152}{arXiv:2309.05152}.

\bibitem[Zhu17]{zhu_affine_grassmannians_and_the_geometric_satake_in_mixed_characteristic}
Xinwen Zhu.
\newblock Affine {G}rassmannians and the geometric {S}atake in mixed
  characteristic.
\newblock {\em Ann. of Math. (2)}, 185(2):403--492, 2017.

\bibitem[Zhu18]{zhu2018geometric}
Xinwen Zhu.
\newblock Geometric {S}atake, categorical traces, and arithmetic of {S}himura
  varieties.
\newblock {\em preprint}, 2018.
\newblock \href{https://arxiv.org/abs/1810.07375}{arXiv:1810.07375}.

\bibitem[Zhu20]{Zhu20}
Xinwen Zhu.
\newblock Coherent sheaves on the stack of {L}anglands parameters.
\newblock {\em preprint}, 2020.
\newblock \href{https://arxiv.org/abs/2008.02998}{arXiv:2008.02998}.

\bibitem[Zhu25]{zhu2025tame}
Xinwen Zhu.
\newblock {Tame categorical local Langlands correspondence}.
\newblock {\em preprint}, 2025.
\newblock \href{https://arxiv.org/abs/2504.07482}{arXiv:2504.07482}.

\end{thebibliography}
	\bibliographystyle{alpha}
\end{document}